\documentclass[10pt]{amsart}

\usepackage{amsmath,amssymb,amsthm,amscd,amsfonts,amstext,amsxtra}
\usepackage{plain}
\usepackage{color}
\usepackage{inputenc}
\usepackage{setspace}
\usepackage{fontenc}
\usepackage{textcomp}
\usepackage{multicol}
\usepackage[all]{xy}

\usepackage[dvips]{graphicx}
\usepackage{psfrag}
\usepackage{epsfig,subfigure}

\newcommand{\SL}{{\mathrm {SL}}}
\newcommand{\SO}{{\mathrm {SO}}}
\newcommand{\GL}{{\mathrm {GL}}}

\newcommand{\tq}{\, :\, }
\newcommand{\id}{\operatorname{id}}
\newcommand{\dist}{\operatorname{dist}}
\renewcommand{\ln}{\log}

\def\e{{\Lambda}}

\def\m{{\mathrm{m}}}
\def\M{{\mathrm{M}}}
\def\wde{{\widehat \Delta_r}}
\def\proj{{\operatorname{proj}}}
\def\Leb{{\mathrm{Leb}}}

\def\Om{\mho}
\def\Ga{\Theta}
\def\Up{\Upsilon}
\def\De{\Xi}
\def\di{\,|\,}

\def\TV{{\mathcal{TV}}}
\def\VF{{\mathcal{VF}}}
\def\TF{{\mathcal{TF}}}

\renewcommand{\AA}{{\mathcal A}}
\newcommand{\BB}{{\mathcal B}}
\newcommand{\CC}{{\mathcal C}}

\newcommand{\FF}{{\mathcal F}}

\newcommand{\HH}{{\mathcal H}}

\newcommand{\MM}{{\mathcal M}}

\newcommand{\QQ}{{\mathcal Q}}
\renewcommand{\SS}{{\mathcal S}}
\newcommand{\TT}{{\mathcal T}}
\newcommand{\UU}{{\mathcal U}}
\newcommand{\VV}{{\mathcal V}}

\newcommand{\MH}{{\mathcal M}{\mathcal H}}
\renewcommand{\TH}{{\mathcal T}{\mathcal H}}
\newcommand{\THI}{{\mathcal T}{\mathcal H}{\mathcal I}}
\newcommand{\MHI}{{\mathcal M}{\mathcal H}{\mathcal I}}
\newcommand{\MHQ}{{\mathcal M}{\mathcal H}{\mathcal Q}}

\newcommand{\bA}{\mathfrak{A}}
\newcommand{\RR}{\mathfrak{R}}

\newcommand{\C}{{\mathbb C}}
\newcommand{\N}{{\mathbb N}}
\renewcommand{\P}{{\mathbb P}}

\newcommand{\R}{{\mathbb R}}
\newcommand{\Z}{{\mathbb Z}}

\newcommand\ssigma{\mathfrak{S}^0}
\newcommand\sssigma{\mathfrak{S}}

\def\BA{{\underline \AA}}
\def\balpha{{\underline \alpha}}
\def\bbeta{{\underline \beta}}
\def\bxi{{\underline \xi}}
\def\bx{{\underline x}}
\def\by{{\underline y}}
\def\bbetaj1{{\underline {\beta_{j_1}}}}
\def\bxij2{{\underline {\xi_{j_2}}}}
\def\balphai{\underline {\alpha_i}}
\def\bbetai{\underline {\beta_i}}

\def\be{\begin{equation}}
\def\ee{\end{equation}}

\newtheorem{thm}{Theorem}[section]
\newtheorem{cor}[thm]{Corollary}
\newtheorem{lem}[thm]{Lemma}
\newtheorem{lemma}[thm]{Lemma}

\newtheorem{prop}[thm]{Proposition}
\theoremstyle{remark}
\newtheorem{rem}[thm]{Remark}

\newtheorem{example}[thm]{Example}

\theoremstyle{definition}
\newtheorem{definition}[thm]{Definition}

\newcommand{\Cb}{C_0}
\newcommand{\eb}{\epsilon_0}
\newcommand{\expansion}{\kappa}
\newcommand{\norm}[1]{\left\| #1 \right\|}
\DeclareMathOperator{\dLeb}{dLeb}
\newcommand{\dd}{\, {\rm d}}
\newcommand{\const}{{\rm const\,}}

\begin{document}

\title[Exponential Mixing for Quadratic Differentials]{Exponential Mixing for the Teichm\"{u}ller flow in the Space of Quadratic Differentials}
\author{Artur Avila and Maria Jo\~ao Resende}

\address{CNRS UMR 7599,
Laboratoire de Probabilit\'es et Mod\`eles al\'eatoires.
Universit\'e Pierre et Marie Curie--Bo\^\i te courrier 188.
75252--Paris Cedex 05, France}
\thanks{This work was partially conducted during the period A.A.
served as a Clay Research Fellow.}
\curraddr{IMPA.
Estrada D. Castorina 110, Jardim Bot\^anico,
22460-320 Rio de Janeiro, Brazil.} 
\urladdr{www.impa.br/$\sim$avila/}
\email{artur@math.sunysb.edu}

\address{
IMPA. Estrada D. Castorina 110, Jardim Bot\^anico.
22460-320 Rio de Janeiro, Brazil.
}
\thanks{M.J.R. was supported by Funda\c{c}\~ao para a Ci\^encia e Tecnologia 
(FCT-Portugal) by the grant SFRH/BD/16135/2004.}
\email{mjoaor@impa.br}

\begin{abstract}
We consider the Teichm\"uller flow on the unit cotangent bundle of the
moduli space of compact Riemann surfaces with punctures.  We show that it is
exponentially mixing for the Ratner class of observables.  More generally,
this result holds for the restriction of the Teichm\"uller flow to an
arbitrary connected component of stratum.  This result generalizes \cite
{AGY} which considered the case of strata of squares.
\end{abstract} 



\maketitle


\section{Introduction}

Let $g,n \geq 0$ be integers such that $3g-3+n>0$ and let
$\TT_{g,n}$ be the Teichm\"uller space of marked Riemann surfaces of genus
$g$ with $n$ punctures.
There is a natural $\SL(2,\R)$ action on the unit cotangent bundle
$\QQ^1_{g,n}$ to $\TT_{g,n}$, which preserves the natural (infinite)
Liouville measure.  The orbits of the diagonal flow project
to the geodesics of the Teichm\"uller metric on $\TT_{g,n}$.

Let $\QQ^*=\QQ^*_{g,n}$ be the quotient of $\QQ^1_{g,n}$ by the modular group
$\operatorname{Mod}(g,n)$.  The $\SL(2,\R)$ action descends to $\QQ^*_{g,n}$.  The
Liouville measure descends
to a {\it finite} measure $\mu=\mu_{g,n}$ on $\QQ^*_{g,n}$.
The diagonal flow $T_t:\QQ^*_{g,n} \to \QQ^*_{g,n}$ is called the
{\it Teichm\"uller geodesic flow}.

Veech showed that $T_t$ is mixing with respect to $\mu$:
if $\phi$ and $\psi$ are observables ($L^2$ functions) with zero
mean then
\be \label {1.1}
\lim_{t \to \infty} \int \phi (\psi \circ T_t) d\mu=\frac {1} {\mu(\QQ^*)}
\int \phi d\mu \int \psi d\mu.
\ee
Here we are interested in the speed of mixing, that is, the rate of
convergence of (\ref {1.1}). As usual, it is
necessary to specify a class of
``regular'' observables. The class for which our results apply
is the {\it Ratner class} $H$ of observables which are H\"older with
respect to the $\SO(2,\R)$ action. More precisely, letting $R_\theta$
denote the rotation of angle $2 \pi \theta$, $H$ is the set of all
$\phi \in L^2(\mu)$ such that $\theta \mapsto R_\theta \phi \in L^2(\mu)$
is a H\"older function (this includes all functions which are H\"older
with respect to the metric of the fiber). This is a natural class to
consider, since exponential
mixing for observables in the Ratner class is known to be
equivalent to the ``spectral
gap'' property for the $\SL(2,\R)$ action, (the hard direction of this
equivalence being due to Ratner, see the Appendix B of \cite {AGY}
for a discussion).

\begin{thm} \label {main}
The Teichm\"uller flow is exponentially mixing with respect to $\mu$ for
observables in the Ratner class.
\end{thm}

In the sequel we will see the Theorem \ref{main} is a special case of a more 
general result on the restriction of Teichm\"uller flow to ``strata'' and 
discuss some of the ingredients in the proof.
First we will discuss in more detail the main notions involved in 
this statement.

\subsection{Quadratic differentials and half-translation surfaces}

A \emph{quadratic differential} $q$ on a Riemann surface $S$ (compact,
with punctures) assigns to each point of the surface a complex quadratic
form on the correspon\-ding tangent space, depending holomorphically 
on the point. Given any local coordinate $z$ on $S$, the quadratic 
differential may be written as $q_z=\phi(z) dz^2$ where the 
coefficient $\phi(z)$ is a holomorphic function; then the
expression $q_w=\phi'(w)dw^2$ with respect to any other local 
coordinate $w$ is determined by
$$
\phi'(w) = \phi(z) \left(\frac{dz}{dw}\right)^2
$$
on the intersection of the domains. The \emph{norm} of a quadratic 
differential is defined by $\|q\|=\int |\phi|\,dz\,d\bar z$ 
(the integral does not depend on the choice of the local 
coordinates). Quadratic differentials with finite norm are called 
\emph{integrable}: in this case the quadratic differential naturally extends
to a {\it meromorphic} quadratic differential on the completion of $S$,
with at worse simple poles at the punctures.  Below we will restrict considerations to
integrable quadratic differentials.

Each quadratic differential $q$ induces a special geometric structure 
on the completion of $S$, as follows. Near any non-singular point (puncture or zero)
one can choose 
\emph{adapted coordinates} $\zeta$ for which the local expression 
of $q$ reduces to $q_\zeta=d\zeta^2$. Given any pair $\zeta_1$ and 
$\zeta_2$ of such adapted coordinates,
\begin{equation}\label{eq.halftranslationstructure}
(d\zeta_1)^2 = (d\zeta_2)^2 \text{\ \ or, equivalently,\ \ } \zeta_1 = \pm \zeta_2 + \const.
\end{equation}
Thus, we say that the set of adapted coordinates is a 
\emph{half-translation atlas} on the complement of the singularities 
and $S$ is a \emph{half-translation surface}. In particular, 
$S\backslash \{\text{singularities}\}$ is endowed with a flat Riemannian 
metric imported from the plane via the half-translation atlas. 
The total area of this metric coincides with the norm $\|q\|$. Adapted 
coordinates $\zeta$ may also be constructed in the neighborhood of 
each singularity $z_i$ such that $$q_\zeta = \zeta^{l_i}d\zeta^2$$ 
with $l_i \geq -1$. 
Through them, the flat metric can be completed with 
a conical singularity of angle equal to $\pi(l_i+2)$ at $z_i$ (thus $l_i=0$ corresponds
to removable singularities).

A quadratic differential $q$ is \emph{orientable} if it is the 
square of some \emph{Abelian differential}, that is, some holomorphic 
complex $1$-form $\omega$. Notice that square roots can always be 
chosen locally, at least far from the singularities, so that 
orientability has mostly to do with having a globally consistent 
choice.  In the orientable case adapted coordinates may be chosen 
so that $\omega_\zeta=d\zeta$. Changes between such coordinates
are given by \begin{equation}\label{eq.translationstructure}
d\zeta_1 = d\zeta_2 \text{\ \ or, equivalently,\ \ } \zeta_1 = \zeta_2 + \const
\end{equation}   
instead of \eqref{eq.halftranslationstructure}. One speaks of 
\emph{translation atlas} and \emph{translation surface} in this case. 
We shall be particularly interested in the case when $q$ is \emph{not} orientable.

\subsection{Strata}\label{strata}

Each element of $\QQ^*_{g,n}$ admits a representation as a meromorphic
quadratic differential $q$
on a compact Riemann surface of genus $g$ with at most $n$ simple poles and
with $\|q\|=1$.
To each $q \in \QQ^*_{g,n}$ we can associate a symbol
$\sigma=(k,\nu,\varepsilon)$ where
\begin{enumerate}
\item $k$ is the number of poles,
\item $\nu=(\nu_j)_{j \geq 1}$ and $\nu_j$ is the number of zeros of order
$j$,
\item $\varepsilon \in \{-1,1\}$ is equal to $1$ if $q$ is the square
of an Abelian differential and to
$-1$ otherwise.
\end{enumerate}
We denote by $\QQ^*_{g}(\sigma)$ the {\it stratum} of all $q$ with symbol
$\sigma$.  A non-empty stratum is an analytic orbifold of real
dimension $4g+2k+2\sum \nu_j+\varepsilon-3$ which is invariant under the
Teichm\"uller flow.  Each
non-empty stratum carries a natural volume form and the corresponding
measure, $\mu_{g}(\sigma)$ has finite mass and is invariant under the
Teichm\"uller flow.

A stratum $\QQ^*_{g,n}(\sigma)$
is not necessarily connected, but it is finitely connected, and the
connected components are obviously $\SL(2,\R)$ invariant 
(see \cite{KZ03}, \cite{La04}, \cite{La08}).
Veech showed that the restriction
of the Teichm\"uller flow restricted to any connected component of
$\QQ_{g,n}^*(\sigma)$ is ergodic with respect to the restriction of
$\mu_{g,n}(\sigma)$.  In \cite {AGY}, it was shown that in the case
of strata of squares (that is, with $\varepsilon=1$) the Teichm\"uller
flow is exponentially mixing (for a class of H\"older observables) with
respect to $\mu$.  Their approach is followed here and generalized to yield:

\begin{thm} \label {1.2}\label{t.main}
The Teichm\"uller flow is exponentially mixing with respect to
each ergodic component of $\mu_{g}(\sigma)$ for observables in the
Ratner class.
\end{thm}

There is a single stratum $\QQ^*_{g,n}(\sigma)$
with maximum dimension, which is open and connected
and has full $\mu_{g,n}$ measure: for this stratum,
$\mu_{g,n}$ coincides with $\mu_{g,n}(\sigma)$.
Thus Theorem \ref {main} is a particular case of this one.

\subsection{Outline of the proof}

Our approach to exponential mixing follows \cite {AGY} which develops
around a combinatorial description of the moduli space of Abelian differentials.

The combinatorial description which we will use in the treatment of quadratic
differentials, essen\-tially the one of \cite {BL}, builds from the observation
that the space of (non-orientable) quadratic differentials can be 
viewed as a subset of the space of Abelian differentials with involution. 
Indeed, it is well known that given any quadratic differential 
$q$ on a Riemann surface $S$ of genus $g$ there exists a double 
covering $\pi:\tilde S \to S$, branched over the singularities 
of odd order, and there is an Abelian differential $\omega$ on 
the surface $\tilde S$ such that  $\pi_*(\omega^2)=q$. In other words, 
$q$ lifts to an orientable quadratic differential on $\tilde S$. 
In this construction,
\begin{itemize}
\item to each zero of $q$ with even multiplicity $l_i\ge 1$ corresponds a 
pair of zeros of $\omega$ with multiplicity $m_j=l_i/2$;
\item to each zero of $q$ with odd multiplicity $l_i\ge 1$ corresponds a 
zero of $\omega$ with multiplicity  $m_j=l_i+1$;
\item to each pole of $q$ with $l_i=-1$ corresponds a removable 
(that is, order $0$) singularity of $\omega$.
\end{itemize} 
The surface $\tilde S$ is connected if and only if $q$ is non-orientable. 
Notice $i_*(\omega)=\pm\omega$, where $i: \tilde S \to \tilde S$ is the 
involution permuting the points in each fiber of the double cover $\pi$.

An Abelian differential induces a
translation structure on the surface.  In particular we can speak of the 
horizontal flow to the ``east'' and the vertical flow to the ``north'' 
(the involution exchanges north with south and east with west).

Thus, we consider moduli spaces of Abelian differentials with involution 
and a certain combinatorial marking.  The combinatorial marking 
includes the order of the zeros at the singularities, but also 
a distinguished singularity with a fixed eastbound separatrix. 
This moduli space $\MM$ is a finite cover of
$\QQ_{g,n}(\sigma)$ where $\SL(2,\R)$ is still acting, and thus it is
enough to prove the result on this space.

We parametrize the moduli space as a moduli space of zippered 
rectangles with involution as follows.
Choosing a convenient segment $I$ inside the separatrix, we look 
at the first return map under the northbound flow
to the union of the $I$ and its image under the involution. This map is an
{\it interval exchange transformation with involution}, and the original
northbound flow becomes a suspension flow, living in the union of some
rectangles.  The original surface can be obtained from the rectangles by
gluing appropriately.
This construction can be carried out in a large open set of $\MM$ 
(with complement of codimension $2$, see \cite{Ve86}).

Once this combinatorial model is setup, one can view the Teichm\"uller 
flow on $\MM$ as a suspension flow over a (weakly) hyperbolic transformation, 
which is itself a skew-product over a (weakly) expanding transformation, 
the Rauzy algorithm with involution.

We then consider some 
appropriate compact subset of the domain of the Rauzy algorithm 
with involution: the induced transformation is automatically 
expanding and the Teichm\"uller flow is thus modelled on an 
``excellent hyperbolic flow'' in the language of \cite {AGY}. 
Two properties need to be verified to deduce exponential mixing: 
the return time should not be cohomologous to locally cons\-tant, 
and it should have exponential tails. The first property is 
an essentially algebraic consequence of the zippered rectangle construction. 
The second depends essentially on proving some distortion estimate. 
Both proofs of the distortion estimate in \cite {AGY} depend 
heavily on certain properties of the usual Rauzy induction 
(simple description of transition probabilities for a random walk) 
which seem difficult to generalize to our setting. 
We provide here an alternative proof which is less dependent 
on precise estimates for the random walk.

\begin{rem}
Since the moduli space of zippered
rectangles with involution can be regarded
as a (Teichm\"uller flow-invariant)
subspace of the larger moduli space of all zippered rectangles, it
would seem natural
to carry out the analysis around an appropriate restriction of the usual
Rauzy algorithm.  However, while the Rauzy algorithm can be modelled
as a random walk
on a finite graph, this property does not persist after restriction
(though there is still a natural random walk model, it takes place in
an infinite graph).
\end{rem}

{\bf Acknowledgements:}  We thank Erwan Lanneau for suggesting the use of
the double cover construction to define an appropriate renormalization
dynamics. We also thank Carlos Matheus and Marcelo Viana for many useful remarks.

\section{Excellent hyperbolic semi-flows}
\label{par:define_exp}

In \cite {AGY}, an abstract result for exponential mixing was proved for the
class of so-called \emph{excellent hyperbolic semi-flows}, following the work
of Baladi-Vall\'ee \cite {BV} based on the
foundational work of Dolgopyat \cite {Do}.  This result can be directly used
in our work.  In this section we state precisely this result, which will
need several definitions.

By definition, a Finsler manifold is a smooth manifold endowed with
a norm on each tangent space, which varies continuously with the base point.

\begin{definition}
A \emph{John domain} $\Delta$ is a finite dimensional connected
Finsler manifold, together with a measure $\Leb$ on $\Delta$, with
the following properties:
\begin{enumerate}
\item For $x,x' \in \Delta$, let  $d(x,x')$ be the infimum
of the length of a $C^1$ path contained in $\Delta$ and
joining $x$ and $x'$. For this distance, $\Delta$ is bounded and
there exist constants $\Cb$ and
$\eb$ such that, for all $\epsilon<\eb$, for all $x\in \Delta$,
there exists $x'\in \Delta$ such that $d(x,x') \leq \Cb \epsilon$
and such that the ball $B(x',\epsilon)$ is compactly contained in
$\Delta$.
\item The measure $\Leb$ is a fully supported finite
measure on $\Delta$, satisfying the following inequality: for all
$C>0$, there exists $A>0$ such that, whenever a ball $B(x,r)$ is
compactly contained in $\Delta$, $\Leb(B(x,Cr))\leq A \Leb(B(x,r))$.
\end{enumerate}
\end{definition}

\begin{definition} \label {markovmap}
Let $L$ be a finite or countable set, 
let $\Delta$ be a John domain, and let $\{\Delta^{(l)}\}_{l\in L}$ be a
partition into open sets of a full measure subset of $\Delta$. A map $T:
\bigcup_{l} \Delta^{(l)} \to \Delta$ is a \emph{uniformly expanding
Markov map} if
\begin{enumerate}
\item For each $l$, $T$ is a $C^1$ diffeomorphism between
$\Delta^{(l)}$ and $\Delta$, and there exist constants
$\expansion>1$ (independent of $l$) and $C_{(l)}$ such that, for all
$x\in \Delta^{(l)}$ and all $v\in T_x \Delta$, $\expansion \norm{v}
\leq \norm{ DT(x)\cdot v} \leq C_{(l)} \norm{v}$.
\item Let $J(x)$ be the inverse of the Jacobian of $T$ with respect to $\Leb$.
Denote by $\HH$ the set of inverse branches of $T$. The function
$\log J$ is $C^1$ on each set $\Delta^{(l)}$ and 
there exists $C>0$ such that, for all $h\in \HH$, $\norm{ D((\log
J)\circ h)}_{C^0(\Delta)}\leq C$.
\end{enumerate}
\end{definition}
Such a map $T$ preserves a unique absolutely continuous measure
$\mu$. Its density is bounded from above and from below and is
$C^1$.

\begin{definition}
\label{def:goodroof}
Let $ T:\bigcup_{l} \Delta^{(l)} \to \Delta$ be
a uniformly expanding Markov map on a John domain. A function $r:
\bigcup_{l} \Delta^{(l)} \to \R_+$ is a \emph{good roof function} if
\begin{enumerate}
\item There exists $\epsilon_1>0$ such that $r \geq \epsilon_1$.
\item There exists $C>0$ such that, for all $h\in \HH$, $\norm{
D(r\circ h)}_{C^0}\leq C$.
\item It is not possible to write $r=\psi + \phi\circ T-\phi$ 
on $\bigcup_{l} \Delta^{(l)}$,
where $\psi:\Delta \to \R$ is constant on each set $\Delta^{(l)}$
and $\phi:\Delta\to \R$ is $C^1$.
\end{enumerate}
\end{definition}
If $r$ is a good roof function for $T$, we will write
$r^{(n)}(x)=\sum_{k=0}^{n-1}r(T^k x)$.

\begin{definition}
A good roof function $r$ as above has
\emph{exponential tails} if there exists $\sigma_0>0$ such that
$\int_{\Delta} e^{\sigma_0 r} \dLeb <\infty$.
\end{definition}

If $\widehat \Delta$ is a Finsler manifold, we will denote by
$C^1(\widehat{\Delta})$ the set of functions $u:\widehat{\Delta}\to
\R$ which are bounded, continuously differentiable, and such that
$\sup_{x\in \widehat{\Delta}} \norm{Du(x)}< \infty$. Let
  \begin{equation}
  \label{define_C1}
  \norm{u}_{C^1(\widehat{\Delta})} = \sup_{x\in \widehat{\Delta}} |u(x)| + \sup_{x\in
  \widehat{\Delta}} \norm{Du(x)}
  \end{equation}
be the corresponding norm.

\begin{definition} \label {skew-product map}
Let $T:\bigcup_{l} \Delta^{(l)} \to \Delta$ be a uniformly expanding
Markov map, preserving an absolutely continuous
measure $\mu$. A \emph{hyperbolic
skew-product} over $T$ is a map $\widehat{T}$ from a dense open subset
of a bounded connected
Finsler
manifold $\widehat{\Delta}$, to $\widehat{\Delta}$, satisfying the following
properties:
\begin{enumerate}
\item There exists a continuous
map $\pi : \widehat{\Delta} \to \Delta$
such that $T\circ \pi = \pi \circ \widehat{T}$ whenever both members
of this equality are defined.
\item There exists a probability measure $\nu$ on $\widehat{\Delta}$,
giving full mass to the domain of definition of $\widehat{T}$, which is
invariant under $\widehat{T}$.
\item There exists a family of probability measures $\{\nu_x\}_{x\in
\Delta}$ on $\widehat{\Delta}$ which is a disintegration of $\nu$ over
$\mu$ in the
following sense: $x\mapsto \nu_x$ is measurable, $\nu_x$ is supported
on $\pi^{-1}(x)$ and, for every measurable set $A \subset
\widehat{\Delta}$, $\nu(A)=\int \nu_x(A) \dd\mu(x)$.

Moreover, this disintegration satisfies the following property: 
there exists a constant $C>0$ such that,
for any open subset $O\subset \bigcup \Delta^{(l)}$, for any
$u\in C^1(\pi^{-1}(O))$, the
function $\bar u : O \to \R$ given by $\bar u(x)=\int u(y)
\dd\nu_x(y)$ belongs to $C^1(O)$ and
satisfies the inequality
  \begin{equation*}
  \sup_{x\in O} \norm{D\bar u(x)}
  \leq C \sup_{y\in \pi^{-1}(O)} \norm{Du(y)}.
  \end{equation*}
\item
There exists $\expansion>1$ such that, for all $y_1,y_2 \in
\widehat{\Delta}$ with $\pi(y_1)=\pi(y_2)$, holds
  \begin{equation*}
  d(\widehat{T} y_1, \widehat{T} y_2) \leq \expansion^{-1} d(y_1,y_2).
  \end{equation*}
\end{enumerate}
\end{definition}

Let $\widehat{T}$ be an hyperbolic skew-product over a uniformly
expanding Markov map $T$. Let $r$ be a good roof function for $T$,
with exponential tails. It is then possible to define a space
$\widehat{\Delta}_r$ and a semi-flow $\widehat{T}_t$ over
$\widehat{T}$ on $\widehat{\Delta}$, using the roof function $r\circ
\pi$, in the following way.
 Let $\widehat{\Delta}_r =
\{(y,s) \tq y\in \bigcup_l \widehat{\Delta}_l, 0\leq s<r(\pi y )\}$.
For almost
all $y\in \widehat{\Delta}$, all $0\leq s<r(\pi y)$ and all $t\geq 0$, there
exists a unique $n\in \N$ such that $r^{(n)}(\pi y)\leq
t+s<r^{(n+1)}(\pi y)$. Set $\widehat{T}_t (y,s)=(\widehat{T}^n y,
s+t -r^{(n)}(\pi y))$. This is a semi-flow defined almost everywhere
on $\widehat{\Delta}_r$. It preserves the probability measure
$\nu_r=\nu \otimes \Leb/ (\nu\otimes \Leb)(\widehat{\Delta}_r)$.
Using the canonical Finsler metric on $\widehat{\Delta}_r$, namely
the product metric given by 
$\|(u,v)\|:=\|u\|+\|v\|$, we define the space $C^1(\widehat{\Delta}_r)$ as
in \eqref{define_C1}.  Notice that $\widehat \Delta_r$ is not connected, and
the distance between points in different connected components is infinite.

\begin{definition}
A semi-flow $\widehat{T}_t$ as above is called an \emph{excellent
hyperbolic semi-flow}.
\end{definition}

\begin{thm}[\cite {AGY}]
\label{main_thm_hyperbolic} Let $\widehat{T}_t$ be an excellent
hyperbolic semi-flow on a space $\widehat{\Delta}_r$, preserving the
probability measure $\nu_r$. There exist constants $C>0$ and
$\delta>0$ such that, for all functions $U,V\in
C^1(\widehat{\Delta}_r)$, for all $t\geq 0$,
  \begin{equation*}
  \left| \int U\cdot V \circ \widehat{T}_t \dd\nu_r -\left( \int U
  \dd\nu_r\right) \left( \int V \dd\nu_r \right) \right| \leq C
  \norm{U}_{C^1} \norm{V}_{C^1}e^{-\delta t}.
  \end{equation*}
\end{thm}

\section{The Veech flow with involution}
\label{combinatorial}

\subsection{Rauzy classes and interval exchange transformations with involution}
\label{Rauzyclasses}

\subsubsection{Interval exchange transformations with involution}
Let $\AA$ be an alphabet on $2d \geq 4$ letters with an involution 
$i:\AA \to \AA$ and let $* \notin \AA$. When
considering objects modulo involution, we will use underline: for instance
the involution class of an element $\alpha \in \AA$ will be denoted by
$\balpha \in \BA=\AA/i$.
An \emph{interval exchange transformation with involution} of type
$(\AA,i,*)$ depends on the specification of the following data:

\begin{description}
\item [Combinatorial data:]
Let $\pi:\AA \cup \{*\} \to \{1,\ldots,2d+1\}$ be a bijection such that
neither $i(\AA_l) \subset \AA_r$ nor $i(\AA_r) \subset \AA_l$, where
$\AA_l=\{\alpha \in \AA,\, \pi(\alpha)<\pi(*)\}$ and $\AA_r=\{\alpha \in
\AA_r,\, \pi(\alpha)>\pi(*)\}$. The combinatorial data can be viewed as a
row where the elements of $\AA \cup \{*\}$ are displayed in the order
$(\pi^{-1}(1),\ldots,\pi^{-1}(2d+1))$.
\item [Length data:] Let $\lambda \in \R_+^\BA$ be a vector satisfying
\be \label {leftright}
\sum_{\pi(\alpha)<\pi(*)}
\lambda_\balpha=\sum_{\pi(\alpha)>\pi(*)} \lambda_\balpha
\ee
(it is easy to find such a vector $\lambda$).
\end{description}

Let $\sssigma=\sssigma(\AA,i,*)$ be the set of all bijections $\pi$ as above.

The transformation is then defined as follows:
\begin{enumerate}
\item Let $I \subset \R$ be the interval (all intervals will be assumed to be
closed at the left and open at the right) centered on $0$ and of length
$|I| \equiv \sum_{\alpha \in \BA} \lambda_\balpha$ (notice that 
$|I| = 2 \sum_{\balpha \in \BA} \lambda_\balpha$).
\item Let $\overline \pi:\AA \cup \{*\} \to \{1,\ldots,2d+1\}$ be defined by
$\overline \pi(*)=2d+2-\pi(*)$ and $\overline
\pi(\alpha)=2d+2-\pi(i(\alpha))$.
\item Break $I$ into $2d$ subintervals $I_\alpha$ of length $\lambda_\balpha$,
ordered according to $\pi$.
\item Rearrange the subintervals inside $I$ in the order given by $\overline
\pi$.
\end{enumerate}



\subsubsection{Rauzy classes with involution}

We define two operations, the \emph{left} and the \emph{right} on $\sssigma$ 
as follows. Let $\alpha$ and $\beta$ be the leftmost and the rightmost letters 
of the row representing $\pi$, respectively.
If $\beta \neq i(\alpha)$ and
taking $\beta$ and putting it into the position immediately after
$i(\alpha)$ results in a row representing an element $\pi'$ of $\sssigma$,
we say that the left operation is defined at $\pi$, and it takes
$\pi$ to $\pi'$. In this case, we say that $\alpha$ wins and $\beta$ loses. 
Similarly, if $\alpha \neq i(\beta)$ and taking $\alpha$ and 
putting it into the position immediately 
before $i(\beta)$ results in a row representing
an element $\pi'$ of $\sssigma$,
 we say that the right operation is defined at $\pi$, and it takes 
$\pi$ to $\pi'$. In this case, we say that $\beta$ wins and $\alpha$ loses.

\begin{rem}
 Notice that $\beta \neq i(\alpha)$ and $\alpha \neq i(\beta)$ are 
equivalent conditions since the involution $i$ is a bijection. But 
to define left (respectively right) operation we also ask that the 
row obtained after moving $\beta$ (respectively $\alpha$) represents 
an element of $\sssigma$. So we can have either just one of the 
operations defined at some permutation or both.
\end{rem}



Consider an oriented diagram with vertices which are the elements of $\sssigma$ 
and oriented arrows representing the operations left and right starting and 
ending at two vertices of $\sssigma$. We will say that such an arrow has
type left or right, respectively.  We will call this diagram by
\emph{Rauzy diagram with involution}. 
A \emph{path} $\gamma$ of length $m \geq 0$ is a sequence of $m$ arrows,
$a_1, \ldots, a_m$,
 joining $m+1$ vertices, $v_0, \ldots, v_m$, respectively. In this case 
we say that $\gamma$ starts at $v_0$, it ends at $v_m$ and pass through 
$v_1, \ldots, v_{m-1}$. Let $\gamma_1$ and $\gamma_2$ be two paths such 
that the end of $\gamma_1$ is the start of $\gamma_2$. We define their 
concatenation denoted by $\gamma_1 \gamma_2$, which also is a path. 
A path of length zero is identified with a vertex and if it has length 
one we identify it with an arrow. 

A \emph{Rauzy class with involution} $\RR$ is a minimal non-empty subset 
of $\sssigma$, which is invariant under the left and the right operations, and
such that any involution class admits a representative which is the winner
of some arrow starting (and ending) in $\RR$.  Elements of Rauzy classes
with involution are said to be \emph{irreducible}. We denote by 
$\ssigma=\ssigma(\AA) \subset \sssigma$ the set of irreducible permutations 
and let $\Pi(\RR)$ be the set of all paths.

\begin{lemma}
If $\pi$ is irreducible then the left operation
(respectively the right operation) is defined at $\pi$ if and only if
there exists $\lambda \in \R^\BA_+$ satisfying (\ref {leftright}) such that
$\lambda_\balpha>\lambda_\bbeta$ (respectively
$\lambda_\bbeta>\lambda_\balpha$) where $\alpha$ and $\beta$ are the
leftmost and the rightmost elements of $\pi$.
\end{lemma}

\begin{proof}
Assume that the left operation is defined at $\pi$ and let $\pi'$ be the
image of $\pi$.  Let $\lambda' \in \R^\BA_+$ be a vector satisfying
$\sum_{\pi'(\xi)<\pi'(*)}
\lambda'_\bxi=\sum_{\pi'(\xi)>\pi'(*)} \lambda'_\bxi$.  Let $\lambda \in
\R^\BA_+$ be given by $\lambda_\balpha=\lambda'_\balpha+\lambda'_\bbeta$ and
$\lambda_\bxi=\lambda'_\bxi$, $\bxi \neq \balpha$.  Then $\lambda$ satisfies
(\ref {leftright}) and we have $\lambda_\balpha>\lambda_\bbeta$.

Assume that $\lambda_\balpha>\lambda_\bbeta$.  Let
$\lambda'_\balpha=\lambda_\balpha-\lambda_\bbeta$.  Let $\pi'(x)=\pi(x)$
for $\pi(x) \leq \pi(i(\alpha))$, $\pi'(\beta)=\pi(i(\alpha))+1$ and
$\pi'(x)=\pi(x)+1$ for $\pi(i(\alpha))<\pi(x)<2d+1$.  We need to show that
$\pi' \in \sssigma$.

Let $\AA_l=\{\pi(\xi)<\pi(*)\}$,
$\AA_r=\{\pi(\xi)>\pi(*)\}$, $\AA_l'=\{\pi'(\xi)<\pi'(*)\}$,
$\AA'_r=\{\pi'(\xi)>\pi'(*)\}$. 
Notice that $\sum_{\pi'(\xi)<\pi'(*)}
\lambda'_\bxi=\sum_{\pi'(\xi)>\pi'(*)} \lambda_\bxi'$, so
$i(\AA_l')$ can not be properly contained or properly
contain $\AA_r'$.  If $i(\AA_l')=\AA_r'$, then
$\pi'(i(\alpha))>\pi'(*)$, so $\pi(i(\alpha))>\pi(*)$ as well.  This implies
that $\AA_l'=\AA_l$ and $\AA_r'=\AA_r$, and since $\pi \in \sssigma$ we
have $i(\AA_l) \neq \AA_r$.
\end {proof}

\subsubsection{Linear action}

Given a Rauzy class $\RR$, we associate to 
each path $\gamma \in \Pi(\RR)$ a linear map 
$B_\gamma \in \SL(\BA,\Z)$. If $\gamma$ is a vertex we take 
$B_\gamma=\id$. If $\gamma$ is an arrow with winner $\alpha$ 
and loser $\beta$ then we define 
$B_\gamma \cdot e_{\bxi}=e_{\bxi}$ for 
$\bxi \in \BA \backslash \{\balpha\}$, 
$B_\gamma \cdot e_{\balpha}=e_{\balpha}+e_{\bbeta}$, 
where $\{e_{\bxi}\}_{\bxi \in \BA}$ is the canonical basis of 
$\R^{\BA}$. If $\gamma$ is a path, of the form $\gamma=\gamma_1 \ldots \gamma_n$, 
where $\gamma_i$ are arrows for all $i=1, \ldots, n$,
we take $B_{\gamma}=B_{\gamma_1 \ldots \gamma_n}=B_{\gamma_n} \ldots B_{\gamma_1}$. 

\subsection{Rauzy algorithm with involution}

Given  a Rauzy class $\RR \subset \sssigma$, consider the set
\begin{equation*}
\mathcal{S}_\pi=\left\{\lambda \in \R^{\BA}: \sum_{\pi(\alpha)<\pi(*)}\lambda_\balpha=\sum_{\pi(\alpha)>\pi(*)} \lambda_\balpha
\right\}
\end{equation*}
We define
\begin{equation*}
\mathcal{S}_\pi^+=\mathcal{S}_\pi \cap \R^{\BA}_+, \quad 
\Delta_\pi=\mathcal{S}_\pi^+ \times \{\pi\}, \quad
\Delta^0_\RR=\bigcup_{\pi \in \RR}\Delta_\pi.
\end{equation*}

Let $(\lambda,\pi)$ be an element of $\Delta^0_\RR$. 
We say that we can apply Rauzy algorithm with involution to 
$(\lambda,\pi)$ if $\lambda_\balpha \neq \lambda_\bbeta$, where 
$\alpha, \beta \in \AA$ are the leftmost and the rightmost elements 
of $\pi$, respectively.  Then we define $(\lambda',\pi')$ as follows:

\begin{enumerate}
\item Let $\gamma=\gamma(\lambda,\pi)$ be an arrow representing the left 
or the right operation at $\pi$, according to whether
$\lambda_\balpha>\lambda_\bbeta$ or $\lambda_\bbeta>\lambda_\balpha$.
\item Let $\lambda'_{\bxi}=\lambda_{\bxi}$ if $\bxi$ is not the class of the
winner of 
$\gamma$, and $\lambda'_{\bxi}=|\lambda_\balpha-\lambda_\bbeta|$ if $\bxi$ 
is the class of the winner of $\gamma$, i.e., $\lambda=B^*_\gamma \cdot \lambda'$
(here and in the following we will use the notation $A^*$ 
to the transpose of a matrix $A$).
\item Let $\pi'$ be the end of $\gamma$.
\end{enumerate}
We say that $(\lambda',\pi')$ is obtained from $(\lambda,\pi)$ 
by applying Rauzy algorithm with involution, 
of type left or right depending on whether the 
operation is left or right. We have $(\lambda',\pi') \in \Delta^0_\RR$. 
In this way we define a map $Q:(\lambda,\pi) \mapsto (\lambda',\pi')$ 
which is called \emph{Rauzy induction map with involution}. Its domain 
of definition is the set of all $(\lambda,\pi) \in \Delta^0_\RR$ such that 
$\lambda_\balpha \neq \lambda_\bbeta$ (where $\alpha$ and $\beta$ are the 
leftmost and the rightmost letters of $\pi$) and we denote it by 
$\Delta^1_\RR$. The connected components $\Delta_\pi \subset \Delta^0_\RR$
are naturally labeled by elements of $\RR$ and the connected components 
$\Delta_\gamma$ of $\Delta^1_\RR$ are naturally labeled by arrows, 
i.e., paths in $\Pi(\RR)$ of length $1$. 

We associate to $(\lambda,\pi)$ and to $(\lambda',\pi')$ two interval 
exchange transformations with involution $f:I \to I$ and $f':I' \to I'$, 
respectively. The relation between $(\lambda,\pi)$ and $(\lambda',\pi')$ 
implies a relation between the interval exchange transformations with 
involution, namely, the map $f'$ is the first return map of $f$ to a 
subinterval of $I$, obtained by cutting two subintervals 
from the beginning and from the end of $I$ with the same length 
$\lambda_\bxi$, where $\xi$ is the loser of $\gamma$.

Let $\Delta^n_\RR$ be the domain of $Q^n$, $n \geq 2$. The connected 
components of $\Delta^n_\RR$ are naturally labeled by paths in $\Pi(\RR)$ 
of length $n$: if $\gamma$ is obtained by following a sequence of arrows 
$\gamma_1,\ldots,\gamma_n$, then 
$\Delta_\gamma=\{x \in \Delta^0_\RR \tq Q^{k-1}(x) \in \Delta_{\gamma_k},\, 1 \leq k \leq n\}$. 
Notice that if $\gamma$ starts at $\pi$  and ends at $\pi'$ then 
$\Delta_\gamma=(B^*_\gamma \cdot \mathcal{S}_{\pi'}^+) \times \{\pi\}$.

If $\gamma$ is a path in $\Pi(\RR)$ of length $n$ ending at $\pi \in \RR$, let
\begin{equation*}
Q^\gamma=Q^n:\Delta_\gamma \to \Delta_\pi.
\end{equation*}
Let $\Delta^\infty_\RR=\bigcap_{n \geq 0} \Delta^n_\RR$.

\begin{definition}
A path $\gamma$ is said to be \emph{complete} if all involution classes 
$\balpha \in \BA$ are winners of some arrow composing $\gamma$.

The concatenation of $k$ complete paths is said to be \emph{$k$-complete}.

A path $\gamma \in \Pi(\RR)$ is \emph{positive} if $B_\gamma$ is given, 
in the canonical basis of $\R^{\BA}_+$, by a matrix with all entries positive.
\end{definition}

\begin{lemma} \label{complete}
A $(2d-3)$-complete path $\gamma \in \Pi(\RR)$ is positive.
\end{lemma}

\begin{proof}
Let $\gamma=\gamma^1\gamma^2\ldots\gamma^N$ where $\gamma^i$ is an 
arrow starting at $\pi^{i-1}$ and ending at $\pi^{i}$. 
Since $\gamma$ is $l$-complete we also can represent it as 
$\gamma=\gamma_{(1)}\gamma_{(2)}\ldots\gamma_{(l)}$ where each 
$\gamma_{(i)}$ is a complete path passing through $\pi^{(i-1)}_1, 
\ldots, \pi^{(i-1)}_{n(i)}, \pi^{(i)}_1$.

Let $B^*_{\gamma_{(i)}}$ be the matrix such that $\lambda_1^{(i)}= 
B^*_{\gamma_{(i)}} \cdot \lambda_1^{(i+1)}$. And let 
$B^*(\balpha, \bbeta, i)$ be the coefficient on row 
$\balpha$ and on column $\bbeta$ of the matrix $B^*_{\gamma_{(i)}}$.
Fix $k<l$. We denote 
$C^*(k)=B^*_{\gamma_{(1)}} \ldots B^*_{\gamma_{(k)}}$. Let 
$C^*(\balpha, \bbeta, k)$ be the coefficient on row 
$\balpha$ and on column $\bbeta$ of the matrix $C^*(k)$.
We want to prove that for all $\balpha, \bbeta \in \BA$ 
we have $C^*(\balpha, \bbeta, l)>0$.
For $r \geq 0$ denote $\widehat C(r)=B^*_{\gamma^1} \ldots B^*_{\gamma^r}$.

Since the diagonal elements of the matrices $B_{\tilde \gamma}^*$, where 
$\tilde \gamma$ is an arrow, are one and all other terms are non-negative 
integers, we obtain that the sequence $C^*(\balpha, \bbeta, k)$ is non-decreasing 
in $k$, thus:
\be\label{nondecreasing}
C^*(\balpha, \bbeta, k)>0 \Rightarrow C^*(\balpha, \bbeta, k+1)>0
\ee

Fix any $\balpha, \bbeta \in \BA$. We will reorder the involution classes of $\BA$, 
as $\balpha=\balpha_1, \balpha_2, \ldots, \balpha_d=\bbeta$ with associate 
numbers $0=r_1 < r_2 < \ldots< r_d$ such that
\be\label{positive}
C^*(\balpha_1, \balpha_j,r) >0 \; \text \quad \forall r \geq r_j
\ee

If $\balpha=\bbeta$ we take $s=1$ and $r_1=0$ and therefore 
we have (\ref{positive}).
Otherwise we choose the smallest positive integer $r_2$ such that the winner 
of $\gamma^{r_2}$ is $\balpha_1=\balpha$ and let $\balpha_2$ be the loser of the same 
arrow. Notice that $\balpha_1 \neq \balpha_2$ by irreducibility, and 
$B^*(\balpha_1, \balpha_2, r_2)=1$, hence $C^*(\balpha_1, \balpha_2, r)>0$ for 
every $r \geq r_2$. This gives the result for $d=2$.

Now we will see the general case. Assume that 
$\balpha_1, \ldots, \balpha_j$ and $r_1, \ldots, r_j$ 
have been constructed with $\bbeta \neq \balpha_m$ for $1 \leq m \leq j$. 
Let $r'_j$ be the smallest integer greater than $r_j$ such that 
the winner of $\gamma^{r'_j}$ does not belong to 
$\{ \balpha_1, \ldots, \balpha_j \}$ and let $r_{j+1}$ be the smallest 
integer greater than $r'_j$ such that the winner of 
$\gamma^{r_{j+1}}$ is in $\{ \balpha_1, \ldots, \balpha_j \}$. 
Let $\balpha_{j+1}$ be the loser of $\gamma^{r_{j+1}}$. Then 
$\balpha_{j+1}$ is the winner of $\gamma^{r_{j+1}-1}$ and therefore 
$\balpha_{j+1} \notin \{ \balpha_1, \ldots, \balpha_j \}$.
Thus, for some $1 \leq m \leq j$ we have 
$B^*(\balpha_m, \balpha_{j+1}, r_{j+1})=1$ and 
$C^*(\balpha_1, \balpha_m, r_{j+1}-1)>0$, since 
$r_{j+1}>r_m$. Thus 
\begin{equation*}
C^*(\balpha_1, \balpha_{j+1}, r)>0 \quad \forall r \geq r_{j+1}.
\end{equation*}

Following this process, we will obtain $\balpha_s=\bbeta$ 
Now, we will see how many complete paths we need until define $r_d$.

We need a complete path to define each $r_j$ and another one to 
define each $r'_j$, for $3\leq j \leq d$. And we need another complete 
path to define $r_2$. Thus we need at most $2(d-2)+1=2d-3$ complete paths 
composing $\gamma$ to conclude it is positive.
\end{proof}

\subsection{Zippered rectangles} \label {recall}

Let $\pi$ be a permutation in a Rauzy class $\RR \subset \ssigma$. 
Let $\Ga_\pi \subset \mathcal{S}_\pi$ be the set of all $\tau$ such that

\be\label{taudomain}
\begin{array}{c}
\displaystyle \sum_{\pi(*) < \pi(\xi) \leq k_r} 
\tau_\bxi>0 \quad\text{for all}\quad \pi(*) < k_r <2d+1 \\
\displaystyle \sum_{k_l \leq \pi(\xi) < \pi(*)} 
\tau_\bxi<0 \quad\text{for all }\quad 1 < k_l <\pi(*) 
\end{array}
\ee


Observe that $\Ga_\pi$ is an open convex polyhedral cone and we 
will see later it is non-empty.

Given a letter $\alpha \in \AA$, we define 
$M_\pi(\alpha)=\max \{\pi(\alpha), \pi(i(\alpha))\}$ and 
$m_\pi(\alpha)=\min \{\pi(\alpha), \pi(i(\alpha))\}$.

Define the linear operator $\Omega(\pi)$ on $\R^{\BA}$ as follows:

\be\label{defOmega}
\Omega(\pi)_{\bx,\by}=
\left \{ \begin{array}{rl}
2 & \text{if} \quad   M_\pi(x)<m_\pi(y)
,\\[5pt]
-2 & \text{if} \quad  M_\pi(y)<m_\pi(x)
,\\[5pt]
1 & \text{if} \quad   m_\pi(x)<m_\pi(y)<M_\pi(x)<M_\pi(y)
,\\[5pt]
-1 & \text{if} \quad  m_\pi(y)<m_\pi(x)<M_\pi(y)<M_\pi(x)
,\\[5pt]
0 & \text {otherwise}.
\end{array}
\right .
\ee

Observe that $\Omega(\pi)$ is well-defined.

We define the vector $w \in \R^{\BA}$ by $w=\Omega(\pi) \cdot \lambda$ 
and the vector $h \in \R^{\BA}$ by $h=-\Omega(\pi) \cdot \tau$. 
For each $\balpha \in \BA$ define 
$\zeta_\balpha = \lambda_\balpha + i \tau_\balpha$.

\begin{lemma}\label{translation}
If $\gamma$ is an arrow between $(\lambda,\pi)$ and $(\lambda',\pi')$, 
then $w'=B_\gamma \cdot w$. 
\end{lemma}

\begin{proof}
We will consider the case when $\gamma$ is a left arrow. The other case is 
entirely analogous. Let $\balpha(l)$ and $\balpha(r)$ be the leftmost and 
the rightmost letters in $\pi$, respectively. Thus $\balpha(l)$ is the 
winner and $\balpha(r)$ is the loser.

By definition,
\begin{equation*}
 w'_\balpha= \sum_{\pi'(\xi)>\pi'(i(\alpha))}\lambda'_\xi-
\sum_{\pi'(\xi)<\pi'(\alpha)}\lambda'_\xi
\end{equation*}
Since $\lambda'_\balpha=\lambda_\balpha$ for all $\balpha \neq \balpha(l)$ 
and $\lambda'_{\balpha(l)}=\lambda_{\balpha(l)}-\lambda_{\balpha(r)}$, it is 
easy to verify that:
\begin{equation*}
 w'_\balpha = \sum_{\pi(\xi)>\pi(i(\alpha))}\lambda_\xi-
\sum_{\pi(\xi)<\pi(\alpha)}\lambda_\xi=w_\balpha 
\quad \text{if} \quad \balpha \neq \balpha(r).
\end{equation*}
And if $\balpha = \balpha(r)$, 
we have:
\be \nonumber
\begin{array}{cc}
\begin{aligned}
\displaystyle
 & w'_{\balpha(r)} = \sum_{\pi'(\xi)>\pi'(i(\alpha(r)))}\lambda'_\xi-
\sum_{\pi'(\xi)<\pi'(\alpha(r))}\lambda'_\xi
\\
\displaystyle
 & = \sum_{\pi'(\xi)>\pi'(\alpha(r))}\lambda'_\xi-
\sum_{\pi'(\xi)<\pi'(i(\alpha(r)))}\lambda'_\xi
= \sum_{\pi(\xi)>\pi(i(\alpha(l)))}\lambda_\xi-
\sum_{\pi(\xi)<\pi(i(\alpha(r)))}\lambda_\xi
\\
\displaystyle
 & = \sum_{\pi(\xi)>\pi(i(\alpha(l)))}\lambda_\xi-
\sum_{\pi(\xi)<\pi(\alpha(l))}\lambda_\xi+
\sum_{\pi(\xi)>\pi(\alpha(r))}\lambda_\xi-
\sum_{\pi(\xi)<\pi(i(\alpha(r)))}\lambda_\xi
= w_{\balpha(l)}+w_{\balpha(r)}
\end{aligned}
\end{array}
\ee

Therefore, $w'=B_\gamma \cdot w$.
\end{proof}

Let $H(\pi)=\Omega(\pi)\cdot \SS_\pi$. According to the previous lemma, given 
a path $\gamma \in \Pi(\RR)$ starting at $\pi$ and ending at $\pi'$,
we have $B_\gamma \cdot H(\pi)=H(\pi')$.

\begin{lemma}\label{hpositive}
If $\pi \in \RR$ and $\tau \in \Ga_\pi$ then $h \in \R_+^{\BA}$.
\end{lemma}

\begin{proof}
Let $\alpha \in \AA$. We have
\begin{equation*}
h_\balpha=\sum_{\pi(\xi)<\pi(\alpha)}\tau_\xi-
\sum_{\pi(\xi)>\pi(i(\alpha))}\tau_\xi.
\end{equation*}

Suppose $\pi(\alpha), \pi(i(\alpha)) < \pi(*)$:

\begin{multline}\nonumber
\displaystyle h_\balpha=\sum_{\pi(\xi)<\pi(\alpha)}\tau_\xi-
\sum_{\pi(\xi)>\pi(i(\alpha))}\tau_\xi= 
\sum_{\pi(\xi)<\pi(\alpha)}\tau_\xi-
\sum_{\pi(i(\alpha))<\pi(\xi)<\pi(*)}\tau_\xi-
\sum_{\pi(*)<\pi(\xi)\leq 2d+1}\tau_\xi=\\
=-\sum_{\pi(\alpha)\leq \pi(\xi) < \pi(*)}\tau_\xi-
\sum_{\pi(i(\alpha))<\pi(\xi)<\pi(*)}\tau_\xi > 0.
\end{multline}

Analogously, if $\pi(\alpha), \pi(i(\alpha)) > \pi(*)$ we have 
$h_\balpha > 0$.

Now we will suppose that $\pi(\alpha) < \pi(*) < \pi(i(\alpha))$. In 
this case, we have:
\begin{multline*}
\displaystyle h_\balpha=\sum_{\pi(\xi)<\pi(\alpha)}\tau_\xi-
\sum_{\pi(\xi)>\pi(i(\alpha))}\tau_\xi=\\
\displaystyle 
=\sum_{1 \leq \pi(\xi)<\pi(*)}\tau_\xi-
\sum_{\pi(\alpha) \leq \pi(\xi)<\pi(*)}\tau_\xi-
\sum_{\pi(*)< \pi(\xi) \leq 2d+1}\tau_\xi+
\sum_{\pi(*)< \pi(\xi) \leq \pi(i(\alpha))}\tau_\xi=\\
\displaystyle
=-\sum_{\pi(\alpha) < \pi(\xi)<\pi(*)}\tau_\xi+
\sum_{\pi(*)< \pi(\xi) < \pi(i(\alpha))}\tau_\xi > 0
\end{multline*}

So, $h \in \R_+^{\BA}$.
\end{proof}

\begin{lemma}\label{Ga_invariant}
If $\gamma \in \Pi(\RR)$ is an arrow starting at $\pi$ and ending
at $\pi'$ then $(B_\gamma^*)^{-1} \cdot \Ga_{\pi} \subset \Ga_{\pi'}$.
\end{lemma}

\begin{proof}
We will suppose that $\gamma$ is a right arrow and 
the other case is entirely analogous.
Let $\tau \in \Ga_\pi$ and let $\balpha \in \BA$ be the winner and 
$\bbeta \in \BA$ be the loser of $\gamma$.

Notice we have $\tau'_\bxi=\tau_\bxi$ for all $\bxi \in \BA \backslash \{\balpha\}$.
Let $m=\pi(i(\alpha))$. Notice that
\be \label{h}
h_\balpha=h_{i(\alpha)}=\sum_{\pi(\xi)<\pi(i(\alpha))}\tau_\xi-
\sum_{\pi(\xi)>\pi(\alpha)}\tau_\xi=
\sum_{\pi(\xi)<m}\tau_\xi
\ee

Suppose $m < \pi(*)$. Since $\pi'(\xi)=\pi(\xi)$ for all $\xi \in \AA$ 
such that $\pi(\xi)\geq m$ we have that the first inequalities 
of (\ref{taudomain}) are satisfied and
\be \nonumber
\sum_{k_l \leq \pi'(\xi) < \pi'(*)} \tau'_\bxi = 
\sum_{k_l \leq \pi(\xi) < \pi(*)} \tau_\bxi 
< 0 \quad\text{for all }\quad m < k_l <\pi(*)
\ee
Thus, it remains to prove the last inequalities to 
$1< k_l \leq m$. 
Let $1< k_l < m$. Since $\tau'_\balpha=\tau_\balpha-\tau_\bbeta$,
\be \nonumber
\sum_{k_l \leq \pi'(\xi) < \pi'(*)} \tau'_\bxi = 
\sum_{k_l \leq \pi(\xi) < \pi(*)} \tau_\bxi
< 0 \quad\text{for all}\quad 1 < k_l <m.
\ee
If $k_l=m$, by (\ref{h})
\be\label{winnerposition}
\sum_{m \leq \pi'(\xi) < \pi'(*)} \tau'_\bxi = 
\sum_{m \leq \pi(\xi) < \pi(*)} \tau_\bxi - \tau_\bbeta=
\sum_{2 \leq \pi(\xi) < \pi(*)} \tau_\bxi - h_\balpha
< 0
\ee

Now suppose $m > \pi(*)$
This case is analogous to the first one. We will just do 
the part corres\-ponding to (\ref{winnerposition}).
\be \nonumber
\sum_{\pi'(*) < \pi'(\xi) \leq m-1} \tau'_\bxi = 
\sum_{\pi(*) < \pi(\xi) \leq m-1 } \tau_\bxi - \tau_\bbeta=
h_\balpha-\sum_{2 \leq \pi(\xi) < \pi(*)} \tau_\bxi
> 0
\ee

Thus $\tau' \in \Ga_{\pi'}$, as we wanted to prove.
\end{proof}

\begin{definition}
Let $\Ga'_\pi \subset \mathcal{S}_\pi$ be the set of all $\tau \neq 0$
such that
\be\label{tauclosure}
\begin{array}{c}
\displaystyle \sum_{\pi(*) < \pi(\xi) \leq k_r}
\tau_\bxi \geq 0 \quad\text{for all}\quad \pi(*) < k_r <2d+1 \\
\displaystyle \sum_{k_l \leq \pi(\xi) < \pi(*)}
\tau_\bxi \leq 0 \quad\text{for all }\quad 1 < k_l <\pi(*)
\end{array}
\ee
\end{definition}

Let $\gamma \in \Pi(\RR)$ be a path starting at $\pi_s$ and ending
at $\pi_e$.  In the same way we showed that
$(B_\gamma^*)^{-1}
\cdot \Ga_{\pi_s} \subset \Ga_{\pi_e}$ in the previous lemma, one sees
that $(B_\gamma^*)^{-1}
\cdot \Ga'_{\pi_s} \subset \Ga'_{\pi_e}$.

\begin{definition}
Let $\pi \in \Pi(\RR)$ and $\alpha \in \AA$. We say that $\alpha$ is
a \emph{simple letter} if $\pi(\alpha) < \pi(*) < \pi(i(\alpha))$ or if
$\pi(i(\alpha)) < \pi(*) < \pi(\alpha)$. We say that $\alpha$
is a \emph{double letter} if $\pi(\alpha)$ and $\pi(i(\alpha))$ are either
both smaller or either both greater than $\pi(*)$. If 
$\pi(\alpha), \pi(i(\alpha)) < \pi(*)$ we say that $\alpha$
is a \emph{left double letter} or has \emph{left type}, otherwise we 
say that $\alpha$ is a \emph{right double letter} or has \emph{right type}. 
\end{definition}

\begin{lemma}\label{Theta'}
If $\pi$ is irreducible then $\Ga'_\pi$ is non-empty.
\end{lemma}

\begin{proof}
By invariance and irreducibility, it is enough to find some $\pi \in \RR$ 
such that $\Ga'_\pi$ is non-empty.

Given $\alpha, \beta \in \AA$ suppose we have $\pi \in \RR$ with one of 
the two following forms: 
\begin{equation}\label{pi1}
\begin{array}{cccccccccccccc} \cdot & \cdot &
\alpha & \cdot & \cdot & i(\alpha) & \cdot & \beta & \cdot & i(\beta) & \cdot & * & \cdot & \cdot
\end{array}
\end{equation}
\centerline {or}
\begin{equation}\label{pi2}
\begin{array}{cccccccccccccc} \cdot & \cdot &
\alpha & \cdot & \cdot & \beta & \cdot & i(\alpha) & \cdot & i(\beta) & \cdot & * & \cdot & \cdot
\end{array}
\end{equation}

We can define $\tau \in \Ga_\pi'$ by choosing $\tau_\balpha=-\tau_\bbeta=1$ 
and $\tau_\bxi=0$ for all $\bxi \in \BA \backslash \{\balpha, \bbeta\}$.

Let us show that there exists some $\pi \in \RR$ satisfying this property.

By definition of permutation in $\sssigma$, there exist at least one double letter of 
each one of the types, i.e., there exist $\alpha, \beta \in \AA$ such that $\alpha$ is 
left double letter and $\beta$ is right double letter.

If there exists more than one double letter of both types, 
we can obtain another irreducible permutation $\pi'$ which has 
at most one double letter of each one of the types, as follows.
First we apply left or right operations until we obtain one double letter 
in the leftmost or the rightmost position, which is possible by 
irreducibility. We will assume, without loss of generality, that such a letter 
is at rightmost position. If there is at most one left double letter, 
we take the permutation obtained to be $\pi'$. 
But, if there are more than one left double letter, we apply right operations, 
until we find a permutation with just 
one left double letter. Those right operations are 
well-defined since we have more than one double letter of both types.

Suppose that $\alpha \in \AA$ is the unique left double letter.
Then, if it is necessary, we apply right operations until obtain
$\pi(\alpha)=1$. 

Let $\beta \in \AA$ such that $\pi(\beta)=2d+1$, i.e., 
\begin{equation*}
\begin{array}{cccccccccc} 
\alpha & \cdot & \cdot & \cdot & i(\alpha) & \cdot & * & \cdot & \cdot & \beta
\end{array}
\end{equation*}
If $\beta$ is simple applying the left operation we obtain a permutation of type
(\ref{pi1}) or (\ref{pi2}) depending on $\pi(i(\beta))>\pi(i(\alpha))$ or 
$\pi(i(\beta))<\pi(i(\alpha))$, respectively.
If $\beta$ is double we apply the left operation until we obtain a simple letter 
in the rightmost position of $\pi$ and we are in the same 
conditions as in the previous case.
\end{proof}

\begin{definition}
Let us say that a path $\gamma \in \Pi(\RR)$, starting at $\pi_s$ 
and ending at $\pi_e$, is \emph{strongly positive} if it is 
positive and $(B_\gamma^*)^{-1}
\cdot \Ga'_{\pi_s} \subset \Ga_{\pi_e}$.
\end{definition}

\begin{lemma} \label {strongly positive}
Let $\gamma$ be a $(4d-6)$-complete path.
Then $\gamma$ is strongly positive.
\end{lemma}

\begin{proof}
Let $d=\#\BA$. Fix $\tau \in \Ga'_{\pi_s} \backslash \{0\}$. Write $\gamma$ 
as a concatenation of arrows $\gamma=\gamma_1 \ldots \gamma_n$,
and let $\pi^{i-1}$ and $\pi^i$ denote the start and the end of 
$\gamma_i$. Let $\tau^0=\tau$,
$\tau^i=(B^*_{\gamma_i})^{-1} \cdot \tau^{i-1}$.  We must show that $\tau^n
\in \Ga_{\pi^n}$.

Let $h^i=-\Omega(\pi^i) \cdot \tau^i$.  Notice that $\tau \in 
\Ga'_{\pi^0} \backslash \{0\}$ implies that $h^0 \in \overline \R^\BA_+
\backslash \{0\}$.  Indeed, since $\tau \in \Ga'_{\pi^0}$,
for every $\bxi \in \BA$, we have 
\begin{equation*}
\sum_{\pi^0(*)<\pi^0(\alpha)<\pi^0(\xi)} \tau_\balpha \geq 0 \quad \text{and} 
\quad \sum_{\pi^0(\xi)<\pi^0(\alpha)<\pi^0(*)} \tau_\balpha \leq 0.  
\end{equation*}

Moreover, since $\tau \neq 0$, there exist $1 \leq k^l \leq \pi^0(*)$ 
maximal and $\pi^0(*) \leq k^r \leq 2d+1$ minimal such that 
$\tau_{(\pi^0)^{-1}(k^l)} \neq 0$ and 
$\tau_{(\pi^0)^{-1}(k^r)} \neq 0$.
Since $\pi^0$ is irreducible, $k^r-k^l<2d-1$. 
Remember that $h^0_\bxi \geq 0$ for all $\bxi$ and the inequality 
is strict if $\pi^0(\xi)=k^r+1$ and $k^r<2d+1$ or if 
$\pi^0(\xi)=k^l-1$ and $1<k^l<\pi^0(*)$.
Since $h^i=B_{\gamma_i} \cdot h^{i-1}$ we can consider a positive path 
$\gamma_1 \ldots \gamma_i$ and then $h^i \in \R^\BA_+$.


Let $\pi^i(*) \leq k^r_i \leq 2d$ be maximal and $2 \leq k^l_i \leq \pi^i(*)$ 
be minimal such that
\be \nonumber
\sum_{\pi^i(*)<\pi^i(\xi) \leq k} \tau^i_\bxi>0 \quad\text{for all }
\pi^i(*) < k \leq k^r_i,
\ee
\be \nonumber
\sum_{k \leq \pi^i(\xi)<\pi^i(*)} \tau^i_\bxi<0 \quad\text{for all } k^l_i 
\leq k < \pi^i(*).
\ee

We claim that
\begin{enumerate}
\item If $h^{i-1} \in \R^\BA_+$ then $k^r_i -\pi^i(*) \geq k^r_{i-1}-\pi^{i-1}(*)$ 
and $\pi^i(*)-k^l_i \geq \pi^{i-1}(*) - k^l_{i-1}$, in particular 
$k^r_i - k^l_i \geq k^r_{i-1}-k^l_{i-1}$;
\item If $h^{i-1} \in \R^\BA_+$ and the winner of $\gamma_i$ is one of the
first $k^r_{i-1}+1$ letters after $*$ in $\pi^{i-1}$ then $k^r_{i} -k^l_i \geq 
\min \{k^r_{i-1}-k^l_{i-1}+1, 2d-k^l_{i-1}\}$;
\item If $h^{i-1} \in \R^\BA_+$ and the winner of $\gamma_i$ is one of the
last $k^l_{i-1}-1$ letters before $*$ in $\pi^{i-1}$ then $k^r_{i} -k^l_i \geq 
\min \{k^r_{i-1}-k^l_{i-1}+1, k^r_{i-1}-2\}$.
\end{enumerate}

Notice that $2<\pi^i(*) < 2d$ for all $i$.
Let us see that (1), (2) and (3) imply the result, which is equivalent to
the statement that $k^r_n - k^l_n \geq 2d-2$.
Let us write $\gamma=\gamma_{(1)}\ldots\gamma_{(4d-7)}$ where 
$\gamma_{(j)}$ is complete and each $\gamma_{(j)}=\gamma_{s_j}\ldots\gamma_{e_j}$.
By Lemma \ref {complete}, $h^k \in \R^\BA_+$ for $k \geq
e_{2d-3}$. From the definition of a complete path, 
for each $j>2d-3$, there exists $e_{j}< i_1 \leq e_{j+1}$ such that the winner of 
$\gamma_{i_1}$ is one of the letters in position $m_1$ at $\pi^{i_1-1}$ such that
$\pi^{i_1-1}(*)<m_1<k^r_{e_j}+1$. It follows that 
$k^r_{e_{j+1}} - k^l_{e_{j+1}} \geq \min \{k^r_{i_1-1}-k^l_{i_1-1}+1, 2d- k^l_{i_1-1}\}$, 
so 
\be \label{kr}
k^r_{e_{j+1}} - k^l_{e_{j+1}} \geq \min \{k^r_{e_j}-k^l_{e_j}+1, 2d- k^l_{e_j}\}.
\ee
In the same way there exists $e_{j-1}< i_2 \leq e_{j}$ such that the winner of 
$\gamma_{i_2}$ is one of the letters in position $m_2$ at $\pi^{i_2-1}$ such that
$k^r_{e_j}+1 < m_2 < \pi^{i_2-1}(*)$. It follows that 
$k^r_{e_j} - k^l_{e_j} \geq \min \{k^r_{i_2-1}-k^l_{i_2-1}+1, k^r_{i_2-1}-2\}$, 
thus
\be \label{kl}
k^r_{e_j} - k^l_{e_j} \geq \min \{k^r_{e_{j-1}}-k^l_{e_{j-1}}+1, k^r_{e_{j-1}}-2\}.
\ee
By (\ref{kr}) and (\ref{kl}), we see that: 
\be \nonumber
k^r_{e_{j+1}} - k^l_{e_{j+1}} \geq \min \{k^r_{e_{j-1}}-k^l_{e_{j-1}}+2, 
2d-(k^l_{e_{j-1}}-1), (k^r_{e_{j-1}}+1)-2, 2d-2\}.
\ee
Therefore, we obtain
$k^r_n-k^l_n=k^r_{e_{2d-3+2d-4}}-k^l_{e_{2d-3+2d-4}} \geq 
\min \{k^r_{e_{2d-3}}-k^l_{e_{2d-3}}+2d-2, 
2d-(k^l_{e_{2d-3}}-2d-4), (k^r_{e_{j-1}}+2d-4)-2, 2d-2\}
=2d-2$.

We now check (1), (2) and (3).  Assume that $h^{i-1} \in  \R^\BA_+$, and
that $\gamma_i$ is a right arrow, the other case being analogous.
Let $\alpha$ be the rightmost letter of $\pi^{i-1}$ which is the winner 
of $\gamma_i$, and let $\beta$ be the leftmost letter of $\pi^{i-1}$ which 
is the loser of $\gamma_i$.

{\bf Case 1:}  Suppose $\pi^{i-1}(i(\alpha))<\pi^{i-1}(*)$.

If the winner of $\gamma_i$ is not one of the $k^l_{i-1}-1$ last letters in 
the left side of $*$ in $\pi^{i-1}$, then for every 
$\bxi \in \BA$ such that $k^l_{i-1} \leq \pi^{i-1}(\xi) \leq 2d+1$,
we have $\pi^{i-1}(\xi)=\pi^i(\xi)$ and
$\tau^{i-1}_\bxi=\tau^i_\bxi$ for all $k^l_{i-1} \leq \pi^{i-1}(\xi) \leq 2d$. 
Hence $k_i^r -\pi^i(*) \geq k_{i-1}^r-\pi^{i-1}(*)$ and 
$\pi^i(*)-k^l_i \geq \pi^{i-1}(*)-k^l_{i-1}$. 

If the winner $\balpha$ of $\gamma_i$ appears in the $k$-th 
position counting from $*$ to the left in $\pi^{i-1}$ with 
$k^l_{i-1}-1 \leq k <\pi(*)$, then

\begin{equation*}
\displaystyle
\sum_{j \leq \pi^i(\xi) < \pi^i(*)} \tau^i_\bxi=
\sum_{j \leq \pi^{i-1}(\xi) <\pi^{i-1}(*)} \tau^{i-1}_\bxi<0 \quad\text{for all}
\; k+1 \leq j < \pi(*),\\
\end{equation*}
\be \nonumber
\displaystyle
\sum_{j \leq \pi^i(\xi) < \pi^i(*)} \tau^i_\bxi=
\sum_{j+1 \leq \pi^{i-1}(\xi) <\pi^{i-1}(*)} \tau^{i-1}_\bxi<0 \quad\text{for all}
\; k^l_{i-1}-1 \leq j \leq k-1,\\
\ee
\be \nonumber
\displaystyle
\sum_{k \leq \pi^i(\xi) <\pi^i(*)} \tau^i_\bxi=
\sum_{2 \leq \pi^{i-1}(\xi) <\pi^{i-1}(*)} \tau^{i-1}_\bxi-h^{i-1}_\balpha 
\leq -h^{i-1}_\balpha<0,
\ee
which implies that $\pi^i(*)-k^l_i \geq \min \{\pi^{i-1}(*)-2,\pi^{i-1}(*)+1-k^l_{i-1}\}$, 
hence $k^r_i-k^l_i \geq \min \{k^r_{i-1}-k^l_{i-1}+1,k^r_{i-1}-2\}$.

This shows that (1) holds and (3) holds. 
Moreover, (2) also holds since 
its hypothesis can only be satisfied if $k_{i-1}^r=2d$.

{\bf Case 2:} Suppose $\pi^{i-1}(i(\alpha))>\pi^{i-1}(*)$.





If the winner of $\gamma_i$ is not one of the $k^r_{i-1}+1$ first letters in 
the right side of $*$ in $\pi^{i-1}$, then for every 
$\bxi \in \BA$ such that $1 < \pi^{i-1}(\xi) \leq k^r_{i-1}$,
we have $\pi^i(\xi)=\pi^{i-1}(\xi)-1$ and $\tau^{i-1}_\bxi=\tau^i_\bxi$, 
so $k^r_i - \pi^i(*) \geq k^r_{i-1}- \pi^{i-1}(*)$ and 
$\pi^i(*)-k^l_i \geq \pi^{i-1}(*)-k^l_{i-1}$. 

If the winner $\balpha$ of $\gamma_i$ appears in the $k$-th 
in $\pi^{i-1}$ with $\pi^{i-1}(*) < k \leq k^r_{i-1}+1$, then
\be \nonumber
\sum_{\pi^i(*)<\pi^i(\xi) \leq j} \tau^i_\bxi=
\sum_{\pi^{i-1}(*)<\pi^{i-1}(\xi) \leq j-1} \tau^{i-1}_\bxi>0 \quad\text{for all}
\; \pi(*) \leq j < k-1,
\ee
\be \nonumber
\sum_{\pi^i(*) < \pi^i(\xi) \leq j} \tau^i_\bxi=
\sum_{\pi^{i-1}(*) < \pi^{i-1}(\xi) \leq j} \tau^{i-1}_\bxi>0 \quad\text{for all}
\; k \leq j \leq k^r_{i-1}+1,
\ee
\be \nonumber
\sum_{\pi^i(*) < \pi^i(\xi) \leq k-1} \tau^i_\bxi=
-\sum_{2\leq \pi^{i-1}(\xi) < \pi^{i-1}(*)} \tau^{i-1}_\bxi+h^{i-1}_\balpha 
\geq h^{i-1}_\balpha>0,
\ee
which implies that 
$k^r_i-\pi^i(*) \geq \min \{k^r_{i-1}+1-\pi^{i-1}(*),2d-\pi^{i-1}(*)\}$,
hence $k^r_i-k^l_i \geq \min \{k^r_{i-1}-k^l_{i-1}+1, 2d - k^l_{i-1}\}$.

This shows that both (1) and (2) holds. Moreover, (3) also holds since 
its hypothesis can only be satisfied if $k_{i-1}^l=2$.
\end{proof}

\begin{cor}\label{non-empty}
If $\pi$ is irreducible then $\Ga_\pi$ is non-empty.
\end{cor}

\begin{proof}
Let $\gamma \in \Pi(\RR)$ be a strongly positive path starting 
and ending at $\pi$, which exists by Lemma \ref{strongly positive}. 
Then $(B_\gamma^*)^{-1}\cdot \Ga'_{\pi} \subset \Ga_{\pi}$ and 
by Lemma \ref{Theta'} the set $\Ga'_\pi$ is non-empty. 
Therefore $\Ga_\pi$ is non-empty.
\end{proof}

Given that $\Ga_\pi$ is non-empty, it is easy to see that $\Ga'_\pi \cup
\{0\}$ is in fact just the closure of $\Ga_\pi$.


\subsubsection{Extension of induction to the space of zippered rectangles}

Let $\gamma \in \Pi(\RR)$ be a path starting at $\pi$ and define 
$\Ga_\gamma$ satisfying:
\be \nonumber
B_\gamma^* \cdot \Ga_\gamma= \Ga_\pi.
\ee

If $\gamma$ is a right arrow ending at $\pi'$, then 
$\Ga_\gamma=\{ \tau \in \Ga_{\pi'} \di \sum_{x \in \AA}
\tau_x<0 \}$, and if $\gamma$ is a left arrow ending at $\pi'$, then
$\Ga_\gamma=\{ \tau \in \Ga_{\pi'} \di \sum_{x \in \AA} \tau_x>0\}$.

Thus, the map
\be \nonumber
\widehat Q^\gamma:\Delta_\gamma \times \Ga_\pi \to
\Delta_{\pi'} \times \Ga_\gamma,
\quad \widehat Q^\gamma(\lambda,\pi,\tau)=
(Q(\lambda,\pi),(B_\gamma^*)^{-1} \cdot \tau)
\ee
is invertible.
With this we define an invertible skew-product $\widehat Q$ 
over $Q$ conside\-ring all $\widehat Q^\gamma$ for every arrow 
$\gamma$. So, we obtain a map from $\bigcup
(\Delta_\gamma \times \Ga_\pi)$ (where the union is taken over all $\pi
\in \RR$ and all arrows $\gamma$ starting at $\pi$) to $\bigcup
(\Delta_{\pi'} \times \Ga_\gamma)$ (where the union is taken over all
$\pi' \in \RR$ and all arrows ending at $\pi'$).
Denote $\widehat
\Delta_\RR=\bigcup_{\pi \in \RR} (\Delta_\pi \times \Ga_\pi)$.

Let $(e_\balpha)_{\balpha \in \BA}$ be the canonical basis of $\R^\BA$.
We will consider a measure in $\widehat \Delta_\RR$ defined as follows.
Let $\{v_1,\ldots,v_d\}$ be a basis of $\R^d$. We have a volume form 
given by $\omega(v_1,\ldots,v_d)=\det(v_1,\ldots,v_d)$. 
We want to define a volume form in $\widehat\Delta_\RR$ 
coherent with $\omega$. Given the subspace $\mathcal{S}_\pi$ define 
the orthogonal vector 
$v_\pi=\sum_{\pi(x)<\pi(*)} e_{\bx}-\sum_{\pi(x)>\pi(*)} e_{\bx}$. 
We can view this vector like a linear functional 
$\psi_\pi: \R^d \to \R$ defined by $\psi_\pi(x)=\langle v_\pi, x \rangle$. 
Now we define a form $\omega^{v_\pi}$ on $\mathcal{S}_\pi$ 
such that $\omega=\omega^{v_\pi} \wedge \psi_\pi$ (where $\wedge$ 
denotes the exterior product). We have that $\omega^{v_\pi}$ is a 
$(d-1)$-form and it is well defined in the orthogonal complement of 
$\psi_\pi$ (i.e., $\psi_\pi(\R^d)^{\perp}=\ker(\psi_\pi)=\mathcal{S}_\pi$). 
Notice that, given $\pi, \pi' \in \RR$ and a path $\gamma \in \Pi(\RR)$ joining 
$\pi$ to $\pi'$, $B_\gamma \cdot v_\pi = v_{\pi'}$ and
\be \nonumber
(B^*_\gamma)^{-1} \cdot \left( \frac{v_\pi}{\langle v_\pi,v_\pi \rangle} \right) -
\frac{v_{\pi'}}{\langle v_{\pi'}, v_{\pi'} \rangle} \in \mathcal{S}_{\pi'}.
\ee 
So, the pull-back of $\omega^{v_{\pi'}}$ is equal to $\omega^{v_\pi}$, i.e.,
\be \nonumber
[(B_\gamma^*)^{-1}]^* \omega^{v_{\pi'}}=\omega^{v_\pi}. 
\ee

Consider the volume form $\omega^{v_\pi}$ and the 
corresponding Lebesgue measure $\Leb_\pi$ on $\mathcal{S}_\pi$.
So we have a natural volume measure $\widehat m_\RR$ on 
$\widehat \Delta_\RR$ which is a product of a counting measure 
on $\RR$ and  the restrictions of $\Leb_\pi$ on $\mathcal{S}_\pi^+$ 
and $\Ga_\pi$.

Let $\phi(\lambda, \pi, \tau)=\|\lambda\|=\sum_{\balpha \in \BA}\lambda_{\balpha}$. 
The subset $\Om_\RR \subset \widehat \Delta_\RR$ of all $x$ such
that either
\begin{itemize}
\item $\widehat Q(x)$ is defined and $\phi(\widehat Q(x))<1 \leq \phi(x)$,
\item $\widehat Q(x)$ is not defined and $\phi(x) \geq 1$,
\item $\widehat Q^{-1}(x)$ is not defined and $\phi(x)<1$.
\end{itemize}
is a fundamental domain for the action of $\widehat Q$: each orbit
of $\widehat Q$ intersects $\Om_\RR$ exactly once.

Let $\Om^{(1)}_\RR$ be the subset of $\Om_\RR$ such that 
$A(\lambda, \pi, \tau)=1$. 
Let $m^{(1)}_\RR$ be the restriction of 
the measure $\widehat m_\RR$ to the subset $\Om^{(1)}_\RR$.

\subsubsection{The Veech flow with involution} \label{veechteichmullerflow}

Using the coordinates introduced before, we define 
a flow $\TV=(\TV_t)_{t \in \R}$ 
on $\widehat \Delta_\RR$ given by 
$\TV_t(\lambda,\pi,\tau)=(e^t \lambda,\pi,e^{-t} \tau)$.
It is clear that $\TV$ commutes with the map $\widehat Q$. 
The \emph{Veech Flow with involution} is defined by
$\VF_t:\Om_\RR \to \Om_\RR$, $\VF_t(x)=\widehat Q^n(\TV_t(x))$. 

Notice that the Veech flow with involution leaves invariant the space 
of zippered rectangles of area one.
So, the restriction $\VF_t:\Om_\RR^{(1)} \to \Om_\RR^{(1)}$ 
leaves invariant the volume form which, as we will see later, is finite.

\section{The distortion estimate} \label {distortion}

We will introduce a class of measures involving the Lebesgue 
measure and its forward iterates under the renormalization map.

For $q\in \R_+^\BA$, let
$\e_{\pi,q}=\{\lambda \in \mathcal{S}_\pi^+ \tq \langle \lambda, q \rangle<1\}$.
Let $\nu_{\pi,q}$ be the measure on the $\sigma$-algebra of subsets of
$\mathcal{S}_\pi$ which are invariant under multiplication by
positive scalars, given by
$\nu_{\pi,q}(A)=\Leb_\pi(A \cap \e_{\pi,q})$.  If $\gamma$ is a path
starting at $\pi$ and ending at $\pi'$ then
\be \nonumber
\nu_{\pi,q}(B_\gamma^* \cdot A)= \Leb_\pi( (B_\gamma^* \cdot A) \cap \e_{\pi,q}) =
\Leb_{\pi'}( A \cap \e_{\pi',B_\gamma\cdot q}) = \nu_{\pi',B_\gamma\cdot q}(A).
\ee

We will obtain estimates for $\nu_{\pi,q}(\Delta_\gamma)$.

Let $\RR \subset \ssigma(\AA)$ be a Rauzy class, $\gamma \in \Pi(\RR)$, let
$\pi$ and $\pi'$ be the start and the end of $\gamma$, respectively. 
We denote $\e_{\gamma,q} \times \{\pi\}=(\e_{\pi,q}\times \{\pi\}) \cap
\Delta_\gamma$, so that $B_\gamma^* \cdot \e_{\pi',B_\gamma \cdot
q}=\e_{\gamma,q}$.




For $\AA' \subset \AA$ non-empty and invariant by
involution, let $\M_{\AA'}(q)=\max_{\alpha \in \AA'} q_\alpha$ 
and $\M(q)=\M_\AA(q)$. Consider also 
$\m_{\AA'}(q)=\min_{\alpha \in \AA'} q_\alpha$ and
$\m(q)=\m_\AA(q)$.

If $\Gamma \subset \Pi(\RR)$ is a set of paths starting at the same $\pi
\in \RR$, we denote $\e_{\Gamma,q}=
\bigcup_{\gamma \in \Gamma} \e_{\gamma,q}$.
Given $\Gamma \subset \Pi(\RR)$ and $\gamma_s \in \Pi(\RR)$
we define $\Gamma_{\gamma_s} =\{ \gamma \in \Gamma \tq \gamma_s$ 
is the start of $\gamma\}$. 

We say that a vertice is \emph{simple} or \emph{double} depending whether 
it is labelled by a simple or a double letter, respectively.
Notice that $\e_{\pi,q}$ is a convex open polyhedron which
vertices are:
\begin{itemize}
\item the trivial vertex $0$;
\item the simple vertices $q_{\balpha}^{-1}
e_{\balpha}$, where ${\alpha}$ is simple;
\item the double vertices $(q_{\balpha}+q_{\bbeta})^{-1}
(e_{\balpha}+e_{\bbeta})$, where $\alpha,\beta$ 
are double and
$\pi(\alpha)<\pi(*)<\pi(\beta)$ or $\pi(\beta)<\pi(*)<\pi(\alpha)$.
\end{itemize}
A simple vertex $v=q_{\balpha}^{-1} e_{\balpha}$
is called of type $\balpha$ and weight $w(v)=q_{\balpha}$, and a double
vertex $v=(q_{\balpha}+q_{\bbeta})^{-1}
(e_{\balpha}+e_{\bbeta})$ is called of type
$\{\balpha,\bbeta\}$ and of weight $w(v)=q_{\balpha}+q_{\bbeta}$. 






An \emph{elementary subsimplex} of $\e_{\pi,q}$ is an open simplex
whose vertices are also vertices of $\e_{\pi,q}$, and one of them is $0$.  Notice
that $\e_{\pi,q}$ can be always written as a union of at most $C_1(d)$
elementary simplices, up to a set of codimension one.


A set of non-trivial vertices of $\e_{\pi,q}$ is
contained in the set of vertices of some elementary subsimplex if and
only if the vertices are linearly independent.  If $\alpha$
is simple then any elementary subsimplex must have a vertex of type
$\balpha$ and if $\alpha$ is double then any elementary subsimplex must
have a vertex of type $\{\balpha,\bx\}$ for some 
$\bx \neq \balpha$ with $x$ double.
If $\e$ is an elementary subsimplex with simple vertices of type $\balphai$ and
double vertices of type $\{\bbetaj1,\bxij2\}$ then
$\nu_{\pi,q}(\e)=k(\pi,\e) \prod q_{\balphai}^{-1} \prod
(q_{\bbetaj1}+q_{\bxij2})^{-1}$ where $k(\pi,\e)$ is a positive
integer only
depending on $v_\pi$ and on the types of the 
double vertices of $\e$. 
In particular there is an integer $C_2(d)$ such that
$k(\pi,\e) \leq C_2(d)$.

Let $\gamma \in \Pi(\RR)$ be an arrow starting at $\pi$
and ending at $\pi'$. If $\Gamma \subset \Pi(\RR)$ 
then we define

\be \nonumber
P_q(\Gamma \di \gamma)=
\frac{\nu_{\pi',B_\gamma \cdot q}(\e_{\Gamma_\gamma,q})}
{\nu_{\pi,q}(\e_{\gamma,q})}
\ee
and
\be \nonumber
P_q(\gamma \di \pi)=
\frac{\nu_{\pi',B_\gamma \cdot q}(\e_{\gamma,q})}{\nu_{\pi,q}(\e_{\pi,q})}.
\ee
We have that $P_q(\Gamma \di \gamma)=P_{B_\gamma \cdot q}(\Gamma_\gamma \di \pi')$.

We define a partial order in the set of paths as follows. Let $\gamma, \gamma_s \in \Pi(\RR)$ 
be two paths. We say that $\gamma_s \leq \gamma$ if and only if 
$\gamma_s$ is the start of $\gamma$.
A family $\Gamma_s \subset \Pi(\RR)$ is called \emph{disjoint} if no two
elements are comparable by the partial order defined before.
If $\Gamma_s$ is disjoint and $\Gamma \subset \Pi(\RR)$ is a family 
such that any $\gamma \in \Gamma$ starts by some element
$\gamma_s \subset \Gamma_s$, then for every $\pi \in \RR$

\be \nonumber
P_q(\Gamma \di \pi)=\sum_{\gamma_s \in \Gamma_s} 
P_q(\Gamma \di \gamma_s)P_q(\gamma_s \di \pi) 
\leq P_q(\Gamma_s \di \pi)\sup_{\gamma_s \in \Gamma_s} 
P_q(\Gamma \di \gamma_s).
\ee

\begin{lem}\label{Pgamma}
There exists $C_3(d)<1$ with the following property.  Let
$q \in \R^{\BA}_+$, $\gamma \in \Pi(\RR)$ be 
an arrow starting at
$\pi$ with loser $\beta$. If $C \geq 1$ is such that
$q_{\bbeta}>C^{-1}M(q)$ then
\be \nonumber
P_q(\gamma \di \pi)>C_3(d) C^{-(d-1)}.
\ee
\end{lem}

\begin{proof}
Let $\alpha$ be the winner of $\gamma$ and let $\pi'$ 
be the end of $\gamma$. 
Let $\e$ be an elementary subsimplex of $\e_{\pi,q}$. 
We are going to show that there exists an elementary subsimplex 
$\e' \subset B_\gamma (\e)$ 
of $\e_{\pi',B_\gamma \cdot q}$
such that $\Leb_{\pi'}(\e') \geq C'^{-1} \Leb_{\pi}(\e)$, which implies
the result by decomposition.

We will separate the proof in four cases depending on whether the
winner and the loser are simple or double.

Suppose that $\alpha$ and $\beta$ are simple.
Let $\e$ be an elementary subsimplex of $\e_{\pi,q}$.
Then $\e'=B_\gamma \cdot
(\e_{\pi,q} \cap \e)$ is an elementary subsimplex of
$\e_{\pi',B_\gamma \cdot q}$.  The
set of vertices of $\e'$ differs from the set of vertices of $\e$ just by
replacing the vertex $q_{\bbeta}^{-1} e_{\bbeta}$ 
by $(q_{\balpha}+q_{\bbeta})^{-1}
e_{\bbeta}$.  It follows that
$\Leb_{\pi'}(\e')/\Leb_\pi(\e)=q_{\bbeta}/
(q_{\balpha}+q_{\bbeta})>1/(C+1)$.
By considering a
decomposition into elementary subsimplices we conclude that
$P_q(\gamma|\pi)=q_{\bbeta}/(q_{\balpha}+q_{\bbeta})>1/(C+1)$.

Suppose that the winner is simple and the loser is double.
Let $\e$ be an elementary subsimplex of $\e_{\pi,q}$.
Then $\e'=B_\gamma \cdot
(\e_{\pi,q} \cap \e)$ is an elementary subsimplex of $\e_{B_\gamma \cdot q}$.
The set of vertices of $\e'$ differs from the set of vertices of
$\e$ just by replacing the vertices $(q_{\bx}+q_{\bbeta})^{-1}
(e_{\bx}+e_{\bbeta})$ by $(q_{\balpha}+q_{\bx}+q_{\bbeta})^{-1}
(e_{\bx}+e_{\bbeta})$.  It follows that
$\Leb_{\pi'}(\e')/\Leb_\pi(\e)=
\prod (q_{\bx}+q_{\bbeta})/(q_{\balpha}+q_{\bx}+q_{\bbeta})>1/(1+C)$,
where the product is over all $\bx$ such that $\e$ has a vertex of type
$\{\bx,\bbeta\}$.  Thus $P_q(\gamma|\pi)>1/(C+1)$.

If the winner is double and the loser is simple, let $\gamma'$ be the other
arrow starting at $\pi$.  Analogous to the
previous case, $P_q(\gamma' \di \pi)=
\prod (q_{\bx}+q_{\balpha})/(q_{\balpha}+q_{\bx}+q_{\bbeta})
<\prod \left(1-q_{\bbeta}/(q_{\balpha}+q_{\bx}+q_{\bbeta})\right)
<2C/(1+2C)$, so
$P_q(\gamma|\pi)>1/(1+2C)$.

Finally, suppose that the winner and the loser 
are both double.  Let $\e$ be an elementary subsimplex with
$\Leb_\pi(\e) \geq \Leb_\pi(\e_{\pi,q})/C_1(d)$.
Let $Z$ be the set of vertices of $\e$ and let $\tilde Z
\subset Z$ be the set of double vertices
of type $\{q_{\bx},q_{\bbeta}\}$ with $\bx \neq \balpha$. 
Notice that $B_\gamma \cdot
(Z \backslash \tilde Z)$ is a subset of the set of vertices of 
$\e_{\pi',B_\gamma
\cdot q}$.  Since $Z \backslash \tilde Z$ is linearly independent, 
$B_\gamma \cdot
(Z \backslash \tilde Z)$ is also.  Thus there exists an elementary
subsimplex $\e'$ of $\e_{\pi',B_\gamma \cdot q}$
whose set $Z'$ of vertices contains $B_\gamma \cdot
(Z \backslash \tilde Z)$.  Let $\tilde Z'=Z' \backslash B_\gamma \cdot
(Z \backslash \tilde Z)$.
The weight of a vertice $v \in Z \backslash \tilde
Z$ is the same weight as the weight of $B_\gamma \cdot v$.  Notice that each
vertex of $\tilde Z$ has weight at least $C^{-1} \M(q)$ and each vertex of
$\tilde Z'$ has weight at most $2\M(B_\gamma \cdot q) \leq 4 \M(q)$.  Thus
\be \nonumber
\frac {\Leb_{\pi'}(\e')} {\Leb_\pi(\e)}=\frac{k(\pi',\e')} {k(\pi,\e)} \frac {\prod_{v
\in \tilde Z}
w(v)} {\prod_{v \in \tilde Z'} w(v)}>C_2(d)^{-1} (4 C)^{1-d}.
\ee
Thus $P_q(\gamma|\pi)>C_1(d)^{-1} C_2(d)^{-1} (4C)^{1-d}$.
\end{proof}

The proof of the recurrence estimates is based on the analysis of the Rauzy
renormalization map. The key step involves a control on the measure of
sets which present big distortion after some long (Teichm\"{u}ller) time.


\begin{thm} \label {4}

There exists $C_4(d)>1$ with the following property.  
Let 
$q \in \R^\BA_+$.
Then for every $\pi \in \RR$,
\be \nonumber
P_q(\gamma \in \Pi(\RR),\, \M(B_\gamma \cdot q)>C_4(d) \M(q) \text { and }
\m(B_\gamma \cdot q)<\M(q) \di \pi) < 1-C_4(d)^{-(d-1)}.
\ee

\end{thm}

\begin{proof}
For $1\leq k\leq d$, let $m_k(q)=\max m_{\AA'}(q)$ where the maximum is taken 
over all involution invariant sets $\AA' \subset \AA$ such that 
$\# \AA'=2k$. In particular $m=m_d$. We will show that for $1\leq k\leq d$
there exists $D>1$ such that
\be \label{estimativa}
P_q(\gamma \in \Pi(\RR),\, \M(B_\gamma \cdot q)>D \M(q) \text { and }
\m_k(B_\gamma \cdot q)<\M(q) \di \pi) < 1-D^{-(d-1)}.
\ee
(the case $k=d$ implies the desired statement). 
The proof is by induction 
on $k$. For $k=1$ it is obvious, by lemma \ref{Pgamma}. Assume that it is proved for some 
$1\leq k < d$ with $D=D_0$. Let $\Gamma$ be the set of minimal paths 
$\gamma$ starting at $\pi$ with $\M(B_\gamma \cdot q)>D_0 \M(q)$. Then
there exists $\Gamma_1 \subset \Gamma$ with $P_q(\Gamma_1 \di \pi)>D_1^{-(d-1)}$
and an involution invariant set $\AA' \subset \AA$ with $\# \AA'=2k$ 
such that if $\gamma \in \Gamma_1$ then 
$m_{\AA'}(B_{\gamma}\cdot q)\geq M(q)$.

For $\gamma_s \in \Gamma_1$, choose a path $\gamma=\gamma_s \gamma_e$
with minimal length such that $\gamma$ ends at a permutation $\pi_e$
such that either the first or the last element of $\pi_e$ is an element of 
$\AA \backslash \AA'$. Let $\Gamma_2$ be the collection of the 
$\gamma=\gamma_s \gamma_e$ thus obtained. 
Then $P_q(\Gamma_2 \di \pi)>D_2^{-(d-1)}$ and $\M(B_\gamma \cdot q)<D_2\M(q)$
for $\gamma \in \Gamma_2$.

Let $\Gamma_3$ be the space of minimal paths 
$\gamma=\gamma_s \gamma_e$ with $\gamma_s \in \Gamma_2$ and 
$M(B_\gamma \cdot q)>2dD_2\M(q)$. Let $\Gamma_4 \subset \Gamma_3$
be the set of all $\gamma=\gamma_s \gamma_e$ where all 
the arrows of $\gamma_e$ have as loser an element of $\AA'$. 
For each $\gamma_s \in \Gamma_2$, there exists at most one 
$\gamma=\gamma_s \gamma_e \in \Gamma_4$, and if  
$P_q(\Gamma_4 \di \gamma_s)< \frac{1}{2d}$, it follows that 
$P_q(\Gamma_3 \backslash \Gamma_4 \di \pi)>\left(1-\frac{1}{2d}\right)D_2^{-(d-1)}$. 
It remains to prove that $P_q(\Gamma_4 \di \gamma_s)< \frac{1}{2d}$.

Let $\gamma=\gamma_s\gamma_e \in \Gamma_4$ such that $\gamma_s \in \Gamma_2$. 
Let $\pi_e \in \Pi(\RR)$ be the end of $\gamma_s$ and let $\alpha$ and $\beta$ 
be the winner and the loser of $\pi_e$, respectively. 
By definition, we have that $\alpha \in \AA \backslash \AA'$ 
and $\beta \in \AA'$. Besides, all losers of $\gamma_e$ are in $\AA'$. 

We claim that $\alpha$ is simple. Suppose this is not the case. 
Assume, without lost of generality, $\pi_e(\alpha)=1$ and 
$\pi_e(\beta)=2d+1$. Applying Rauzy algorithm with involution 
one time we would obtain 
$\pi'_e(\beta)<\pi'_e(*)$ and to keep the same winner $\alpha$ 
we just can apply Rauzy algorithm with involution at most $2d-4$ times. But even if we could 
apply Rauzy algorithm with involution those number of times, we will have
\be \nonumber
\M(B_\gamma \cdot q)<(2d-3)D_2\M(q)<2dD_2\M(q)
\ee
what contradicts that $\gamma \in \Gamma_3$. 
Then $\alpha$ is simple as we claim.


Suppose $\gamma=\gamma_1\ldots\gamma_n$ where $\gamma_i \in \Pi(\RR)$ 
are arrows joining $\pi_e^{(i-1)}$ and $\pi_e^{(i)}$. 
We have $\gamma_s \in \Gamma_2$, so, to obtain 
$\M(B_\gamma \cdot q)>2dD_2\M(q)$ we need $n \geq 2d+1$. We 
also have $\mathcal{S}_{\pi_e}=\mathcal{S}_{\pi_e^{(i)}}$ for all $i\in \{1,\ldots,n\}$.

Let $\e=\e_{\pi_e,B_{\gamma_s} \cdot q}$ which is a finite union of 
elementary simplices $\e_j$. Let $\beta_0=\beta$ and $\beta_i$ be 
the loser of $\pi_e^{(i)}$ for $i=0,\ldots,n$. For each $i\in \{0,\ldots,n\}$ 
if $\beta_i$ is a simple vertex then all $\e_j$ has a vertex of 
type $\bbetai$ and if $\beta_i$ is double then all $\e_j$ 
has a vertex of type $\{\bbetai,\bx\}$ for some 
$\bx \notin \{\balpha,\bbetai\}$.
Let $\e_j^{(n)}=B_{\gamma_e}(\e_j)$. Notice that the type of 
vertices of $\e_j^{(n)}$ coincides with type of vertices of 
$\e_j$. Let $Z$ be the set of vertices of $\e_j$ and 
$Z^{(n)}$ be the set of vertices of $\e_j^{(n)}$. Then 
\be \label{PGamma_4}
\frac {\Leb_{\pi^{(n)}_e}(\e_j^{(n)})}{\Leb_{\pi_e}(\e_j)}=
\frac {\prod_{v \in Z} w(v)} 
{\prod_{v \in Z^{(n)}} w(v)}
\ee
But $\M(B_{\gamma_s} \cdot q)<D_2\M(q)$ and $\M(B_{\gamma}\cdot q)>2dD_2\M(q)$, 
so there is one term in (\ref{PGamma_4}) which is less then $\frac{1}{2d}$. 
Thus
\be \nonumber
P_q(\Gamma_4 \di \gamma_s)<\frac{1}{2d}
\ee
Let $\gamma=\gamma_s\gamma_e \in \Gamma_3 \backslash \Gamma_4$. 
Let us show that $m_{k+1}(B_\gamma \cdot q)>M(q)$, 
which implies (\ref{estimativa}) with $k+1$ in place of $k$ 
and $D=2dD_2$. Assume that this is not the case. In this case, 
the last arrow composing $\gamma_e$ must have as loser an 
element of $\AA'$. Moreover, no arrow composing $\gamma_e$ 
has as winner an element of $\AA'$ (otherwise, the loser $\beta$ 
of the first such arrow does not belong to $\AA'$ 
and is such that 
$m_{\AA' \cup \{\beta\}}(B_\gamma \cdot q)>\M(q)$). 
Let $\gamma_e=\gamma_{e,s}\gamma_{e,e}$ where $\gamma_{e,s}$ 
is maximal such that all losers of $\gamma_{e,s}$ are in $\AA'$. 
Then all losers in $\gamma_{e,s}$ are distinct and 
$\M(B_{\gamma_s\gamma_{e,s}} \cdot q) <2D_2\M(q)$. Let 
$\gamma_{e,e}=\gamma_{e,1}\ldots\gamma_{e,l}$ where
$\gamma_{e,j}=\gamma_{e,j,s}\gamma_{e,j,e}$ with $\gamma_{e,j,s}$ and 
$\gamma_{e,j,e}$ non-trivial such that all the losers of $\gamma_{e,j,s}$ 
are in $\AA \backslash \AA'$ and all the losers of $\gamma_{e,j,e}$ 
are in $\AA'$. Let $\gamma_j=\gamma_s\gamma_{e,s}\gamma_{e,1}\ldots\gamma_{e,j}$,
 $0\leq j \leq l$. Notice that for each $j$, $\gamma_{e,j,e}$ 
has distinct losers, and the same winner $\alpha \in \AA \backslash \AA'$ 
which is also the last winner of $\gamma_{e,j,s}$.
Let $\beta \in \AA \backslash \AA'$ be the last loser of $\gamma_{e,j,s}$. 
Then
\be \nonumber
\M(B_{\gamma_{j+1}}\cdot q) - \M(B_{\gamma_j}\cdot q) \leq
\M_{\bbeta}(B_{\gamma_{j+1}}\cdot q) - \M_{\bbeta}(B_{\gamma_j}\cdot q)
\ee
which implies that
\be \nonumber
(2d-1)D_2\M(q) \leq \M(B_\gamma \cdot q) - \M(B_{\gamma_s}\cdot q) 
\leq
\sum_{\beta \in \AA \backslash \AA'}\M_\beta(B_\gamma \cdot q) - 
\M_\beta(B_{\gamma_s}\cdot q) \leq dD_2\M(q)
\ee
which is a contradiction.
\end{proof}

\section{Recurrence estimates}
\label{sec:proof_rec}

\begin{lemma}
\label{lem61}
For every $\hat \gamma \in \Pi(\RR)$, there exist $M \geq 0$, $\rho<1$
such that for every
$\pi \in \RR$, $q \in \R^\BA_+$,
\be \nonumber
P_q(\gamma \text { can not be written as }
\gamma_s \hat \gamma \gamma_e
\text { and } \M(B_\gamma \cdot q)>2^M \M(q) \di \pi)
\leq \rho.
\ee
\end{lemma}

\begin{proof}
Fix $M_0 \geq 0$ large and let $M=2 M_0$.
Let $\Gamma$ be the set of all minimal paths $\gamma$ starting at $\pi$
which can not be written as $\gamma_s \hat \gamma \gamma_e$ and such that
$\M(B_\gamma \cdot q)>2^M \M(q)$. Any path
$\gamma \in \Gamma$ can be written as
$\gamma=\gamma_1 \gamma_2$ where $\gamma_1$ is minimal with $\M(B_{\gamma_1}
\cdot q)>2^{M_0} \M(q)$.  Let $\Gamma_1$ collect the possible $\gamma_1$. 
Then $\Gamma_1$ is disjoint, by minimality. Let $\tilde \Gamma_1 \subset \Gamma_1$ be the
set of all $\gamma_1$ such that $\M_{\AA'}(B_{\gamma_1} \cdot q) \geq \M(q)$
for all invariant involution set $\AA' \subset \AA$ non-empty.  
By Theorem \ref {4}, if $M_0$ is sufficiently large we have
\be \nonumber
P_q(\Gamma_1 \backslash \tilde \Gamma_1 \di \pi)<\frac {1} {2}.
\ee

For $\pi_e \in \RR$, let $\gamma_{\pi_e}$ be a shortest possible path
starting at $\pi_e$ with $\gamma_{\pi_e}=\gamma_s \hat \gamma$. Let 
$\pi_f$ be the end of $\gamma_{\pi_e}$. If $M_0$ is sufficiently 
large then $\norm{B_{\gamma_{\pi_e}}}<
\frac {1} {d-1} 2^{M_0-1}$.  It follows that if $\gamma_1
\in \Gamma_1$ ends at $\pi_e$ then
\be \nonumber
P_q(\Gamma \di \gamma_1) \leq 1-P_{B_{\gamma_1} \cdot q}(\gamma_{\pi_e} \di
\pi_e).
\ee

Let $\e \subset \e_{\pi_e,B_{\gamma_1}\cdot q}$ be an elementary 
subsimplex with 
\be \nonumber
\Leb_{\pi_e}(\e) \geq \frac{\Leb_{\pi_e}(\e_{\pi_e,B_{\gamma_1}\cdot q})}{C_1(d)}. 
\ee
Choose an elementary subsimplex $\e'$ of $\e_{\pi_f,B_{\gamma_1\gamma_{\pi_e}}\cdot q}$, 
such that for all $\balpha \in \BA$ there exists a vertice $v'$ of $\e'$ 
of type $\balpha$ or of type $\{\balpha, \bxi \}$ for some $\bxi \in \BA$ 
and $\Leb_{\pi_f}(\e')\geq \Leb_{\pi_f}(B_{\gamma_{\pi_e}} 
\cdot \e_{\pi_e,B_{\gamma_1,q}})/C_1(d)$. 
Let $Z$ and $Z'$ be the set of vertices of $\e$ and $\e'$, respectively. 
If furthermore $\gamma_1 \in \tilde \Gamma_1$ then
\be \nonumber
\frac{\Leb_{\pi_f}(\e')}{\Leb_{\pi_e}(\e)}=
\frac{k(\pi_f,\e')\prod_{v \in Z}w(v)}
{k(\pi_e,\e)\prod_{v \in Z'}w(v)} 
\geq \frac {\M(q)^{d-1}} {(2^{2 M_0} \M(q))^{d-1}}=2^{-2 (d-1) M_0}.
\ee
So, $P_{B_{\gamma_1} \cdot q}(\gamma_{\pi_e} \di \pi_e) \geq 
2^{-2(d-1)M_0}$ and $P_q(\Gamma \di \pi) \leq 1-2^{-2(d-1)M_0-1}$.
\end{proof}

\begin{prop} \label {6.2}
For every $\hat \gamma \in \Pi(\RR)$,
there exist $\delta>0$, $C>0$ such that for every $\pi \in \RR$, $q \in
\R^\BA_+$ and for every $T>1$
\be \nonumber
P_q(\gamma \text { can not be written as }
\gamma_s \hat \gamma
\gamma_e \text { and } \M(B_\gamma \cdot q)>T \M(q) \di \pi)
\leq C T^{-\delta}.
\ee
\end{prop}

\begin{proof}
Let $M$ and $\rho$ be as in the previous lemma.  Let $k$ be maximal with $T
\geq 2^{k(M+1)}$.
Let $\Gamma$ be the set of minimal paths
$\gamma$ such that $\gamma$ is not of the form
$\gamma_s \hat \gamma \gamma_e$ and $\M(B_\gamma \cdot q)>2^{k(M+1)} \M(q)$.
Any path $\gamma \in
\Gamma$ can be written as $\gamma_1\ldots\gamma_k$ where
$\gamma_{(i)}=\gamma_1\ldots\gamma_i$ is minimal with $\M(B_{\gamma_{(i)}}
\cdot q)>2^{i (M+1)} \M(q)$.
Let $\Gamma_{(i)}$ collect the $\gamma_{(i)}$.  Then
the $\Gamma_{(i)}$ are disjoint. Moreover, by Lemma \ref{lem61}, 
for all $\gamma_{(i)} \in
\Gamma_{(i)}$,
\be \nonumber
P_q(\Gamma_{(i+1)} \di \gamma_{(i)}) \leq \rho.
\ee
This implies that $P_q(\Gamma \di \pi) \leq \rho^k$.  The result follows.
\end{proof}

\begin{rem}
Notice that in the case of \cite{AGY}, they obtain a better recurrence estimate. 
In fact, they obtain $T^{-(\delta-1)}$ instead $T^{-\delta}$. But our estimate will be 
enough.
\end{rem}

\section{Construction of an excellent hyperbolic semi-flow}
\label{excelent}



\subsection{The Veech flow with involution as a suspension over the Rauzy renormalization}

Let $\widehat \Up_\RR$ be the subset of $\Om_\RR$ of all
$(\lambda,\pi,\tau)$ with $\phi(\lambda,\pi,\tau)=\|\lambda\|=
\sum_{\balpha \in \BA} \lambda_\balpha=1$. We denote by $\widehat \Up_\pi$ 
the connected components of $\widehat \Up_\RR$. Consider
$\widehat \Up^{(1)}_\RR=\Om^{(1)}_\RR \cap \widehat \Up_\RR$ and
$\widehat \Up^{(1)}_\pi=\Om^{(1)}_\RR \cap \widehat \Up_\pi$.  
Let $\widehat m^{(1)}_\RR$ be
the induced Lebesgue measure to $\widehat \Up^{(1)}_\RR$.

We have that $\widehat \Up^{(1)}_\RR$ is transverse to the Veech 
flow with involution on $\Om^{(1)}_\RR$ which is given by a certain 
iterate of the Rauzy induction with involution, after 
applying the flow $\TV_t$. We are interested 
in the first return map $\widehat R$ to this section. 
The domain of $\widehat R$ is the intersection of $\widehat
\Up^{(1)}_\RR$ with the domain of definition of $\widehat Q$, and we have
\be \nonumber
\widehat R(\lambda,\pi,\tau)=(e^r \lambda',\pi',e^{-r} \tau'),
\ee
where $(\lambda',\pi',\tau')=\widehat Q(\lambda,\pi,\tau)$ and
$r=r(\lambda,\pi)=-\ln \|\lambda'\|=-\ln \phi(\widehat Q(\lambda,\pi,\tau))$
is the first return time.  Notice that the map $\widehat R$ 
is a skew-product $\widehat R(\lambda,\pi,\tau)=(R(\lambda,\pi),e^{-r} \tau')$ 
over the non-invertible map $R$ defined by $R(\lambda, \pi)=(e^r \lambda',\pi)$. 
The map $\widehat R$ is called the \emph{Rauzy renormalization map with involution} 
and it preserves the measure $\widehat m^{(1)}_\RR$. The renormalization map $\widehat R$ 
is an ``invertible extension'' of the map $R$.

We can see the Veech flow with involution as a suspension over 
the renormalization map $\widehat R$. In this suspension model, 
we lose the control of the orbits which do not return to 
$\widehat \Up^{(1)}_\RR$. However, this do not cause any 
problem to our considerations because the set of such 
orbits has zero Lebesgue measure.

\subsection{Precompact section}

The suspension model for the Veech flow with involution 
presented above is obtained over a discrete transformation 
$\widehat R$ which is not sufficiently hyperbolic.  
In general, $\widehat R$ can not be expected to be uniformly hyperbolic, 
in fact, it does not even have appropriate distortion properties. This is 
related to the fact that the domain is not compact. The approach 
taken in \cite {AGY} and other recent works as \cite {AF07} and 
\cite {AV3} is to introduce a class of suitably small (precompact 
in $\widehat \Up^{(1)}_\RR$) sections, and to prove that the corresponding 
return maps have good distortion properties.

So, we will choose a specific precompact section which is the 
intersection of $\widehat \Up^{(1)}_\RR$ with
(finite unions of) sets of the form $\Delta_\gamma \times \Ga_{\gamma'}$. 
Let $\gamma$ be a path starting at $\pi_s$ and ending at $\pi_e$.  
Precompactness in the $\lambda$ direction is equivalent to having
$B^*_\gamma \cdot (\mathcal{\overline S}^+_{\pi_e} \backslash \{0\}) \subset
\mathcal{S}^+_{\pi_s}$, which is a necessary condition if 
$\gamma$ is a positive path.  To take care of both the 
$\lambda$ and the $\tau$ direction, we have already 
introduced the notion of strongly positive.

Let $\Pi(\pi) \subset \Pi(\RR)$ be the set of paths 
starting and ending at the same $\pi \in \RR$.
Let $\pi \in \RR$ and
let $\gamma_* \in \Pi(\pi)$ be a strongly positive path. Assume further
that if $\gamma_*=\gamma_s\gamma=\gamma \gamma_e$ then either
$\gamma=\gamma_*$ or $\gamma$
is trivial. We will say that $\gamma_*$ is \emph{neat}.
If for example $\gamma_*$ ends by a left arrow and starts by a sufficiently 
long (at least half the length of $\gamma_*$) sequence 
of right arrows then the last condition of being neat 
is automatically satisfied.

Let $\widehat \De=\widehat \Up^{(1)}_\RR \cap
(\Delta_{\gamma_*} \times \Ga_{\gamma_*})$ and let $\De=\Up^0_\RR \cap
\Delta_{\gamma_*}$.
We will study the first return map $T_{\widehat \De}$ to the section 
$\widehat \De$ under the Veech flow with involution.
Notice that the connected components of its domain are given by
$\widehat \Up^{(1)}_\RR \cap (\Delta_{\gamma\gamma^*} 
\times \Ga_{\gamma_*})$, where
$\gamma$ is either $\gamma_*$, or a minimal path of the form 
$\gamma_* \gamma_0 \gamma_*$ not beginning by $\gamma_*\gamma_*$.  
The restriction of $T_{\widehat \De}$ to each connected component of its
domain has the expression
\be \nonumber
T_{\widehat \De}(\lambda,\pi,\tau)=\left (\frac
{(B_{\gamma}^*)^{-1} \cdot \lambda}
{\|(B_{\gamma}^*)^{-1} \cdot \lambda\|},\pi,
\|(B_{\gamma}^*)^{-1} \cdot \lambda\|
(B_{\gamma}^*)^{-1} \cdot \tau \right ).
\ee
and the return time function is given by
\be \nonumber
r_{\widehat \De}(\lambda,\pi,\tau)=r_\De(\lambda,\pi)=
-\ln \|(B_{\gamma}^*)^{-1} \cdot \lambda\|.
\ee
The map $T_{\widehat \De}(\lambda,\pi,\tau)=(\lambda',\pi,\tau')$
is a skew-product over a
non-invertible transformation $T_\De(\lambda,\pi)=(\lambda',\pi)$.

Analogously to the case of the renormalization map, the 
Veech flow with involution can be seen as a suspension over $T_{\widehat \De}$,
with roof function $r_{\widehat \De}$. But considering 
this suspension model, we lose the control of many more orbits which
do not come back to $\widehat \De$.  Still, due to
ergodicity of the Veech flow with involution, almost every orbit is captured by the
suspension model.

\subsection{Hyperbolic properties} \label {hyperbolic proper}

The reason to choose the section $\widehat \De$ is because the 
transformation $T_{\widehat \De}$ has better hyperbolic properties than 
transformations considering larger sections and we can also describe easily 
the connected components of its domain.

\begin{lemma} \label {skew-product}
$T_{\widehat \De}$ is a hyperbolic skew-product over $T_\De$.
\end{lemma}

Recalling the definition \ref{skew-product map}, we observe 
that associated to a hyperbolic skew-product we have: 
a probability measure $\widehat \nu$ (which we chose as the normalized 
restriction of $\widehat m^{(1)}_\RR$ to $\widehat \De$) and a Finsler metric 
$\| \cdot \|_{\widehat \De}$ (which we will choose in order 
to obtain the hyperbolic properties we want from $T_{\widehat \De}$).
At first we will introduce a complete Finsler metric on 
$\widehat \Up^{(1)}_\pi$, and then we will consider its restriction
denoting it by $\| \cdot \|_{\widehat \De}$. Since $\gamma_*$
is strongly positive, the section $\widehat \De$ is a 
precompact open subset of $\widehat \Up^{(1)}_\RR$, 
therefore $\widehat \De$ will have bounded diameter with respect 
to such metric.

\subsection{Hilbert metric}

Now, we will introduce the Hilbert projective metric and 
state some of its pro\-perties which we will use after. This notion 
can be defined for a general convex cone $C$ in any vector space, 
but in our case we only need $C=\R^\BA_+$.

We call the Hilbert pseudo-metric on $\R^2_+$ the function $\dist_{\R^2_+}$ 
defined by 
$$\dist_{\R^2_+}(x,y) =
\ln \max_{1 \leq i,j \leq 2} \frac {x_i y_j} {x_j y_i},
\text{ for each }  
x,y \in \R^2_+.$$ Given a linear operator $B \in \GL(2,\R)$ such 
that $B \cdot \R^2_+ \subset \R^2_+$, we have $\dist_{\R^2_+}
(B \cdot x,B \cdot y) \leq \dist_{\R^2_+}(x,y)$ for all 
$x,y \in \R^2_+$, or equivalently, such that all coefficients 
of the matrix $B$ are non-negative, which means that $B$ contracts 
weakly the Hilbert pseudo-metric. In particular, the Hilbert 
pseudo-metric is invariant under linear isomorphisms of $\R^2_+$.

In general, if we consider an open convex cone 
$C \subset \R^\BA \backslash \{0\}$ whose closure does 
not contain any one-dimensional subspace of $\R^\BA$, 
we define the Hilbert pseudo-metric on $C$ by $\dist_C(x,y)=0$ if 
$x$ and $y$ are collinear and $\dist_C(x,y)=\dist_{\R^2_+}(\psi(x),\psi(y))$ 
where $\psi$ is any isomorphism between the intersection of $\R^2_+$ 
and the subspace generated by $x$ and $y$ (this isomorphism exist 
since $x$ and $y$ are not collinear).
If $C=\R^\BA_+$ then we define $\dist_C(x,y)=\max_{\balpha,\bbeta \in \BA} 
\ln \frac {x_\balpha y_\bbeta}{x_\bbeta y_\balpha}$.

Given two convex cones $C$ and $C'$ such that $C' \subset C$ then
$\dist_C(x,y) \leq \dist_{C'}(x,y)$, i.e., the inclusion map 
$C' \to C$ is a weak contraction of the respective 
Hilbert pseudo-metrics. But if the diameter of $C'$ 
with respect to $\dist_C$ is bounded by some $M$ then 
we have an uniform contraction by some constant $\delta=\delta(M)<1$, 
i.e., $\dist_C(x,y) \leq \delta \dist_{C'}(x,y)$.

It is clear that $\dist_C(x,y)=0$ if and only if there exists 
$t>0$ such that $y=tx$. If we restrict the Hilbert pseudo-metric 
on a convex cone $C$ to the space of rays $\{t x\tq t \in \R_+\} 
\subset C$ we have the Hilbert metric, which is a complete Finsler metric.

\subsection{Uniform expansion and contraction}

Recall that $\widehat \Up^{(1)}_\pi$ is contained in
$\Delta_\pi \times \Ga_\pi$, which is a product of two convex cones.  In
$\Delta_\pi \times \Ga_\pi$, we have the product Hilbert pseudo-metric
$$\dist((\lambda,\pi,\tau),
(\lambda',\pi,\tau'))=\dist_{\Delta_\pi}((\lambda,\pi),(\lambda',\pi))+
\dist_{\Ga_\pi}(\tau,\tau').$$
Each product of rays $\{(a \lambda,\pi,b \tau) \tq a,b \in \R_+\} \subset
\Delta_\pi \times \Ga_\pi$ intersects transversely $\widehat \Up^{(1)}_\RR$ in a
unique point.  It follows that the product Hilbert pseudo-metric induces a
metric $\dist$ on $\widehat \Up^{(1)}_\pi$.  It is a complete Finsler
metric.

\begin{lemma}\label{Lip}
Given $\pi \in \RR$, define 
$\Delta_\pi^{1}=\{ (\lambda, \pi) \in \Delta_\pi \tq \|\lambda\|=1\}$.
Let $g^\pi: \Delta^1_\pi \to \Delta^1_\pi$ be a functional defined 
by $g^\pi(\lambda, \pi)=\left(\sum_{\bbeta}g^\pi_\bbeta \lambda_\bbeta, \pi\right)$, 
where $g^\pi_\bbeta \geq 0$ for all $\bbeta \in \BA$. Then 
$\log g^\pi(\lambda, \pi)$ is $1$-Lipschitz relative to the 
Hilbert metric.
\end{lemma}

\begin{proof}
Given $(\lambda, \pi), (\lambda', \pi) \in \Delta^1_\pi$, we have
\be \nonumber
\frac{g^\pi(\lambda, \pi)}{g^\pi(\lambda', \pi)}=
\frac{\left(\sum_{\bbeta}g^\pi_\bbeta \lambda_\bbeta, \pi \right)}
{\left(\sum_{\bbeta}g^\pi_\bbeta \lambda'_\bbeta, \pi \right)} \leq 
\sup_\bbeta \frac{\left(\lambda_\bbeta, \pi \right)}{\left(\lambda'_\bbeta, \pi \right)} 
\leq e^{\dist_{\Delta_\pi}((\lambda,\pi),(\lambda',\pi))}.
\ee
Thus $\log g^\pi(\lambda, \pi)$ is $1$-Lipschitz with respect to $\dist_{\Delta_\pi}$.
\end{proof}

\smallskip

\noindent{\it {Proof of Lemma \ref {skew-product}}.}
Let us first show that $T_\De$ is a uniformly expanding Markov map (the
underlying Finsler metric being the restriction of
$\dist_{\Delta_\pi}$, and the underlying measure $\Leb$ being the
induced Lebesgue measure). It is clear that $\De$ is a John domain.

Condition (1) of Definition \ref {markovmap}
is easily verified, except for the definite contraction of inverse branches. 
To check this property, we notice that an inverse branch can be written as
$h(\lambda,\pi)=\left (\frac {B_\gamma^* \cdot \lambda} {\|B_\gamma^*
\cdot \lambda\|},\pi \right )$.  Since $\gamma_*$ is neat, we can write
$B^*_\gamma=B^*_{\gamma_*} B^*_{\gamma_0}$ for some $\gamma_0$.  Thus $h$
can be written as (the restriction of) the composition of two maps
$\Delta^1_\pi \to \Delta^1_\pi$, $h=h_* \circ h_0$, where
$h_0$ is weakly contrac\-ting and $h_*$ is uniformly
contracting by precompactness of $\De$ in $\Delta_\pi$
(which is a consequence of strong positivity of $\gamma_*$).

To check condition (2) of Definition \ref {markovmap},
let $h(\lambda,\pi)$ be an inverse branch of $T_\De$.

Let $V=\{ v \in \SS_\pi \tq \sum v_\balpha=0\}$ be the hyperplane tangent to 
$\Delta^1_\pi$ at a point $(\lambda, \pi) \in \Delta^1_\pi$. 
Since the coordinate $\pi$ is fixed by $h$ we can dismiss it.
Thus the simplified expression of $h$ is $h(\lambda)=\frac {B_\gamma^* \cdot \lambda} 
{\|B_\gamma^*\cdot \lambda\|}$. We will denote 
$\phi(B_\gamma^* \cdot \lambda)=\|B_\gamma^*\cdot \lambda\|= 
\sum_{\balpha, \bbeta} (B_\gamma^*)_{\balpha, \bbeta} \lambda_\bbeta$, where 
$(B_\gamma^*)_{\balpha, \bbeta}$ is the coefficient of $B_\gamma^*$ 
in the line $\balpha$ and the column $\bbeta$.
So,
\be \nonumber
Dh(\lambda)\cdot v = \frac{B_\gamma^*\cdot v}{\phi(B_\gamma^* \cdot \lambda)}
-\frac{B_\gamma^*\cdot \lambda}{\phi(B_\gamma^* \cdot \lambda)}
\frac{\sum_{\balpha} \left(B_\gamma^*\cdot v \right)_\balpha}{\phi(B_\gamma^* \cdot \lambda)}.
\ee
So $Dh(\lambda)=P_\lambda \circ \phi(B_\gamma^* \cdot \lambda)^{-1} \circ B_\gamma^*$, where 
$B_\gamma^*: V \to B_\gamma^* \cdot V$, $\phi(B_\gamma^* \cdot \lambda)^{-1}$ is the division 
by the scalar $\phi(B_\gamma^* \cdot \lambda)$ on $B_\gamma^* \cdot V$ and 
$P_\lambda: B_\gamma^* \cdot V \to V$ is the projection on $V$ 
along the direction $B_\gamma^* \cdot \lambda$.
The Jacobian of $h$ at $(\lambda,\pi)$ is $J \circ h(\lambda)= \det Dh(\lambda)$, 
so,
\be \nonumber
\ln J \circ h= \ln \det P_\lambda-(d-2) \ln \det \phi(B_\gamma^* \cdot \lambda) 
+\ln \det B_\gamma^*
\ee
We want to prove that $\ln J \circ h$ is Lipschitz relative to 
the Hilbert metric. We have that $\ln \det B_\gamma^*$ is constant 
and, by lemma \ref{Lip}, $\ln \det \phi(B_\gamma^* \cdot \lambda)$ is 
$1$-Lipschitz. Now we have to verify what happens with 
$\ln \det P_\lambda$. We have that
\be \nonumber
\det P_\lambda = \frac{\langle n_1, B_\gamma^* \cdot \lambda \rangle}
{\langle n_0, B_\gamma^* \cdot \lambda \rangle}
\ee
where $n_0$ and $n_1$ are unit vectors in $\SS_\pi$ orthogonal to the 
hyperplanes $V$ and $B_\gamma^* \cdot V$. Indeed, the vector $n_0$ 
and the vector $B_\gamma \cdot n_1$ are 
collinear with the orthogonal projection of $(1,\ldots,1)$ on $\SS_\pi$. 
Note that $n_0$ has non-negative coefficients, so, 
neither $B_\gamma \cdot n_0$ and $B_\gamma \cdot n_1$ have.
Once again by lemma \ref{Lip}, we have 
that each $\log (\langle n_i, B_\gamma^* \cdot \lambda \rangle)$ is 
$1$-Lipschitz.
Therefore, $\ln J \circ h$ is $d$-Lipschitz with respect to
$\dist_{\Delta_\pi}$. 

To see that $T_{\widehat \De}$ is a hyperbolic skew-product over $T_\De$,
one checks the conditions (1-4) of Definition \ref {skew-product map}.
Condition (1) is
obvious, and condition (4) follows from precompactness of $\widehat \De$ in
$\Delta_\pi \times \Ga_\pi$ as before.
Since $T_{\widehat \De}$ is a first return map,
the restriction of $\widehat m^{(1)}_\RR$
to $\widehat \De$ is $T_{\widehat \De}$-invariant.  Its normalization is the
probability measure $\widehat \nu$ of condition (2).  In order to check condition
(3), it is convenient to trivialize $\widehat \De$
to a product (via the natural diffeomorphism
$\widehat \De \to \De \times \P\Ga_{\gamma_*}$, where $\P V$ denotes 
the projective space of $V$). Since
$\widehat \nu$ has a smooth density with respect to the product of the Lebesgue
measure on the factors, condition (3) follows by the Leibniz rule.
\qed

\smallskip

Our results give the finiteness of the measure and the integrability 
of the cocycle.

\begin{prop}\label{finitemeasure}
The space $\Om^{(1)}_\RR$ has finite volume. 
\end{prop}

\begin{proof}
Consider the section $\widehat \De$. Notice that 
this section has positive measure and almost every orbit return 
to $\widehat \De$. We have already mention the probability 
measure $\widehat \nu$ on $\widehat \De$, 
which is the normalized restriction of $\widehat m^{(1)}_\RR$ to 
$\widehat \De$.

We want to compute $\int_{\widehat{\De}} 
\log \|(B_\gamma^*)^{-1} \cdot x\|\, d\widehat\nu$.

A connected component of the domain of $T_{\widehat \De}$ 
which intersects the set $\{ x \in \widehat \De: 
r_{\widehat \De}(x) > T \}$ is of the form 
$(\Delta_{\gamma_1} \times \Ga_{\gamma_2}) \cap \widehat \Up^{(1)}_\RR$ 
where $\gamma_1$ can not be written as 
$\gamma_s \widehat \gamma \gamma_e$ with 
$\widehat \gamma = \gamma_*\gamma_*\gamma_*\gamma_*$, 
$\M(B_\gamma \cdot q_0) \geq D^{-1}T$, where $q_0=(1,\ldots,1)$ and 
$D=D(\gamma_*)$ is some constant.

The projection of $\widehat \nu|_{(\Delta_{\gamma_1} \times 
\Ga_{\gamma_2})\cap \widehat \Up^{(1)}_\RR}$ on $\Up^{(1)}_\RR$ is 
absolutely continuous with a bounded density, so
\begin{multline*}
\widehat \nu \{ x \in (\Delta_{\gamma_1} \times \Ga_{\gamma_2}) 
\cap \widehat \Up^{(1)}_\RR: r_{\widehat \De}(x)>T \} \leq
\\
C P_{q_0}(\gamma \text{ can not be written as } 
\gamma_s \widehat \gamma \gamma_e \text{ and } 
\M(B_\gamma \cdot q_0) \geq D^{-1} T \di \pi)
\end{multline*}
and the result follows by Proposition \ref{6.2}.
\end{proof}

\subsection{Properties of the roof function}

Recall $H(\pi)=\Omega(\pi) \cdot \mathcal{S}_\pi$. As we have observed, 
given a path $\gamma \in \Pi(\pi)$, $H(\pi)$ is invariant under 
the map $B_\gamma$. By 
Lemma \ref{hpositive}, if $\tau \in \Ga_\pi$ then 
$-\Omega(\pi) \cdot \tau \in \R^\BA_+$ and by Corollary \ref{non-empty}, 
$\Ga_\pi$ is a non-empty set, so $H(\pi) \cap \R^\BA_+ \neq
\emptyset$.

\begin{lemma}\label{Hpi}
Let $\pi$ be an irreducible permutation and $\gamma \in \Pi(\RR)$. 
The subspace $H(\pi)$ has dimension greater than one.
\end{lemma}

\begin{proof}
Let $A$ be a minimal double letter in the sense that $A$ is a left double 
letter and there is no double letter $Z$ such that $\pi(Z)<\pi(A)$ or 
$\pi(i(Z))<\pi(A)$. Let $B$ be a maximal double letter in the sense that $B$ 
is a right double letter and there is no double letter $Z$ such that 
$\pi(B)<\pi(Z)$ or $\pi(B)<\pi(i(Z))$.

We have that:
\be \nonumber
\mathcal{S}_\pi=\left\{\lambda \in \R^{\BA}: 
\lambda_{\underline {A}}=\lambda_{\underline{B}} + \sum \epsilon_\balpha \lambda_\balpha
\right\}
\ee
where $\epsilon_\balpha=0$ if $\alpha$ is simple, $\epsilon_\balpha=-1$ 
if $\alpha$ is left double letter and $\epsilon_\balpha=1$ 
if $\alpha$ is right double letter.

{\bf Case I:}
Let $C$ be a simple letter. Denote $x_C=\Omega(\pi)_{AC} \in \{ 0, 1, 2\}$ 
and $y_C=\Omega(\pi)_{CB} \in \{ 0, 1, 2\}$. Thus the matrix $\Omega(\pi)$ 
has a submatrix of the form:
\be\nonumber
\left( \begin{array}{ccc}
0 & 2 & x_C \\
-2 & 0 & -y_C \\
-x_C & y_C & 0 \end{array} \right).
\ee
So, the matrix $\Omega(\pi) \cdot \mathcal{S}_\pi$ has a submatrix of the form:
\be \nonumber
\left( \begin{array}{cc}
2 & x_C \\
-2 & -y_C \\
-x_C + y_C & 0 \end{array} \right).
\ee
In this case the matrix $\Omega(\pi)\cdot \mathcal{S}_\pi$ has 
rank 2, except if $x_C=y_C$.

{\bf Case II:}
Let $D$ be a left double letter (the other case is analogous). 
Denote $z_D=\Omega(\pi)_{AD} \in \{ -2, -1, 0, 1, 2\}$. Notice that 
$\Omega(\pi)_{DB}=2$. In this case, the matrix $\Omega(\pi)$ 
has a submatrix of the form:
\be \nonumber
\left( \begin{array}{ccc}
0 & 2 & z_D \\
-2 & 0 & -2 \\
-z_D & 2 & 0 \end{array} \right),
\ee
therefore, the matrix $\Omega(\pi) \cdot \mathcal{S}_\pi$ has a submatrix of the form:
\be \nonumber
\left( \begin{array}{cc}
2 & 2+z_D \\
-2 & -2 \\
2-z_D & 2 \end{array} \right).
\ee
In this case the matrix $\Omega(\pi)\cdot \mathcal{S}_\pi$ has 
rank 2, except if $z_D=0$.

Suppose that if $C$ is a simple letter of the permutation $\pi$ then 
$x_C=y_C$ and if $D$ is a double letter of the permutation $\pi$ then 
$z_D=0$. Since $\pi$ is irreducible, $x_C=y_C=1$. 
Thus $\pi$ has the form
\be \nonumber
\begin{array}{cccccccccccc} A & \cdot &
 \cdot & \cdot & \cdot & i(A) & * & i(B) & \cdot & \cdot & \cdot & B
\end{array}
\ee
If we apply the right or the left operation we obtain a reducible 
permutation.

So, there exists a simple letter $C$ such that 
$x_C \neq y_C$ or there exists a double letter $D$ such that 
$z_D \neq 0$.
\end{proof}

Recall that $v_\pi=\sum_{\pi(x)<\pi(*)} e_{\bx}-\sum_{\pi(x)>\pi(*)} e_{\bx}$ 
is the orthogonal vector to $\SS_\pi$.

\begin{lemma}\label{dimH}
Let $\pi$ be an irreducible permutation and $\gamma \in \Pi(\RR)$. 
If $v_\pi \in H(\pi)$, then the subspace $H(\pi)$ has dimension greater 
than two.
\end{lemma}

\begin{proof}
Let $A$ and $B$ be the leftmost and the rightmost letters of $\pi$.

Since $H(\pi)$ and $v_\pi$ are invariant under $B_\gamma$, 
we can suppose that $A$ is simple.

If $B$ is simple, then $\pi$ has the form:
\be \nonumber
\begin{array}{cccccccccccc} A & \cdot &
 \cdot & i(B) & \cdot & \cdot & * & \cdot & i(A) & \cdot & \cdot & B
\end{array}
\ee
We have that $\Omega(\pi) \cdot e_A$, $\Omega(\pi) \cdot e_B$ and $v_\pi$ are 
linearly independent, since $(v_\pi)_A=(v_\pi)_B=0$.

If $\pi(i(B))<\pi(i(A))$ and $B$ is double, we take a left double letter $C$, 
i.e., $\pi$ has the form:
\be \nonumber
\begin{array}{ccccccccccccccc} A & \cdot & 
C & \cdot & \cdot & i(C) & \cdot & * & \cdot & i(B) & \cdot & i(A) & \cdot & \cdot & B
\end{array}
\ee
And in the case that $\pi(i(B))>\pi(i(A))$, $\pi$ has the form:
\be \nonumber
\begin{array}{ccccccccccccccc} A & \cdot &
C & \cdot & \cdot & i(C)& \cdot & * & \cdot & \cdot & i(A) &\cdot & i(B) & \cdot & B
\end{array}
\ee
In these last two cases, we have $\Omega(\pi) \cdot e_A$, 
$\Omega(\pi) \cdot (e_B+e_C)$ and $v_\pi$ are 
linearly independent, since $(v_\pi)_A=0$ and $(v_\pi)_B=-(v_\pi)_C=1$.
\end{proof}

\begin{lemma} \label {bla}
The roof function $r_\De$
is good (in the sense of Definition \ref {def:goodroof}).
\end{lemma}

\begin{proof}
Let $\Gamma \subset \Pi(\RR)$ be the set of all $\gamma$ such that 
$\gamma$ is either $\gamma_*$, or a minimal path of the 
form $\gamma_* \gamma_0 \gamma_*$ not beginning by $\gamma_*\gamma_*$. 
Notice that $\Gamma$ consists of positive paths. 

The set $\HH$ of inverse branches $h$ of $T_\De$ is in bijection
with $\Gamma$, since each inverse branch is of the form
$h(\lambda,\pi)=\left( \frac{B_{\gamma_h}^* \cdot 
\lambda}{\|B_{\gamma_h}^* \cdot \lambda\|},\pi \right)$ for some
$\gamma_h \in \Gamma$.

Let $h \in \HH$. Then $r_\De(h(\lambda,\pi))=
\ln \|B_{\gamma_h}^* \cdot \lambda\|$.  Since $\gamma_h$ is positive,
$r_\De \geq \ln 2$, which implies condition (1). 
By lemma \ref{Lip}, $r_\De(h(\lambda,\pi))$ is $1$-Lipschitz 
with respect to $\dist_{\Delta_\pi}$, so (2) follows.

Let us check condition (3).  We identify the tangent space to $\De$ at a
point $(\lambda,\pi) \in \De$ with $V=\{\lambda \in \mathcal{S}_\pi \tq \sum
\lambda_\balpha=0\}$.  Assume that we can write $r_\De=
\psi+\phi \circ T_\De-\phi$ with $\phi$ $C^1$ and $\psi$ locally constant. 
Write $r^{(n)}(\lambda,\pi)=\sum_{j=0}^{n-1} r_\De(T_\De^j(\lambda,\pi))$.
Then $D(r^{(n)} \circ h^n)=D \phi-D(\phi \circ
h^n)$, which can be rewritten as
\be\label{vwn}
\frac {\|(B^*_{\gamma_h})^n \cdot v\|} {\|(B^*_{\gamma_h})^n \cdot
\lambda\|}=D\phi(\lambda,\pi) \cdot v-D(\phi \circ h^n)(\lambda,\pi)
\cdot v, \quad (\lambda,\pi) \in \De,\, v \in V.
\ee
If we define
\be \nonumber
w_{n,h}=\frac {B^n_{\gamma_h} \cdot (1,\ldots,1)}
{\langle \lambda,B^n_{\gamma_h} \cdot (1,\ldots,1) \rangle},
\ee
we replace (\ref{vwn}) by
\be \nonumber 
\langle v, w_{n,h} \rangle=
D\phi(\lambda,\pi) \cdot v-D(\phi \circ h^n)(\lambda,\pi) \cdot v, \quad
(\lambda,\pi) \in \De,\, v \in V.
\ee
We have that $\langle \lambda, w_{n,h} \rangle=1$ for all 
$\lambda \in \SS_\pi$ and $w_{n,h}$ are vectors with all coordinates positive.
By the Perron-Frobenius Theorem $w_{n,h}$ 
converges to some $w_h$ collinear with the unique positive
eigenvector of $B_{\gamma_h}$ (which also corresponds to the largest
eigenvalue). And $w_h=w_{0,h} + t_h v_\pi$, where 
$w_{0,h} \in \SS_\pi \backslash \{ 0 \}$ is the orthogonal projection of 
$w_h$ in $\SS_\pi$ and $v_\pi$ is the 
orthogonal vector of $\SS_\pi$, which is 
invariant under $B_{\gamma_h}$ for all $\gamma_h$. 

Since $Dh^n \to 0$, we conclude that
\be \nonumber
\langle v, w_{h} \rangle =
D\phi(\lambda,\pi) \cdot v, \quad (\lambda,\pi) \in \De,\, v \in V.
\ee
Since $\langle \lambda, w_{0,h} \rangle=1$, we have $w_{0,h}=w_0$.
Thus $w_h=w_0 + t_h v_\pi$ where $w_0 \in \SS_\pi \backslash \{ 0\}$.

Recalling that $H(\pi)$ is invariant under $B_{\gamma_h}$, 
and intersects $\R^\BA_+$, it follows that 
$w_h \in H(\pi)$ and $\R_+ w_{n,h}$ 
is converging to $\R_- \left(\Omega (\pi)\cdot \tau\right)$.

Let $W=\R w_0 \oplus \R v_\pi$. We have that $W \cap H(\pi) \neq \emptyset$ 
is closed and invariant by $B_{\gamma_h}$. By the previous two lemmas,
there exists $\tau \in \Ga_{\pi}$ such that 
$\Omega(\pi)\cdot \tau \notin W$. 
But, given such $\tau$, we can construct paths $\gamma_n$ such that 
$B_{\gamma_n}\cdot \Omega(\pi) \cdot \Ga_\pi$  converges to 
$\R_- \left(\Omega (\pi) \cdot \tau\right)$ as follows.
We have already observed that $\widehat Q^{-1}$ is recurrent, so 
given $(\lambda, \pi, \tau) \in \widehat \De$, we apply $\widehat Q^{-1}$ 
until obtain $(\lambda^{(-1)}, \pi, \tau^{(-1)}) \in \widehat \De$. 
We denote by $\gamma_1$ the path obtained previously, 
starting at $(\lambda^{(-1)}, \pi, \tau^{(-1)})$ and ending at 
$(\lambda, \pi, \tau)$. We follow the same procedure to obtain 
$\gamma_n$ starting at $(\lambda^{(-n)}, \pi, \tau^{(-n)})$ and ending at 
$(\lambda, \pi, \tau)$.

By definition of $\widehat \De$, we have that such paths $\gamma_n$ 
are strongly positive, so the image of $B_{\gamma_n}^* \cdot \Ga_{\gamma_n}$ 
is contracted, relatively to the Hilbert metric. Thus we have that 
$B_{\gamma_n}\cdot \Omega(\pi) \cdot \Ga_\pi$ 
is converging to $\R_- \left(\Omega(\pi) \cdot \tau\right)$.
Thus we have a contradiction.
\end{proof}

\begin{thm} \label {exponentialtails}
The roof function $r_\De$ has exponential tails.
\end{thm}

\begin{proof}
Let $\pi$ be the start of $\gamma_*$.  
The push-forward under radial projection
of the Lebesgue measure on $\e_{\pi,q_0}$ onto $\Delta_\pi \cap \Up^{(1)}_\RR$
yields a smooth measure $\tilde \nu$.  It is enough to show that $\tilde
\nu\{x \in \De \tq r_\De(x) \geq T\} \leq C T^{-\delta}$, for some $C>0$,
$\delta>0$.
A connected component of the domain of $T_\De$ that intersects the set
$\{x \in \De \tq r_\De(x) \geq T\}$ is of the form
$\Delta_\gamma \cap \Up^{(1)}_\RR$
where $\gamma$ can not be written as $\gamma_s \hat \gamma \gamma_e$ with
$\hat \gamma=\gamma_* \gamma_* \gamma_* \gamma_*$ and $\M(B_\gamma \cdot
q_0) \geq C^{-1} T$, where $q_0=(1,\ldots,1)$ and
$C$ is a constant depending on $\gamma_*$.  Thus
\be \nonumber
\tilde \nu \{x \in \De \tq r_\De(x) \geq T\} 
\leq P_{q_0}(\gamma \text { can
not be written as } \gamma_s \hat \gamma \gamma_e \text { and }
\M(B_\gamma \cdot q_0) \geq C^{-1} T \di \pi).
\ee
The result follows from Proposition \ref {6.2}.
\end{proof}

Using both the map $T_{\widehat \De}$ and the roof function
$r_\De$ we will define a flow $\widehat T_t$ on the space $\wde=
\{(x,y,s)\,:\, (x,y) \in \widehat \De,\,T_{\widehat\De}(x,y)\text{ is defined
and } 0 \leq s <r_\De(x)\}$.
Since $T_{\widehat \De}$ is
a hyperbolic skew-product (Lemma \ref {skew-product}),
and $r_\De$ is a good roof
function (Lemma \ref {bla}) with exponential tails (Theorem \ref
{exponentialtails}), $\widehat
T_t$ is an excellent hyperbolic semi-flow.  By Theorem \ref
{main_thm_hyperbolic}, we get
exponential decay of correlations
\be \nonumber
C_t(\tilde f,\tilde g)=\int \tilde f \cdot \tilde g \circ \widehat T_t
\dd\nu-\int \tilde f \dd\nu \int
\tilde g \dd\nu,
\ee
for $C^1$ functions
$\tilde f$, $\tilde g$, that is
\be \label {tilde f}
|C_t(\tilde f,\tilde g)| \leq C e^{-3 \delta t} \|\tilde f\|_{C^1}
\|\tilde g\|_{C^1},
\ee
for some $C>0$, $\delta>0$.

\section{The Teichm\"{u}ller flow}
\label{surfaces}

\subsection{Half-translation surfaces}

Let $S$ be a compact oriented surface of genus $g \geq 0$, let 
$\Sigma$ be a finite non-empty subset of $S$,
which we call the \emph{singular set}. Let 
$l=\{l_x\}_{x \in \Sigma}$ (the \emph{multiplicity vector})
be such that $l_x \in \{-1\} \cup \N$ and 
$\sum l_x = 4g-4$. We say that $l_x$ is the multiplicity of the 
singular point $x$.
Consider a maximal atlas $\bA=\{(\UU_\lambda,\phi_\lambda:\UU_\lambda \to
\VV_\lambda \subset \C)\}$ of orientation preserving charts
on $S\backslash\Sigma$ such that for all $\lambda_1, \lambda_2$ with 
$\UU_{\lambda_1} \cap \UU_{\lambda_2} \neq \emptyset$ 
we have $\phi_{\lambda_1}\phi_{\lambda_2}^{-1}(z)=\pm z + constant$, 
i.e., coordinate changes are compositions of rotations by 
$180^{\circ}$ and translations.  We call these coordinates the 
\emph{regular charts}. We also assume that each 
singular point $x$ has an open neighborhood $\UU$ which is
isomorphic to the $\frac{l_x+2}{2}$-folded cover of an open neighborhood
$\VV \subset \C$ of $0$, that is, there exists a homeomorphism, called
a \emph{singular chart},
$\phi:\UU \to \VV$ such that any branch of
$z \mapsto \phi(z)^{(l_x+2)/2}$ is a regular chart.
Under these conditions, we say that the atlas $\bA$ defines a
\emph{half-translation structure} on $(S,\Sigma)$ with multiplicity vector $l$, and
we call $S$ a \emph{half-translation surface}.

Since the change of coordinates preserves families of parallel 
lines in the plane, we have a well-defined singular foliation $\FF_\theta$
of $S$, for each direction $\theta \in \P\R^2$ (the projective 
space of $\R^2$).
In particular, we have well-defined
vertical and horizontal directions.
Notice that we can pullback the Euclidean metric in $\R^2$ by the
regular charts to define a flat metric on $S \backslash \Sigma$.  This flat
metric does not extend smoothly to $\Sigma$ except at points with $l_x=0$. 
The other points of $\Sigma$ are genuine \emph{conical singularities} with
total angle $\pi(l_x+2)$, and are thus responsible for any curvature.  The
corresponding volume form on  $S \backslash \Sigma$ has finite total mass.

Notice that
from each $x \in \Sigma$, there are $l_x+2$ \emph{horizontal separatrices}
alternating with $l_x+2$ \emph{vertical separatrices} emanating from $x$.
A half-translation surface together with the choice of some $x_0 \in \Sigma$
and of one of the horizontal separatrices $X$ emanating from $x_0$ is called a
\emph{marked half-translation surface}.

\subsection{Translation surfaces}

If there exists a compatible atlas
such that the coordinate changes are just translations, then any maximal
such atlas is said to define a \emph{translation structure} on $(S,\Sigma)$
compatible with the half-translation structure, and we call $S$ a 
\emph{translation surface}.
A half-translation surface has thus either $0$ or $2$ compatible
translation structures. Locally, each half-translation 
structure is compatible with a translation structure, but in general 
it is not true globally.
Given a half-translation surface $S$, we can associate a number
$\varepsilon$ where $\varepsilon=1$ or $\varepsilon=-1$ according to whether
the half-translation structure is, or is not, compatible with a translation
structure (on the other hand, obviously each translation structure is
compatible with a unique half-translation structure).
Notice that if $\varepsilon=1$ then $l_x \in 2 \N$ for every $x
\in \Sigma$ (and thus necessarily $g \geq 1$),
but the converse is not generally true.

Given a translation surface $S$, each oriented direction
$\theta \in \mathbb{S}^1$ determines a singular oriented foliation 
$\FF_\theta$ on $S$.  From every singularity thus emanate
$(l_x+2)/2$ eastbound (respectively, northbound, westbound, southbound)
oriented separatrices.
A translation surface together with the choice of some $x_0 \in \Sigma$ and of
a eastbound separatrix $X$ emanating from $x_0$ is called a \emph{marked translation
surface}.

\subsection{Translation surfaces with involution}\label{tsinv}

Let $\tilde S$ be a compact oriented surface of genus $g \geq 1$, let
$\tilde \Sigma$ be a finite non-empty subset of $\tilde S$, and let $I:\tilde S \to
\tilde S$ be an involution preserving $\tilde \Sigma$ and whose fixed points
are contained in $\tilde \Sigma$.
A \emph{translation structure with involution}
on $(\tilde S,\tilde \Sigma,I)$ is a translation structure such that for
every regular chart $\phi$ of the translation structure,
$-\phi \circ I$ is also a regular chart.

Notice that given $(\tilde S,\tilde \Sigma,I)$ we can consider the canonical
projection $p:\tilde S \to S=\tilde S/I$.  Denote $\Sigma=\tilde
\Sigma/I$. We see that any translation structure with involution on
$(S,\Sigma,I)$ induces by $p$ a
half-translation structure on $(S,\Sigma)$, with $\varepsilon=-1$ 
if $\tilde S$ is connected.

Conversely, given $(S,\Sigma)$ and a multiplicity vector $l$ such that there
exists a half-translation structure on $(S,\Sigma)$ with such multiplicity
vector and $\varepsilon=-1$, there exists a ramified double covering
$p:(\tilde S,\tilde \Sigma) \to (S,\Sigma)$ which is unramified in $\tilde S
\backslash \tilde \Sigma$. Indeed, given such a
half-translation structure, we can define $\tilde S \backslash \tilde \Sigma$
to be the set of pairs $(z,\alpha)$ where $z \in S \backslash \Sigma$
and $\alpha$ is an orientation of the horizontal direction through $z$ (the
assumption that $\varepsilon=-1$ guarantees that $\tilde S$ is connected). 
It is then easy to define the missing set $\tilde
\Sigma$ necessary to compactify: each $x \in \Sigma$ with odd $l_x$ giving
rise to a single point of $\tilde \Sigma$ with multiplicity $2l_x+2$ 
and each $x \in \Sigma$ with even $l_x$ giving rise to a pair 
points of $\tilde \Sigma$ with multiplicity $l_x$ each one.
To each half-translation surface we can associate a combinatorial data
$\tilde l$, which is the multiplicity vector considered up to labelling.
The construction above gives rise to a translation surface $\tilde S$ 
with singularity set $\tilde \Sigma$ and there is a natural involution 
defined, interchanging points $(z, \alpha)$ with fixed $z \in \tilde S$.

A translation surface with involution together with the choice of 
some $\tilde x_0 \in \tilde \Sigma$ and of one of the horizontal separatrices 
$\tilde X$ emanating from $\tilde x_0$ to east is called a \emph{marked 
translation surface with involution}. We say that $\tilde x_0$ is the start 
point of $\tilde X$. It is obvious that fixing $\tilde x_0$ 
and $\tilde X$ we also fix $I(\tilde x_0)$ and $I(\tilde X)$.

Notice that when we do the double covering construction above 
we can do it in such a way that $p(\tilde X)=X$ and $p(\tilde x_0)=x_0$, 
i.e., the marked separatrix and its start point are 
preserved.

As we will see in section \ref{coordinates}, we can obtain combinatorial 
and length data $(\lambda, \pi, \tau)$ (as in the section 
\ref{Rauzyclasses}) associated to a marked translation surface with involution 
$(\tilde S, \tilde \Sigma, I)$.

\subsection{Moduli spaces}

Let $S$ be a surface with singular set $\Sigma$ and genus $g$. 
To consider the space of surfaces with fixed genus, 
singularity set, multiplicity vector and the marked sepa\-ratrix, 
we can define equivalence relations on those surfaces, obtaining
moduli spaces. Although moduli spaces are not manifolds, 
we can see them as a quotient of a less restricted 
space, which has a complex affine manifold structure, and the 
modular group of $(S,\Sigma)$, i.e., the group of orientation 
preserving diffeomorphisms of $S$ fixing $\Sigma$ modulo those 
isotopic to the identity. Thus, moduli spaces are complex 
affine orbifolds.

\subsubsection{Moduli space of marked translation surfaces}
Given $g \geq 1$, a function $\kappa:\N \to 2\N$ with finite
support and $\sum_{i \geq 0} i \kappa(i)=4g-4$, 
and an integer $j \geq 0$ with $\kappa(j) \geq 0$,
we let $\MH(g,\kappa,j)$ to be the \emph{moduli space of marked 
translation surfaces} $(S,\Sigma,x_0,X)$ with genus
$g$, $\#\{x \in \Sigma:\, l_x=i\}=\kappa(i)$ and $l_{x_0}=j$.  
Thus two surfaces $(S,\Sigma,x_0,X)$ and $(S',\Sigma',x'_0,X')$ are
equivalent if there exists a homeomorphism $\phi:(S,\Sigma,x_0,X) \to
(S',\Sigma',x'_0,X')$ preserving the translation structure, 
the marked point and the given preferred separatrix.

An alternative way to view $\MH(g, \kappa, j)$ is as follows. 
Given a fixed surface $S$, with finite singular set $\Sigma$, a
multiplicity vector $l$ satisfying $\sum l_i =2g-2$, a fixed 
point $x_0 \in \Sigma$ and some horizontal separatrix $X$ starting 
from $x_0$ going to east, consider the space $\TH(S,\Sigma,x_0, X)$ 
of all marked translation surfaces modulo the following equivalence 
relation: two surfaces $(S,\Sigma,x_0, X)$ and $(S',\Sigma',x'_0, X')$ 
are equivalent if there exists a homeomorphism $\phi: (S,\Sigma,x_0, X) 
\to (S',\Sigma',x'_0, X')$ isotopic to the identity relatively to 
$\Sigma$, which preserves the translation structure. The space 
$\MH(g, \kappa, j)$ is recovered in this way by taking the quotient 
by an appropriate modular group, i.e., the group of orientation 
preserving diffeomorphisms of $S$, fixing $\Sigma$ modulo those 
isotopic to the identity. The advantage of seeing the moduli 
space as a quotient like this, is that it inherits a 
structure of complex affine orbifold, since charts in 
$\TH(S, \Sigma, x_0, X)$ are complex affine.
Indeed, given a path 
$\gamma \in C^\circ([0,T],S)$, we can lift it in $\C$. 
Since we have the translation structure, we can do this 
lifting everywhere. Thus, we can obtain a linear map 
$H_1(S,\Sigma;\Z) \to \C$, which we can see as an element 
of the relative cohomology group $H^1(S,\Sigma;\C)$.
This map is a local homeomorphism, thus it is a local 
coordinate chart. So $\TH(S,\Sigma,x_0,X)$ has 
a complex affine manifold structure.

The Lebesgue measure on space $H^1(S,\Sigma;\C)$ 
(normalized so that the integer lattice 
$H^1(S,\Sigma;\Z) \oplus i H^1(S,\Sigma;\Z)$ has covolume one) 
can be pulled back via these local coordinates, and 
we obtain a smooth measure on the space $\TH(S,\Sigma,x_0,X)$. 
In charts, the modular group acts (discretely and properly 
discontinuously) by complex affine maps preserving the 
integer lattice (and hence the Lebesgue measure). This 
exhibits $\MH(g,\kappa,j)$ as a complex affine orbifold, 
with a canonical absolutely continuous measure $\nu_{\MH}$.
Denoting $\MH^{(1)}$ the moduli space of marked translation 
surfaces with area one, we obtain the respective induced measure 
$\nu_{\MH}^{(1)}$.

The moduli spaces $\MH$ are also called strata and they 
can be disconnected. Kontsevich and Zorich (\cite{KZ03}) 
classified these connected components and they proved 
that they are at most three, for each strata.

\subsubsection{Moduli space of marked translation surfaces with involution}
\label{MHI}

Given $g \geq 1$, functions $\tilde\kappa, \eta:\N \to \N$ with finite
support and $\sum_{i \geq 0} i \tilde\kappa(i)=4\tilde g-4$, and 
an integer $\tilde j \geq 0$ with $\tilde\kappa(\tilde j) \geq 0$,
we let $\MHI(\tilde g,\tilde\kappa,\eta,\tilde j)$ to be the 
\emph{moduli space of marked translation surfaces with involution} 
$(\tilde S,\tilde\Sigma,I,\tilde x_0,\tilde X)$ with genus
$\tilde g$, an involution $I: \tilde S \to \tilde S$ preserving 
$\tilde\Sigma$ and whose fixed points are contained in $\tilde\Sigma$, 
$\#\{x \in \tilde\Sigma:\, l_x=i\}=\tilde\kappa(i)$, 
$l_{x_0}=\tilde j$, and $\#\{x \in \tilde\Sigma:\, l_x=2i \; \text{and} \; 
I(x)=x\}=\eta(2i)$.
Thus two surfaces $(\tilde S,\tilde \Sigma,I,\tilde x_0,\tilde X)$ 
and $(\tilde S',\tilde\Sigma',I',\tilde x'_0,\tilde X')$ are
equivalent if there exists a homeomorphism 
$\phi:(\tilde S,\tilde\Sigma,I,\tilde x_0,\tilde X) \to
(\tilde S',\tilde \Sigma',I',\tilde x'_0,\tilde X')$ preserving the 
translation structure and preserving the involution, in the sense 
that $\phi \circ I=I' \circ \phi$. The marked point and the 
chosen separatrix are also preserved.

Analogous to the previous case, we will consider the moduli 
space of marked translation surfaces with involution 
$\MHI(\tilde g, \tilde \kappa, \eta, \tilde j)$ as a 
larger space, which has an affine complex manifold 
structure, quotiented by a modular group.
Consider a fixed translation surface $\tilde S$, 
an associated involution $I:\tilde S \to \tilde S$, 
a finite singular set $\tilde \Sigma$ invariant by 
$I$, with a multiplicity vector $\tilde l$ satisfying 
$\sum l_i=4\tilde g-4$, together with some fixed 
$\tilde x_0 \in \tilde \Sigma$ and one fixed horizontal 
separatrix $\tilde X$ emanating from $\tilde x_0$. Let 
$\THI(\tilde S, \tilde \Sigma, I, \tilde x_0, \tilde X)$ be 
the set of $(\tilde S, \tilde \Sigma, I, \tilde x_0, \tilde X)$ 
modulo homeomorphism $\phi$ isotopic to the identity relatively to 
$\Sigma$, which preserves the translation structure 
with involution, in particular $\phi \circ I = I \circ \phi$. 

Let $I:\tilde{S}\rightarrow \tilde{S}$ be the involution as
defined before. Consider the induced involution 
$I^*:H^1(\tilde{S},\tilde{\Sigma};\C) \rightarrow 
H^1(\tilde{S},\tilde{\Sigma};\C)$ 
on the relative cohomology group. We can decompose 
the cohomo\-logy group into a direct sum 
$H^1(\tilde{S},\tilde{\Sigma};\C)=H_+^1(\tilde{S},\tilde{\Sigma};\C) 
\oplus H_-^1(\tilde{S},\tilde{\Sigma};\C)$, 
where $H_+^1(\tilde{S},\tilde{\Sigma};\C)$ and 
$H_-^1(\tilde{S},\tilde{\Sigma};\C)$ are, respectively, the invariant and 
the anti-invariant subspaces of $I^*$. Observe that, since the 
involution changes the orientation of regular charts, the 
element of $H^1(\tilde{S},\tilde{\Sigma};\C)$ which represents 
$\tilde{S}$ is in $H_-^1(\tilde{S},\tilde{\Sigma};\C)$ and a small neighborhood of it 
gives a local coordinate chart of a neighborhood of $\tilde S$ in
$\THI(\tilde S, \tilde \Sigma, I, \tilde x_0, \tilde X)$.
Notice that we are considering that the translation surface with involution $\tilde S$ 
can have some regular points in $\tilde \Sigma$. But if we consider the set 
$\hat \Sigma \subset \tilde \Sigma$ such that $\hat \Sigma$ has no 
regular points, we have that the canonical homomorphism 
$H_-^1(\tilde{S},\tilde{\Sigma};\C) \to H_-^1(\tilde{S},\hat{\Sigma};\C)$ 
induced by the inclusion $\hat{\Sigma} \hookrightarrow \tilde{\Sigma}$ is 
an isomorphism. So we can choose $\hat{\Sigma}$ or $\tilde{\Sigma}$ 
to define the coordinate charts (see \cite{MZ}). 
Since the modular group acts discretely and properly discontinuously, 
we obtain a complex affine structure of orbifold to 
$\MHI(\tilde g, \tilde \kappa, \eta, \tilde j)$.
The space $H_-^1(\tilde{S},\tilde{\Sigma};\C)$ 
has a smooth standard measure which we can 
transport to $\THI(\tilde S, \tilde \Sigma, I, \tilde x_0, \tilde X)$ 
obtaining a smooth measure in this space. Hence, the space 
$\MHI$ inherits a smooth measure $\mu_{\MHI}$ and the moduli space 
of surfaces with area one $\MHI^{(1)}$ inherits the induced measure 
$\mu_{\MHI}^{(1)}$.


\subsubsection{Moduli space of marked half-translation surfaces}
Given $g \geq 0$, a function $\kappa:\N \cup \{-1\} \to \N$ with finite
support and $\sum_{i \geq -1} i \kappa(i)=4g-4$,
$\varepsilon \in \{-1,1\}$, and an integer $j \geq -1$ with $\kappa(j)>0$,
we let $\MHQ(g,\kappa,\varepsilon,j)$ to be the
\emph{moduli space of marked half-translation surfaces} $(S,\Sigma,x_0,X)$ with genus
$g$, $\#\{x \in \Sigma:\, l_x=i\}=\kappa(i)$ and $l_{x_0}=j$.  Two surfaces
$(S,\Sigma,x_0,X)$ and $(S',\Sigma',x'_0,X')$ are
equivalent if there exists a homeomorphism $\phi:(S,\Sigma,x_0,X) \to
(S',\Sigma',x'_0,X')$ preserving the half-translation structure, 
the marked point and the fixed separatrix.

If $\varepsilon=1$, the half-translation structure is compatible 
with two translation structures (correspon\-ding to both possible 
orientations) and there exists a natural map 
$\MH(g,\kappa,j) \to \MHQ(g,\kappa,1,j)$ which forgets the 
polarization. This map is a ramified double cover 
of orbifolds. 

Now given a half-translation structure which is not compatible 
with a translation structure, we will associate a translation 
structure using the (ramified) double covering construction.
Define $\tilde \kappa:\N \cup \{-1\} \to \N$
by $\tilde \kappa(2 i-1)=0$, $\tilde \kappa(4i)=2 \kappa(4i)+\kappa(2i-1)$,
$\tilde \kappa(4i+2)=2 \kappa(4i+2)$. 
Let $\tilde g=4+\sum_{i \geq -1} i \tilde \kappa(i)=2g-1+\frac {1} {2}
\sum_{i \geq 0} \kappa(2i-1)$, $\tilde j=j$ if $j$ is even or $\tilde j=2
j+2$ if $j$ is odd.  Thus, we obtain a canonical injective map
$\MHQ(g,\kappa,-1,j) \to \MH(\tilde g,\tilde \kappa,\tilde j)$. 
In fact, by construction, the image of this map is 
$\MHI(\tilde g,\tilde \kappa,\eta,\tilde j)$, 
where the map $\eta: \N \cup \{-1\} \to \N$ is such that 
$\eta(2i)=\#\{x \in \tilde\Sigma:\, l_x=2i \; \text{and} \; I(x)=x\}$ 
and $\eta(2i+1)=0$.

We also can define the quotient map 
$\MHI(\tilde g, \tilde \kappa, \eta, \tilde j) \to 
\MHQ(g, \kappa, -1, j)$, such that, to each structure 
$(\tilde S, \tilde \Sigma, I, \tilde x_0, \tilde X)$ 
associates the quotient structure $(S , \Sigma , x_0 , X ) =
(\tilde S / I, \tilde \Sigma / I, \tilde x_0 / I, \tilde X / I)$.
Notice that this map is well-defined and it is injective, since 
$\tilde S$ is connected. 
Thus, we have a bijection between marked half-translation 
surfaces which are not translation 
surfaces and connected marked translation surfaces with involution.

As in the case of translation surfaces, the moduli space 
of marked half-translation surfaces is called strata. 
Lanneau classified the connected components of each strata, 
which are at most two (\cite{La04}, \cite{La08}).

\subsection{Teichm\"uller Flow}
The group $\SL(2,\R)$ acts on 
$\MHI$ (or more generally, on the space of marked translation 
surfaces with involution) by postcomposition in the charts. This action 
preserves the hypersurface $\MHI^{(1)}$ and measures 
$\mu_{\MHI}$ on $\MHI$ and $\mu_{\MHI}^{(1)}$ on $\MHI^{(1)}$. 

The \emph{Teichm\"uller Flow} is the particular action of the diagonal subgroup 
$\TF_t:= \left(\begin{array}{cc} e^{t} & 0 \\0& e^{-t}\end{array}\right)$ 
and it is measure-preserving.

\begin{thm}[Masur, Veech]
\label{thm_masur_veech} The Teichm\"{u}ller flow is mixing 
on each connected component of each stratum of the moduli 
space $\MHQ^{(1)}$, with respect to the finite 
equivalent Lebesgue measure, $\mu^{(1)}$.
\end{thm}

The Theorem \ref{t.main}, in the setting of translation 
surfaces was proved by Avila, Gou\"ezel and Yoccoz \cite{AGY}. 
So, we will restrict the proof just to the case of half-translation 
surfaces which are not translation surfaces. Thus, we can prove 
it, just considering marked translation surfaces with involution.
In this setting, the Theorem \ref{t.main} is equivalent to:

\begin{thm}
The Teichm\"uller flow is exponential mixing on each connected 
component of the moduli space $\MHI^{(1)}$ with respect to the measure 
$\mu_{\MHI}^{(1)}$ for observables in the Ratner class.
\end{thm}


\section{From the model to the Teichm\"uller flow}
\label{zr}

\subsection{Zippered rectangles construction}\label{ZR}

Consider an irreducible permutation $\pi \in\RR$ and length 
data $\lambda \in \SS_\pi$, $\tau \in \Ga_\pi$ 
and $h \in \R^\BA_+$ defined by $h=-\Omega(\pi) \cdot \tau$. 
Notice that $h=(h_\balpha)_{\balpha \in \BA}$ is such that 
$h_\balpha >0$ for all $\balpha \in \BA$.
Let $\alpha(l)$ and $\alpha(r)$ be the leftmost and 
the rightmost letters of $\pi$, i.e., 
$\pi(\alpha(l))=1=\overline{\pi}(i(\alpha(r)))$ and 
$\pi(\alpha(r))=2d+1=\overline{\pi}(i(\alpha(l)))$.

Define the sets:
\be \nonumber
\BB_\pi(\alpha)=
\left \{ \begin{array}{lll}
\{ \beta \in \AA : \, \pi(\alpha)<\pi(\beta)<\pi(*)\} & \text{if} 
& 1\leq \pi(\alpha)< \pi(*)
,\\[5pt]
\{ \beta \in \AA : \, \pi(*)<\pi(\beta)<\pi(\alpha)\} & \text{if} 
& \pi(*)< \pi(\alpha)\leq 2d+1
\end{array}
\right.
\ee

\be \nonumber
\BB'_\pi(\alpha)=
\left \{ \begin{array}{lll}
\{ \beta \in \AA : \, \pi(\alpha)\leq\pi(\beta) < \pi(*)\} & \text{if} 
& 1\leq \pi(\alpha)< \pi(*)
,\\[5pt]
\{ \beta \in \AA : \, \pi(*)<\pi(\beta) \leq \pi(\alpha)\} & \text{if} 
& \pi(*)< \pi(\alpha)\leq 2d+1
\end{array}
\right.
\ee

\be \nonumber
\BB_{\overline{\pi}}(\alpha)=
\left \{ \begin{array}{lll}
\{ \beta \in \AA : \, \overline{\pi}(\alpha)<
\overline{\pi}(\beta)<\overline{\pi}(*)\} & \text{if} 
& 1\leq \overline{\pi}(\alpha)< \overline{\pi}(*)
,\\[5pt]
\{ \beta \in \AA : \, \overline{\pi}(*)<
\overline{\pi}(\beta)<\overline{\pi}(\alpha)\} & \text{if} 
& \overline{\pi}(*)< \overline{\pi}(\alpha)\leq 2d+1
\end{array}
\right.
\ee

\be \nonumber
\BB'_{\overline{\pi}}(\alpha)=
\left \{ \begin{array}{lll}
\{ \beta \in \AA : \, \overline{\pi}(\alpha)\leq\overline{\pi}(\beta) 
< \overline{\pi}(*)\} & \text{if} 
& 1\leq \overline{\pi}(\alpha)< \overline{\pi}(*)
,\\[5pt]
\{ \beta \in \AA : \, \overline{\pi}(*)<\pi(\beta) 
\leq \overline{\pi}(\alpha)\} & \text{if} 
& \overline{\pi}(*)< \overline{\pi}(\alpha)\leq 2d+1
\end{array}
\right.
\ee

For each $\alpha \in \AA$ consider the rectangles with 
horizontal sides $\lambda_\alpha$ and vertical sides $h_\alpha$ 
defined by:
\be \nonumber
\begin{aligned}
&R^{t,r}_\alpha=\left( \sum_{\beta \in \BB_\pi(\alpha)}\lambda_\bbeta, 
\sum_{\beta \in \BB'_\pi(\alpha)} \lambda_\bbeta \right) 
\times [0,h_\alpha], 
\\
&R^{t,l}_\alpha=\left( -\sum_{\beta \in \BB'_\pi(\alpha)}\lambda_\bbeta, 
-\sum_{\beta \in \BB_\pi(\alpha)} \lambda_\bbeta \right) 
\times [0,h_\alpha], 
\\
&R^{b,r}_\alpha=\left( \sum_{\beta \in \BB_{\overline{\pi}}(\alpha)}\lambda_\bbeta, 
\sum_{\beta \in \BB'_{\overline{\pi}}(\alpha)} \lambda_\bbeta \right) 
\times [-h_\alpha,0]. 
\\
&R^{b,l}_\alpha=\left( -\sum_{\beta \in \BB'_{\overline{\pi}}(\alpha)}\lambda_\bbeta, 
-\sum_{\beta \in \BB_{\overline{\pi}}(\alpha)} \lambda_\bbeta \right) 
\times [-h_\alpha,0]. 
\end{aligned}
\ee

If $\alpha \notin \{\alpha(l), \alpha(r)\}$, also consider the 
vertical segments:
\be \nonumber
\begin{aligned}
&S^{t,r}_\alpha=\left\{ \sum_{\beta \in \BB'_\pi(\alpha)}\lambda_\bbeta \right\} 
\times \left[ 0, \sum_{\beta \in \BB'_\pi(\alpha)} \tau_\bbeta \right] 
\\
&S^{t,l}_\alpha=\left\{ -\sum_{\beta \in \BB'_\pi(\alpha)}\lambda_\bbeta \right\} 
\times \left[ 0, -\sum_{\beta \in \BB'_\pi(\alpha)} \tau_\bbeta \right] 
\\
&S^{b,r}_\alpha=\left\{ \sum_{\beta \in \BB'_{\overline{\pi}}(\alpha)}
\lambda_\bbeta \right\} 
\times \left[ \sum_{\beta \in \BB'_{\overline{\pi}}(\alpha)} \tau_\bbeta, 0 \right]
\\
&S^{b,l}_\alpha=\left\{ -\sum_{\beta \in \BB'_{\overline{\pi}}(\alpha)}
\lambda_\bbeta \right\} 
\times \left[ -\sum_{\beta \in \BB'_{\overline{\pi}}(\alpha)} \tau_\bbeta, 0 \right] 
\end{aligned}
\ee

If $L_\pi=\sum_{\pi(*)<\pi(\beta)\leq \pi(\alpha(r))}\tau_\bbeta>0$ we define
\be \nonumber
\begin{aligned}\displaystyle
&S^{t,r}_{\alpha(r)}=-S^{b,l}_{i(\alpha(r))}=
\left\{ \sum_{\pi(*)<\pi(\beta)\leq \pi(\alpha(r))}\lambda_\bbeta \right\} \times 
\left[ 0, L_\pi \right]
\\
&S^{t,l}_{\alpha(l)}=S^{b,r}_{i(\alpha(l))}=\emptyset
\end{aligned}
\ee

If $L_\pi<0$ we define
\be \nonumber
\begin{aligned}\displaystyle
&S^{t,l}_{\alpha(l)}=-S^{b,r}_{i(\alpha(l))}=
\left\{ \sum_{\pi(\alpha(l))\leq \pi(\beta)<\pi(*)}\lambda_\bbeta \right\} \times 
\left[ 0, L_\pi \right]
\\
&S^b_{i(\alpha(r))}=S^t_{\alpha(r)}=\emptyset
\end{aligned}
\ee

Otherwise, if $L_\pi=0$ we take 
\be \nonumber
\begin{aligned}
&S^{t,l}_{\alpha(l)}=
S^{b,l}_{i(\alpha(r))}=
\left\{ -\sum_{\pi(\alpha(l))\leq \pi(\beta)<\pi(*)}\lambda_\bbeta \right\} \times 
\{0\}
\\
&S^{t,r}_{\alpha(r)}=S^{b,r}_{i(\alpha(l))}=
\left\{ \sum_{\pi(*)<\pi(\beta)\leq \pi(\alpha(r))}\lambda_\bbeta \right\} \times 
\{0\}
\end{aligned}
\ee

Notice that, for each $\alpha \in \AA$, the labels $l$ and $r$ in 
$X^{\epsilon,l}_\alpha$ and $X^{\epsilon,r}_\alpha$, where $\epsilon \in \{t,b\}$
and $X \in \{R, S\}$, are just to make clear when $\pi(\alpha)<\pi(*)$ or 
$\pi(*)<\pi(\alpha)$. When it does not lead to confusion, we will 
omit $l$ and $r$. 

\begin{example}

The Figure \ref{fig2} represents a zippered rectangle associated to 
$$\pi=D \quad i(B) \quad i(D) \quad C \quad i(C) \quad * \quad A \quad i(A) \quad B$$.

\begin{figure}[h]
   \begin{center}
\psfrag{RAt}{{\scriptsize $R_A^t$}}  \psfrag{RAb}{{\scriptsize $R_A^b$}}  
\psfrag{RinvAt}{{\scriptsize $R_{i(A)}^t$}}  \psfrag{RinvAb}{{\scriptsize $R_{i(A)}^b$}}  
\psfrag{SAt}{{\scriptsize $S_A^t$}}  \psfrag{SAb}{{\scriptsize $S_A^b$}}  
\psfrag{SinvAt}{{\scriptsize $S_{i(A)}^t$}}  \psfrag{SinvAb}{{\scriptsize $S_{i(A)}^b$}}

\psfrag{RBt}{{\scriptsize $R_B^t$}}  \psfrag{RBb}{{\scriptsize $R_B^b$}}  
\psfrag{RinvBt}{{\scriptsize $R_{i(B)}^t$}}  \psfrag{RinvBb}{{\scriptsize $R_{i(B)}^b$}}  
\psfrag{SBt}{{\scriptsize $S_B^t$}}  \psfrag{SBb}{{\scriptsize $S_B^b$}}  
\psfrag{SinvBt}{{\scriptsize $S_{i(B)}^t$}}  \psfrag{SinvBb}{{\scriptsize $S_{i(B)}^b$}}

\psfrag{RCt}{{\scriptsize $R_C^t$}}  \psfrag{RCb}{{\scriptsize $R_C^b$}}  
\psfrag{RinvCt}{{\scriptsize $R_{i(C)}^t$}}  \psfrag{RinvCb}{{\scriptsize $R_{i(C)}^b$}}  
\psfrag{SCt}{{\scriptsize $S_C^t$}}  \psfrag{SCb}{{\scriptsize $S_C^b$}}  
\psfrag{SinvCt}{{\scriptsize $S_{i(C)}^t$}}  \psfrag{SinvCb}{{\scriptsize $S_{i(C)}^b$}}

\psfrag{RDt}{{\scriptsize $R_D^t$}}  \psfrag{RDb}{{\scriptsize $R_D^b$}}  
\psfrag{RinvDt}{{\scriptsize $R_{i(D)}^t$}}  \psfrag{RinvDb}{{\scriptsize $R_{i(D)}^b$}}  
\psfrag{SDt}{{\scriptsize $S_D^t$}}  \psfrag{SDb}{{\scriptsize $S_D^b$}}  
\psfrag{SinvDt}{{\scriptsize $S_{i(D)}^t$}}  \psfrag{SinvDb}{{\scriptsize $S_{i(D)}^b$}}

   \includegraphics{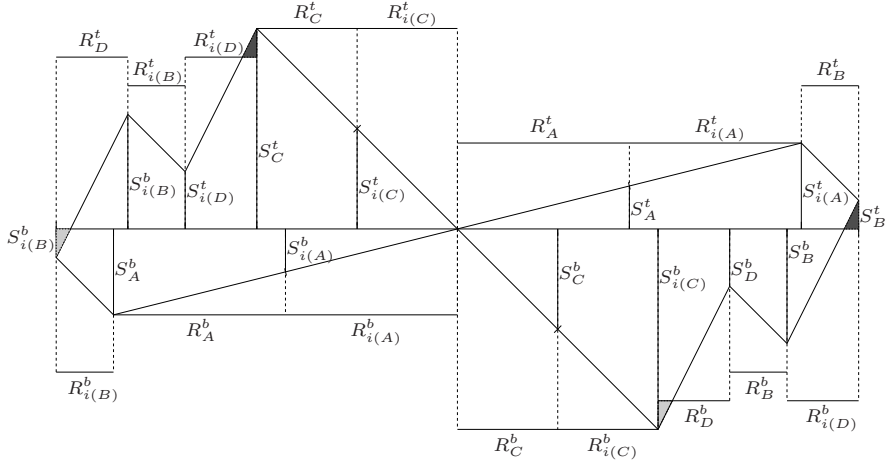}
       \caption{Zippered rectangle} \label{fig2}
   \end{center}
\end{figure}
\end{example}

Define the set 
\be \nonumber
R_{(\lambda, \pi, \tau)}=\bigcup_{\alpha \in \AA} 
\bigcup_{\epsilon \in \{l,r\}}
\left( R^\epsilon_\alpha \cup S_\alpha^\epsilon \right)
\ee

We will identify, by translation, the rectangle $R^t_\alpha$ with 
$R^b_\alpha$ for all $\alpha \in \AA$.

If $L_\pi>0$ we identify $S^t_{\alpha(r)}$ with the vertical segment 
$S_1$ of length $L_\pi$ 
at the bottom of the right side of the rectangle $R_{\alpha(r)}^b$
if $\alpha(r)$ is the winner of $\pi$ or 
at the top of the right side of the rectangle $R_{i(\alpha(l))}^t$
if $\alpha(r)$ is the loser of $\pi$.
Symmetrically, we identify $S^b_{i(\alpha(r))}$ with $-S_1$.

If $L_\pi<0$ we identify $S^b_{i(\alpha(l))}$ with the vertical segment 
$S_2$ of length $-L_\pi$ 
in the bottom of the right side of the rectangle $R_{\alpha(r)}^b$
if $\alpha(r)$ is the winner of $\pi$ or
in the top of the right side of the rectangle $R_{i(\alpha(l))}^t$
if $\alpha(r)$ is the loser of $\pi$.
Symmetrically, we identify $S^t_{\alpha(l)}$ with $-S_2$.

Let $\tilde S^*(\lambda, \pi, \tau)$ be the topological space obtained from 
$R_{(\lambda, \pi, \tau)}$ by these identifications. Thus, 
$\tilde S^*(\lambda, \pi, \tau)$ inherits from $\R^2=\C$ the structure of a 
Riemann surface and also a holomorphic 1-form $\omega$ (given by $dz$).

For each $\alpha \in \AA$ recall 
$\zeta_\alpha = \lambda_\alpha + i \tau_\alpha$.
We call \emph{vertices} the extreme points in the top of segments $S_\alpha^t$
and the extremes in the bottom of segments $S_\alpha^b$, for all 
$\alpha \in \AA$. So, the vertices are points with following coordinates:
\be \nonumber
\xi_\alpha^t =
\left\{
\begin{array}{lll}
\displaystyle 
 \sum_{\pi(\alpha) \leq \pi(\beta) <\pi(*)} -\zeta_\beta 
& \quad \text{if} \quad & \pi(\alpha)<\pi(*)
\\[5pt]
\displaystyle
 \sum_{\pi(*) < \pi(\beta) \leq \pi(\alpha)} \zeta_\beta 
& \quad \text{if} \quad & \pi(*)<\pi(\alpha)
\end{array}
\right.
\ee

\be \nonumber
\xi_\alpha^b =
\left\{
\begin{array}{lll}
\displaystyle
 \sum_{\overline{\pi}(\alpha) \leq \overline{\pi}(\beta) 
<\overline{\pi}(*)} -\zeta_\beta 
& \quad \text{if} \quad & \overline{\pi}(\alpha)<\overline{\pi}(*)
\\[5pt]
\displaystyle
 \sum_{\overline{\pi}(*) < \overline{\pi}(\beta) \leq 
\overline{\pi}(\alpha)} \zeta_\beta 
& \quad \text{if} \quad & \overline{\pi}(*)<\overline{\pi}(\alpha)
\end{array}
\right.
\ee

Now we will define a relation to identify vertices 
between them.
Define the set of all pairs $(\alpha, Y)$ with 
$\alpha \in \AA \cup \{*\}$ and $Y \in \{L, R\}$. 
Consider the following identification: 

\be \nonumber
\begin{aligned}
(\pi(\pi^{-1}(*)+1),L) \sim (*,R)\sim 
(\overline{\pi}(\overline{\pi}^{-1}(*)+1),L)
\\
(\pi(\pi^{-1}(*)-1),R) \sim (*,L)\sim 
(\overline{\pi}(\overline{\pi}^{-1}(*)-1),R)
\end{aligned}
\ee
\be \nonumber
\begin{aligned}
(\alpha(r),R) \sim (i(\alpha(l)),R)
\\
(\alpha(l),L) \sim (i(\alpha(r)),L)
\end{aligned}
\ee

We say that these pairs are \emph{irregular} and all other 
pairs we call \emph{regular}.
We also identify: 
\be \nonumber
\begin{aligned}
(\alpha(r),R) \sim (\beta, L) \quad \text{if} \quad \pi(\alpha)+1=\pi(\beta)
\\
(\alpha(r),R) \sim (\beta, L) \quad \text{if} \quad 
\overline{\pi}(\alpha)+1=\overline{\pi}(\beta)
\end{aligned}
\ee

We can extend $\sim$ to an equivalence relation in the set of pairs 
$(\alpha, Y)$. This equivalence relation describes how half-planes 
are identified when one winds around an end of 
$\tilde S^*(\lambda, \pi, \tau)$. Let $\tilde \Sigma$ be the set of 
equivalence classes relative to the relation $\sim$. Thus to each 
$c \in \tilde \Sigma$ we have one, and only one, end $v_c$ of 
$\tilde S=\tilde S^*(\lambda, \pi, \tau)$. When it does not lead 
to confusion we will use $\tilde S$ to mean $\tilde S(\lambda, \pi, \tau)$.
From the local structure around $v_c$, the compactification 
\be \nonumber
\tilde S(\lambda, \pi, \tau)=\tilde S^*(\lambda, \pi, \tau)
\cup(\cup_{\tilde \Sigma}\{v_c\})
\ee
is a compact Riemann surface with marked points $\{v_c\}$. The 1-form 
$\omega$ extends to a holomorphic 1-form on $\tilde S(\lambda, \pi, \tau)$
such that at the points $v_c$ we have marked zeroes of angle 
$2k_c\pi$ where $2k_c$ is the cardinality of the equivalence class 
of $c$. 

Given $(x,0)$ on the bottom side of the rectangle $R_\alpha^t$, 
we can transport this point vertically and when we reach the top 
side, which is the point $(x,h_\alpha)$ we identify it with the 
point $(x+\omega_\alpha, 0)$, where 
$\omega=\Omega(\pi)\cdot\lambda$, in the top side 
of $R_\alpha^b$. So, we have the vertical flow well-defined 
almost everywhere (except in the points which reach 
singularities in finite time). It is clear that the return 
time of points in the rectangle $R_\alpha^t$ is equal to $h_\balpha$ 
and the area of the surface $\tilde S(\lambda, \pi, \tau)$ is 
$A(\lambda, \pi, \tau)=-2\langle \lambda,\Omega(\pi) \cdot \tau \rangle$. 

When we constructed the surface $\tilde S$, we have 
an implicit relation between the horizontal coor\-dinates 
$\lambda_\alpha$ and $\lambda_{i(\alpha)}$ and the vertical 
coordinates $\tau_\alpha$ and $\tau_{i(\alpha)}$. Indeed, we have 
an involution $I: \tilde S \to \tilde S$, 
with a fixed point at the origin,
defined as follows. Given any point $x \in \tilde S$ 
there exists $\alpha \in \AA$ such that $x \in R_\alpha^t$ or 
$x \in S_\alpha^t \cup S_\alpha^b$. Thus $-x \in R_{i(\alpha)}^b$ or 
$-x \in S_{i(\alpha)}^b \cup S_{i(\alpha)}^t$, respectively. 
So $I(x)$ is identified with $-x$. 

Let $S(\lambda, \pi, \tau)$ be the surface 
$\tilde S(\lambda, \pi, \tau)$ quotiented by the involution $I$ and let 
$\Sigma$ to be the set $\tilde \Sigma$ quotiented by the involution $I$.
We can see that this identification by involution implies that, 
for each $\alpha \in \AA$, the rectangle $R^t_\alpha$ is identified with 
the rectangle $R^t_{i(\alpha)}$ by a translation composed with a 
rotation of $180$ degrees. So, the top side (resp. the bottom side) 
of the rectangle $R^t_\alpha$ is identified with the bottom side 
(resp. the top side) of the rectangle $R^t_{i(\alpha)}$.

\subsection{Coordinates}\label{coordinates}

Let $(\tilde S, \tilde \Sigma, I, \tilde x_0, \tilde X)$ 
be a marked translation surface with an involution $I$.
The marked separatrix $\tilde X$ starts at $\tilde x_0$ 
and it goes to east.
A segment $\sigma$ adjacent to $\tilde x_0$ contained 
in $\tilde X$ is called \emph{admissible} if the vertical 
geodesic $Y$ passing through the right endpoint 
$\tilde z$ of $\sigma$ meets a singularity in the positive 
or in the negative direction before returning to 
$\sigma \cup I(\sigma)$. Symmetrically, if $\sigma$ is an admissible 
segment then $I(\sigma)$ starting at $I(\tilde x_0)$ 
going to west and ending at $I(\tilde z)$ (which has a vertical 
geodesic meeting a singularity in the negative or 
in the positive direction before return to $\sigma \cup I(\sigma)$),
also is an admissible segment if we consider the marked 
separatrix $I(\tilde X)$ instead consider $X$. 

We call a separatrix \emph{incoming} if its natural 
orientation points towards the associated singularity 
and we call it \emph{outgoing} otherwise.
Let $\sigma^+$ be the set of points of the first intersection of 
incoming vertical separatrices with $\sigma \cup I(\sigma)$ and 
$\sigma^-$ be the set of points of the first intersection of 
outgoing vertical separatrices with $\sigma \cup I(\sigma)$. 
Notice that $\tilde x_0$ and $I(\tilde x_0)$ are in both sets 
$\sigma^+$ and $\sigma^-$. If $Y$ is incoming (resp. outgoing) 
we extend it to the past (resp. future) until intersect 
$\sigma \cup I(\sigma)$ again and the second intersection 
point is an element of $\sigma^+$ (resp. $\sigma^-$). Thus 
$\tilde z$ is an element of both $\sigma^+$ and $\sigma^-$. By
the same argument applied to $I(Y)$ we conclude that 
$I(\tilde z)$ also is an element of both $\sigma^+$ and $\sigma^-$. 

Notice that $p \in \sigma^-$ if and only if $I(p) \in \sigma^+$, for 
all $p \in \sigma^-$. Thus we will consider just the set $\sigma^+$ which 
determines the set $\sigma^-$ by involution.

Let $|\lambda|$ be the length of $\sigma$ and of $I(\sigma)$. 
Let $\phi_r: \left[0,|\lambda|\right] \to \tilde S$ and 
$\phi_l: \left[-|\lambda|, 0\right] \to \tilde S$ be the parametrizations 
of $\sigma$ and $I(\sigma)$, respectively, by arc-length with 
$\phi_r(0)=\tilde x_0$ and $\phi_l(0)=I(\tilde x_0)$.

We can write:
\be \nonumber
\sigma^+= \{ I(\tilde z)=p_{-l}^+ > \ldots > p_{-1}^+ > p_{0^-}^+ =I(\tilde x_0) \}
\cup \{ \tilde x_0=p_{0^+}^+ < p_1^+ < \ldots < p_r^+ =\tilde z \}. 
\ee
where $<$ and $>$ refer to the natural orientation on $\sigma$ 
and $I(\sigma)$, respectively.

Therefore, we have numbers 
\be \nonumber
-|\lambda|=a_{-l}^+ < \ldots < a_{-1}^+ < a_{0^-}^+= 
0 = a_{0^+}^+ < a_1^+ < \ldots < a_r^+ = |\lambda|
\ee
such that $\phi_r(p_j^+)=a_j^+$, for all $j \in \{-l \ldots , -1, 1 \ldots r\}$, 
$\phi_r(p_{0^-}^+)=a_{0^-}^+=0$ and $\phi_r(p_{0^+}^+)=a_{0^+}^+=0$.

Let $a_0^+=a_{0^+}^+=a_{0^-}^+$.
Define $\lambda_j=|I_j|=a_{j+1}^+-a^+_j$ for $-l\leq j \leq r-1$ 
Let $\tau_j$ be the length of the vertical segment from the 
horizontal section to the singularity corresponding to 
the point $p_j^+$. Notice that $\tau_0=0$.

It is possible to verify that the first return map to the cross-section 
$\sigma \cup I(\sigma)$ is well-defined except at the points $p_j^+$. 
Moreover the first return time is constant on each open interval 
$(a_{j}^+, a_{j+1}^+)$.
So, we can consider the interval exchange transformation with involution $\pi$ 
associated to the cross-section $\sigma \cup I(\sigma)$ where the points 
defined before are the points of discontinuity. 

Let $h_j$ be the first return time of the points in the interval 
$(a_{j-1}, a_j)$. We can define the zippered rectangle which represents 
$(\tilde S, \tilde \Sigma, I, \tilde x_0, \tilde X)$ by:
\be \nonumber
ZR(\lambda, \pi, \tau, h)=\cup_{j} (a_{j-1}, a_j) \times [0,h_j].
\ee

\begin{lemma}
If two admissible segments $\sigma$ and $\tilde \sigma$ with the 
same left extreme point $\tilde x_0$ are such that 
$\tilde \sigma \subset \sigma$, then the corresponding 
zippered rectangles $(\lambda, \pi, \tau, h)$ and 
$(\tilde \lambda, \tilde \pi, \tilde \tau,\tilde h)$ satisfy:
\be \nonumber
\exists \ n \in \N \text{ such that } 
(\tilde \lambda, \tilde \pi, \tilde \tau)=\widehat Q^n(\lambda, \pi, \tau) 
\ee
\end{lemma}

\begin{proof}
Let $\sigma$ and $\tilde \sigma$ be admissible segments and let 
the respective zippered rectangles representations
$ZR(\lambda, \pi, \tau, h)$ and 
$ZR(\tilde \lambda, \tilde \pi, \tilde \tau,\tilde h)$.

Consider a sequence of maximal admissible segments $\sigma^{i+1}$ 
strictly contained in $\sigma^i$ such that $\sigma^1=\sigma$. 
Let $z_i$ be the right endpoint of $\sigma^i$.
The right endpoint $z_2$ of $\sigma^2$ corresponds to a discontinuity 
point of the first return map of the vertical flow to the section $\sigma^1$ 
and there is no other discontinuity point between $z_2$ and $z_1$. 
By maximality, we conclude that, up to relabeling, $\widehat Q(\lambda,\pi,\tau)$ 
is the representation of such first return map. We follow this process until 
obtain $\sigma^n=\tilde \sigma$ for some $n \in \N$. For such $n$ we have
$\widehat Q^n(\lambda, \pi, \tau)=(\tilde \lambda, \tilde \pi, \tilde \tau)$.
\end{proof}

\begin{cor}
Let $(\tilde S, \tilde \Sigma, I, \tilde x_0, \tilde X)$ be a marked 
translation surface with involution and $ZR(\lambda, \pi, \tau, h)$ and 
$ZR(\tilde \lambda, \tilde \pi, \tilde \tau,\tilde h)$ be two zippered 
rectangle representations of the surface. Then there exists $n \in \Z$ such that 
$(\tilde \lambda, \tilde \pi, \tilde \tau)=\widehat Q^n(\lambda, \pi, \tau)$. 
\end{cor}

\begin{proof}
Let $\sigma$ and $\tilde \sigma$ be admissible segments of 
$ZR(\lambda, \pi, \tau, h)$ and $ZR(\tilde \lambda, \tilde \pi, \tilde \tau,\tilde h)$, 
respectively.
By definition, the initial points of $\sigma$ and $\tilde \sigma$ 
are the same $\tilde x_0$. Suppose, without loss of generality, 
that $\tilde \sigma \subset \sigma$.

By the previous lemma there exists $n \in \N$ such that 
$(\tilde \lambda, \tilde \pi, \tilde \tau)=\widehat Q^n(\lambda, \pi, \tau)$.
Thus, the result follows.
\end{proof}

Given a marked translation surface with involution 
$(\tilde S, \tilde \Sigma, I, \tilde x, \tilde X)$
with zippered rectangle represen\-tation $ZR(\lambda,\pi,\tau, h)$, 
we can cut and paste it appropriately until we obtain a surface 
$(\tilde S', \tilde \Sigma', I', \tilde x', \tilde X')$ 
which representation in zippered rectangles is an iterated 
by Rauzy induction with involution of the first marked translation 
surface with involution. Since 
these operations preserve the relation between parallel sides, 
then $\tilde S$ and $\tilde S'$ are isomorphic and the marked 
separatrix is mapped to one another.
Moreover if we have a marked translation surface with involution 
$(\tilde S_1, \tilde \Sigma_1, I_1, \tilde x_1, \tilde X_1)$
which is near from $(\tilde S, \tilde \Sigma, I, \tilde x, \tilde X)$, 
by the continuity of the marked separatrix and of the 
singularities, we will obtain a zippered rectangle construction
$ZR(\lambda_1,\pi_1,\tau_1, h_1)$ near, up to relabel, from $ZR(\lambda,\pi,\tau, h)$.
So, the zippered rectangle construction, gives a system of 
local coordinates in each stratum of the moduli space.

Using the zippered rectangles construction, 
we obtain a finite covering $ZR$, of a stratum  
of the moduli space of marked translation surfaces 
with involution, $\MHI(\tilde g, \tilde \kappa, \eta, \tilde j)$.
Under the condition $\langle \lambda, h \rangle=1$, we get 
the space of zippered rectangles of area one covering the 
space $\MHI^{(1)}(\tilde g, \tilde \kappa, \eta, \tilde j)$.
We have a bijection between Rauzy classes with involution 
and a connected component of a stratum of the moduli space 
of translation surfaces with involution (see \cite{BL}).
Thus we have a well-defined map $\proj: \Om_\RR \to \CC$, 
where $\CC=\CC(\RR)$ is a connected component of 
$\MHI(\tilde g, \tilde \kappa, \eta, \tilde j)$ and 
$\proj \circ \widehat Q=\proj$. The fibers of this map 
are almost everywhere finite (with constant cardinality).
The projection of the standard Lebesgue measure on
$\Om_\RR$ is (up to scaling) the standard volume form on $\CC$.

The subset $\Om_\RR^{(1)}=\proj^{-1}(\CC^{(1)})$ of surfaces with area one is 
invariant by the Veech flow. So, the restriction 
$\TV_t(x):\Om_\RR^{(1)} \to \Om_\RR^{(1)}$ leaves invariant 
the volume form that projects, up to scaling, to the invariant 
volume form on $\CC^{(1)}$. It was proved by Veech that this volume form 
is finite using the lift measure on $\Om^{(1)}_\RR$.

\subsubsection{Homology and cohomology}

For each $\alpha \in \AA$ consider the curve $c_\alpha$ which is a 
path in $R_{(\lambda, \pi, \tau)}$ joining $\xi_\alpha^t-\zeta_\alpha$ 
to $\xi_\alpha^t$ if $\pi(\alpha)>\pi(*)$ or 
joining $\xi_\alpha^t$ to $\xi_\alpha^t+\zeta_\alpha$ 
if $\pi(\alpha)<\pi(*)$. Note that $I(c_\alpha)=-c_{i(\alpha)}$.

Consider the relative homology group $H_1(\tilde S, \tilde \Sigma; \Z)$ of 
the surface $\tilde S$ relative to the finite set of singularities 
$\tilde \Sigma$. We have a decomposition of the relative homology 
group into an invariant subgroup $H_1^+(\tilde S,\tilde \Sigma; \Z)$ 
and an anti-invariant subgroup $H_1^-(\tilde S,\tilde \Sigma; \Z)$, 
with respect to the involution $I$.
Following Masur and Zorich (\cite{MZ}), 
we can choose a basis in $H_1^-(\tilde S,\tilde \Sigma; \Z)$ 
which has dimension $d-1$, where $d$ is the number of classes of 
$\BA$\footnote{The complex dimension of the moduli space of 
half-translation surfaces of genus $g$ with $\sigma$ singularities 
is $2g + \sigma -2$ and $d=2g+ \sigma -1$.}. 
The elements of the basis will be lifts of a collection of saddle connections 
on $S$, where $S$ is the surface $\tilde S$ quotiented by involution.


Analogously, the first (de Rham) cohomology group 
$H^1(\tilde S,\tilde \Sigma; \C)$, is decomposed into an invariant subspace 
$H^1_+(\tilde S,\tilde \Sigma; \C)$ and an anti-invariant subspace 
$H^1_-(\tilde S,\tilde \Sigma; \C)$, under the induced involution 
$I^*: H^1(\tilde S,\tilde \Sigma; \C) \to H^1(\tilde S,\tilde \Sigma; \C)$ 
Notice that $[\omega]$ is anti-invariant under the induced involution, 
so $[\omega] \in H^1_-(\tilde S, \tilde \Sigma, \C)$ 
and we also have:
\begin{equation*}
 \label{cycle}
\int_{c_\alpha} \omega=\zeta_\alpha
\end{equation*}


In section \ref{MHI} we have observed that 
$H^1_-(\tilde S,\tilde \Sigma; \C)$ yields local coordinates 
of an element of a stratum of the moduli space of translation 
surfaces with involution. So if we consider the set of 
$\zeta_\balpha=\lambda_\balpha + i \tau_\balpha$ such that $\lambda, \tau \in \SS_\pi$
we obtain coordinates which describe $\tilde S(\lambda, \pi, \tau)$.
And as we have seen in section \ref{coordinates}, for any other pair 
$(\tilde S', \omega')$ in a neighborhood of $(\tilde S, \omega)$ 
we can find coordinates $(\lambda', \pi', \tau')$ of $(\tilde S', \omega')$, 
and so we can define the vectors $\zeta'_\balpha$ as in section \ref{ZR}.
For more details in the construction of coordinates see \cite{Ve86}.  

\subsection{Teichm\"uller flow is exponential mixing}

In section \ref{veechteichmullerflow} we have seen the 
relation between the Teichm\"uller flow and the Veech flow, 
which is naturally identified with the first return map of the 
renormalization operator to the section $\widehat \Up^{(1)}_\RR$.

We will identify $\widehat \Up^{(1)}_\RR \times \R$ with 
a connected component $\CC^{(1)}$ by the map
$P:\widehat \Up^{(1)}_\RR \times \R \to \CC^{(1)}$ defined 
by $P(z,s)=\TF_s(\proj(z))$, where $z=(\lambda, \pi, \tau)$ 
and $\proj:\widehat \Delta_\RR \to \CC$ is the natural projection.

\begin{lemma}
\label{lem:regularisation}
Let $f:\CC^{(1)} \to \R$ be a $C^1$ compactly supported function and let
$\delta>0$ be as in (\ref {tilde f}).
There exists $\epsilon_0>0$ and $C>0$ such that for
every $t>0$, there exists a $C^1$ function $f^{(t)}:\wde \to \R$,
such that $\|f \circ P-f^{(t)}\|_{L^2(\nu)} \leq
C e^{-\epsilon_0 t}$ and $\|f^{(t)}\|_{C^1(\wde)} \leq
C e^{\delta t}$.
\end{lemma}

\begin{proof}
Let $\delta_0>0$ be small and let $Y_t \subset \wde$ be the union of
connected components of $\wde$ which contain points $(\lambda,\pi,\tau,s)$
with $s>\delta_0 t$.
Let $f^{(t)}=0$ in $Y_t$ and $f^{(t)}=f \circ P$ in the complement.  The
estimate $\|f \circ P-f^{(t)}\|_{L^2(\nu)} \leq
C e^{-\epsilon_0 t}$ is then clear since $\|f \circ P-f^{(t)}\|_{C^0} \leq
\|f\|_{C^0}$, while the support of $f \circ P-f^{(t)}$ has exponentially
small $\nu$ measure (since the roof function has exponential tails).

For the other estimate, it is enough to show that if $(z,s) \in \wde$ and
$P(z,s)$ belongs to any fixed compact set $K \subset \CC^{(1)}$
then $P$ is locally Lipschitz near
$(z,s)$, with constant bounded by $C(K) e^{C(K) s}$.  Here we fix some
arbitrary Finsler metric in $\CC^{(1)}$ (the precise choice is irrelevant
since $K$ is compact). This result is obvious if we impose some bound on
$s$, say $0 \leq s \leq 1$, since $P$ is smooth.
If $s_0>0$ is such that
$s_0<s<s_0+1$, notice that for $(z',s')$ in a neighborhood of $(z,s)$,
$P(z',s')$ is obtained from $P(z,s'-s_0)$
by applying the Teichm\"uller flow for time $s_0$. Thus, it is enough to
show that if $x$ and $\TF_{s_0}(x)$ belong to some fixed compact set of
$\CC^{(1)}$ then $\TF_{s_0}$ is locally $C e^{C s_0}$ Lipschitz in a
neighborhood of $x$. This is a well known estimate, for instance, we can
define a Finsler metric on $\CC^{(1)}$ such that $\TF_{s_0}$ is globally
Lipschitz with Lipschitz constant $e^{2 s_0}$ (see \cite {AGY} \S 2.2.2 for
the construction of a metric in the whole strata of squares, the Finsler
metric we need here being just the restriction to the substrata).
\end{proof}

\begin{lemma}
If $f:\CC^{(1)} \to \R$ is $C^1$ and compactly supported with $\int f
d\nu_{\CC^{(1)}}=0$ then there
exists $C>0$, $\epsilon>0$ such that for $t>0$,
\be \label {corr}
\int f \cdot \left(f \circ \TF_t \right)
d\nu_{\CC^{(1)}} \leq C e^{-\epsilon t}.
\ee
\end{lemma}

\begin{proof}
We can estimate (\ref {corr}) with exponentially small error
by comparison with the correlations
$\int f^{(t)} \cdot f^{(t)} \circ \widehat T_t d\nu-(\int f^{(t)} d\nu)^2$,
where $f^{(t)}$ is provided by the previous
lemma.  Those decays exponentially by (\ref {tilde f}).
\end{proof}

Finally, we are in a position to prove the main theorem:

\begin{proof}[Proof of Theorem \ref {t.main}]
Let $H$ be the Hilbert space of $\SO(2,\R)$ invariant $L^2(\nu_{\CC^{(1)}})$
functions with zero mean.
As shown in Appendix B of \cite {AGY}, exponential decay of correlations
for the Ratner class follows from the existence of a dense set of $f$
in $H$ such that (\ref {corr}) decays exponentially fast.  Since compactly
supported smooth functions are dense in $H$, the result follows.
\end{proof}

\end{document}